\documentclass[reqno,11pt]{amsart}

\usepackage[top=2.0cm,bottom=2.0cm,left=3cm,right=3cm]{geometry}
\usepackage{amsthm,amsmath,amssymb,dsfont}
\usepackage{mathrsfs,amsfonts,functan,extarrows,mathtools}
\usepackage[colorlinks]{hyperref}

\usepackage{marginnote}
\usepackage{xcolor}
\usepackage{stmaryrd}
\usepackage{esint}
\usepackage{graphicx}
\usepackage{bbm}

\usepackage{parcolumns}
\usepackage{caption}
\usepackage{subcaption}
\usepackage{enumitem}
\usepackage{makecell}
\usepackage{tikz}
\usetikzlibrary{calc,decorations.pathreplacing}

\newtheorem{theorem}{Theorem}[section]
\newtheorem{proposition}{Proposition}[section]
\newtheorem{definition}{Definition}[section]
\newtheorem{corollary}{Corollary}[section]
\newtheorem{lemma}{Lemma}[section]

\newtheorem{assumption}{Assumption}[section]

\newtheoremstyle{remarkstyle}{}{}{\normalfont}{0pt}{\itshape}{.}{.5em}{\itshape\thmname{#1}\thmnumber{ #2}}   
\theoremstyle{remarkstyle}
\newtheorem{remark}{Remark}[section]

\numberwithin{equation}{section}
\allowdisplaybreaks

\def\d{\mathrm{d}}

\def\div{\mathrm{div}}

\def\exp{\mathrm{exp}}

\newcounter{wronumber}

\begin{document}
	\title[]
    {Quantitative propagation of chaos for particle systems with bounded kernels and multiplicative noise}

	\author[Ning Jiang]{Ning Jiang}
	\address[Ning Jiang]{\newline School of Mathematics and Statistics, Wuhan University, Wuhan, 430072, P. R. China}
	\email{njiang@whu.edu.cn}

	\author[Rongli Mo]{Rongli Mo}
	\address[Rongli Mo]
	{\newline School of Mathematics and Statistics, Wuhan University, Wuhan, 430072, P. R. China}
	\email{ronglimo@whu.edu.cn}
	\thanks{${}^*$ Corresponding author \quad \today}
		
\begin{abstract}
	We prove the quantitative propagation of chaos for stochastic particle systems with interaction in both the drift and the diffusion coefficients, provided the drift kernel is bounded and free of Lipschitz or smoothness assumptions. Our proof is based on the relative entropy framework of Jabin and Wang \cite{JW2018}, and applies and extends their work on the exponential laws of large numbers. We extend one of their exponential laws of large numbers from the drift to the diffusion kernel to handle the error term arising from multiplicative noise in the entropy evolution equation. Proving this extension relies on a dynamic combinatorial analysis.\\
	\noindent\textsc{Keywords.}  Interacting particle systems; Multiplicative noise; McKean-Vlasov equation; Propagation of chaos; Relative entropy\\

	\noindent\textsc{MSC2020.}  60K35, 82C22, 60H10
\end{abstract}

\maketitle	
	
	
\phantomsection
	

	
\section{Introduction}
\subsection{Motivation} We consider the large systems of $N$ indistingguishable interacting particle  on the torus $\mathbb{T}^d$ governed by the coupled stochastic differential equations (SDEs)
\begin{equation}\label{Particle-sys}
	\d X^i_t = \dfrac{1}{N} \sum \limits_{ k= 1}^N K ( X^i_t - X^k_t ) \d t + \dfrac{ \sqrt{2} }{N} \sum \limits_{ k= 1 }^N \sigma( X^i_t - X^k_t ) \d B^i_t \,, \ i = 1\,, \cdots\,, N \,,
\end{equation}
where $X^i_t\in \mathbb{T}^d$ denotes the position of the $i\mbox{-}$th particle at time $t\geq 0$, $\{B^i_t\}$ are independent $d\mbox{-}$dimensional Brownian motions, the vector field $K: \mathbb{T}^d\to\mathbb{R}^d$ is the interaction drift kernel, and the diagonal matrix field $\sigma: \mathbb{T}^d\to\mathbb{R}^{d \times d}$ is the interaction diffusion kernel. The factor $1/N$ in \eqref{Particle-sys} corresponds to the so-called mean-field scaling, i.e. the total effect exerted on any particle by the others remains of order unity \cite{Spo1991}. The stochastic terms in \eqref{Particle-sys} should be understood in the It\^o sense. In this paper, we take $\mathbb{T}^d=[0,1]^d$ for simplicity.

The multiplicative noise setup in System \eqref{Particle-sys} can be practically achieved by adding noise to the interaction strength function when modeling emergent behaviors of many-body systems via particle systems. Examples include perturbing the coupling strength in the Kuramoto synchronization model \cite{HKMZ2020} or the communication weight function in the Cucker–Smale flocking model \cite{AH2010, HJNXZ2017}. Besides modeling emergent behaviors, interacting particle systems also find applications in control engineering \cite{LPLSFD2007}, mean-field games \cite{CDLL2019} and machine learning \cite{BFFT2012}.

As $N \to \infty$, the {\em mean field} limit of the system \eqref{Particle-sys} is formally described by the McKean-Vlasov equation
\begin{equation}\label{Limit-Eq}
	\partial_t \bar{\rho}+\nabla \cdot ( (K*\bar{\rho}) \bar{\rho}  )
	=\sum^d_{\alpha=1} \partial^2_{x^\alpha} \Big( \bar{\rho} \big( \sigma_{\alpha\alpha}*\bar{\rho} (x) \big)^2  \Big)\,,
\end{equation}
where $\bar{\rho}=\bar{\rho}(t,x)$ is the one-particle density function; see for instance \cite{Szn1991} for the formal derivation. Usually, the rigorous derivation of the limiting equation \eqref{Limit-Eq} from the particle system \eqref{Particle-sys} can be approached by tackling one of two equivalent problems:
\begin{itemize}
	\item Convergence of the empirical measure: to show that the empirical measure of \eqref{Particle-sys}  $$\mu_N(t):=\frac{1}{N}\sum^N_{i=1}\delta_{X^i_t}$$ converges weakly to $\bar{\rho}$ as $N\to\infty$, provided it holds at $t=0$. 
	\item Propagation of chaos: to show that for all fixed $k$ and all $t>0$, the $k\mbox{-}$particle marginal of the $N$-particle system \eqref{Particle-sys}
	$$ \rho_N^k(t,x_1,\cdots,x_k):=\int_{\mathbb{T}^{d(N-k)}} \rho_N(t,x_1,\cdots,x_N)\d x_{k+1}\cdots \d x_N$$
	converges weakly to $\bar{\rho}^{\otimes k}$ as $N\to\infty$, provided it holds at $t=0$.
\end{itemize}
Here $\rho_N$ is the joint distribution density of the $N\mbox{-}$particle system \eqref{Particle-sys}, formally governed by the Liouville euqaiton (or forward Kolmogorov equation) 
\begin{equation}\label{Liouville-Eq}
	\begin{aligned}
		\partial_t \rho_N &+ \sum^N_{i=1} \nabla_{x_i} \cdot \Big( \rho_N\Big( \dfrac{1}{N} \sum \limits_{k = 1}^N K(x_i - x_k) \Big) \Big)
		=\sum^N_{i=1} \sum^d_{\alpha=1} \partial^2_{x^\alpha_i} \Big( \rho_N  \dfrac{1}{N^2} \Big( \sum\limits_{k = 1}^N \sigma_{\alpha\alpha}(x_i - x_k) \Big)^2 \Big)\,;
	\end{aligned}
\end{equation}
see for instance \cite{Pav2014} for the formal derivation. This paper adopts the second perspective and aims to establish quantitative propagation of chaos for the system \eqref{Particle-sys}.

 Most literature on propagation of chaos addresses the case of the constant diagonal diffusion kernel $\sigma=cI_d$ with $c>0$, including studies where the interaction drift kernel $K$ has some good regularity \cite{Szn1991, BCC2011, Mal2001, Mal2003, DEGZ2020, GPG2025} and those where $K$ has weak regularity or is singular \cite{JW2016, JW2018, BJW2023, CCS2019, RS2023, GLM2025, RS2024}. In particular, Jabin and Wang \cite{JW2018}  developed a new approach based on the relative entropy for systems with the weakly regular drift kernel $K$, thereby providing an effective analytical framework for proving the quantitative propagation of chaos in this setting.
 
In this paper, we assume the drift kernel $K$ to be  weakly regular (Assumption \ref{Ass-A1}) but allow the diffusion kernel $\sigma$ to depend on the particle interaction. This generalization leads to a nonlocal variable-coefficient diffusion term in the Liouville equation \eqref{Liouville-Eq} and a nonlinear distribution‑dependent diffusion term in the limiting equation \eqref{Limit-Eq}. Consequently, the framework of \cite{JW2018} is not directly applicable to our situation, as its key estimates are valid only for a constant diffusion kernel.
 
 On the other hand, to our best knowledge, existing propagation of chaos results for systems with non-constant diffusion coefficients do not cover particle systems of the form \eqref{Particle-sys} governed by Assumptions \ref{Ass-A1} and \ref{Ass-A2}. For mean-field particle systems that can contain System \eqref{Particle-sys}, the early propagation of chaos results by M\'el\'eard \cite{Mel1996},  G\"artner \cite{Gar1988}, and Chiang \cite{Chi1994} all require both the drift and diffusion coefficients to have some continuity. Recently,
 \begin{itemize}
	\item Huang \cite{Hua2026} uses the gradient estimate of decoupled SDEs to prove  the quantitative propagation of chaos in $L^\eta(\eta \in (0, 1])$-Wasserstein distance for mean-field particle systems, under suitable smoothness and Lipschitz assumptions on the coefficients.

    \item Huang \cite{Hua2023} presents several quantitative propagation of chaos results. For mean-field particle systems with the diffusion coefficient allowed to depend on interactions, propagation of chaos estimates with exponential decay rates are proved. These estimates hold under Lipschitz interaction kernels and a nontrivial dissipative drift component independent of the interaction.
	
		\item  Ning and Wu \cite{NW2026} prove the propagation of chaos for a class of $N$-particle stochastic variational inequality systems with superlinear, locally Lipschitz drift coefficients and locally H\"older diffusion coefficients. Such a system is described by SDEs with a subdifferential term and reduces to the classical mean-field particle system when the subdifferential terms vanish.
	
	\item Chassagneux, Szpruch, and Tse \cite{CST2022} prove quantitative propagation of chaos for mean-field particle systems under sufficiently high-order smoothness assumptions on the coefficients and obtain convergence rates of order $1/{N^k}$ for arbitrary $k$ (depending on the smoothness of the coefficients). 
	
	\item Crucianelli and Tangpi \cite[Theorem 3.16]{CT2024} prove the propagation of chaos for an interacting particle system on a random graph with interaction-dependent diffusion allowed. The result holds under Lipschitz and linear growth conditions on the coefficients and the limiting system consists of McKean-Vlasov type graphon SDEs.

\end{itemize}
For recent propagation of chaos results for other types of particle systems with non-constant diffusion coefficients, examples include the following.
\begin{table}[h]
	\centering
    \footnotesize
	\begin{tabular}{c|c|c|c|c}
		\hline
		Example & \makecell[c]{Carrillo, Guo, \\and Jabin \cite{CGJ2024}}& Grass, Guillin, and Poquet \cite{GGP2025} & Feng and Wang \cite{FW2025} & Lacker \cite{Lac2018}\\
		\hline
		\makecell[c]{Diffusion \\ coefficient }& $( \frac{1}{N} \sum\limits^N_{j=1} \sigma(X^i_t- X^j_t) )^{1/2}$ & $( \sigma_1(X^i_t) + \frac{1}{N} \sum\limits^N_{j=1}\sigma_2(X^i_t - X^j_t) )^{1/2}$ & $\frac{\sqrt{2}}{\sqrt{N}} \sum\limits^N_{j=1} \sigma(X^i_t - X^j_t)$ & $\sigma(t, X_t^i)$ \\
		\hline
	\end{tabular}
\end{table}

\noindent Here $\sigma$, $\sigma_1$ and $\sigma_2$ are matrix-valued functions. In addition, the Keller–Segel models in \cite{BCL2026, WLH2026} and the collision-oriented particle system in \cite{DL2024} also have non-constant diffusion coefficients.

In the study of stochastic particle systems with non-constant diffusion coefficients, one may also consider the case of common noise, where all particles are driven by the same Brownian motion or family of Brownian motions. This case corresponds to the so-called  conditional propagation of chaos problem; see \cite{CS2019, CF2016, Nik2025, Ros2020, NRS2022, Hua2025, CO2025} and the references therein. 

 In the current work, we adopt the relative entropy method as our analytical framework and develop a new exponential law of large numbers (large deviation estimate) to address the structural challenges in Eqs. \eqref{Liouville-Eq} and \eqref{Limit-Eq} posed by the interaction-dependent diffusion coefficient in the system \eqref{Particle-sys}. This estimate, presented in Theorem \ref{Thm-LDE} and inspired by \cite[Theorem 4]{JW2018}, constitutes the main technical contribution of this paper. With this estimate, it thus advances the relative entropy method to the setting of interaction-dependent diffusion, ultimately establishes a quantitative propagation of chaos for System \eqref{Particle-sys}.

\subsection{Main result} We begin by stating the precise assumptions on the drift kernel $K$ and the diffusion kernel $\sigma$. 

\begin{assumption}\label{Ass-A1}
	The interaction drift kernel $K: \mathbb{T}^d \to \mathbb{R}^d$ satisfies
	$$ 
	K \in L^\infty (\mathbb{T}^d) \quad \text{and} \quad \div K \in \dot{W}^{-1,\infty} ( \mathbb{T}^d ). $$
\end{assumption}
\begin{assumption} \label{Ass-A2}
	The interaction diffusion kernel $\sigma: \mathbb{T}^d \to \mathbb{R}^{d \times d} $ is diagonal with
	$$\sigma(x)=\textrm{diag}( \sigma_{11}(x), \cdots, \sigma_{dd}(x) )\,,$$ satisfies $\sigma(x) \in W^{2,\infty} ( \mathbb{T}^d )$, and there exists a constant $\underline{\sigma} > 0 $ such that
	$$\sigma_{\alpha\alpha}(x) > \underline{\sigma} \quad \text{for all } x \in \mathbb{T}^d \text{ and all } \alpha = 1, \dots, d. $$
\end{assumption}
In the above assumption on $K$, the space $\dot{W}^{-1,\infty}( \mathbb{T}^d )$ is the set of all functions $f$ with $\int_{ \mathbb{T}^d } f=0 $ for which there exists a vector field $g\in L^\infty(\mathbb{T}^d)$ such that $f =\div g$. Its norm is given by
 \begin{equation*}
 	\| f\|_{ \dot{W}^{-1,\infty} } :=\inf \{ \|g\|_{L^\infty} : g\in L^\infty(\mathbb{T}^d;\mathbb{R}^d)\,, f=\div g \}\,.
 \end{equation*}

\begin{remark}
Assumption \ref{Ass-A1} is inspired by the proof of Theorem 1 in \cite{JW2018}, which establishes the propagation of chaos for System \eqref{Particle-sys} with $\sigma=\sqrt{2}I_d$ under the hypothesis $K \in\dot{W}^{-1,\infty} (\mathbb{T}^d)$ with $\div K \in \dot{W}^{-1,\infty} ( \mathbb{T}^d)$. In that proof, $K$  is decomposed as $$K=\bar{K} +\tilde{K}\,,$$
where $\bar{K}\in \dot{W}^{-1,\infty} (\mathbb{T}^d)$ with $\div \bar{K}=0$ and $\tilde{K}\in L^\infty (\mathbb{T}^d)$ with $\div \tilde{K} \in \dot{W}^{-1,\infty} (\mathbb{T}^d)$. Assumption \ref{Ass-A1} for $K$ is simply that of $\tilde{K}$. One may try to relax the Assumption \ref{Ass-A1} to $K \in\dot{W}^{-1,\infty}$ with $\div K \in \dot{W}^{-1,\infty}$. This potential extension is not pursued here as our analysis focuses on the case of propagation of chaos for System \eqref{Particle-sys} where the diffusion coefficient depends on the interaction.
\end{remark}

\begin{remark}
	Assumption \ref{Ass-A1} on $K$ naturally excludes singular kernels such as the Coulomb kernel, but it covers both smooth and highly oscillatory kernels.
\end{remark}

\begin{remark}
The condition that $\sigma(x) \in W^{2,\infty}$ is a rather strong regularity assumption, introduced primarily to overcome all computational difficulties arising from the diffusion terms in equations \eqref{Limit-Eq} and \eqref{Liouville-Eq}. Such strong regularity requirements are also found in studies of particle systems with non‑constant diffusion coefficients, for example, the diffusion coefficient is required to be at least $C^3$ in \cite{JM1998} for moderately interacting systems and assumed to be $C^1$ in \cite{Nik2025} for systems with common noise.
\end{remark}

The relative entropy method requires analyzing the evolution of the rescaled relative entropy between $\rho_N$ and $ \bar{ \rho }^{ \otimes N }$ (denoted simply by $\bar{ \rho }_N$), defined by
$$ H_N(t) = H_N ( \rho_N | \bar{ \rho }_N )(t) = \frac{1}{N} \int_{\mathbb{T}^{ dN } }  \rho_N \log \frac{ \rho_N }{ \bar{ \rho }_N } \d X \,,$$
where $X := (x_1,\cdots,x_N)$ is the notation used throughout the article for convenience. Following \cite{JW2016, JW2018}, we consider $\rho_N$ to be an entropy solution of Equation \eqref{Liouville-Eq}, as per 
\begin{definition}\label{Def-1-1}(Entropy solution). Let $T>0$, we say $L^\infty([0,T], L^1\cap L^2( \mathbb{T}^{dN} ))$ with $ \rho_N\geq 0$ and $ \int_{ \mathbb{T}^{dN} } \rho_N \d X = 1 $, is an entropy solution to Eq. \eqref{Liouville-Eq} on the time interval $[0,T]$ if
	\begin{itemize}[leftmargin=15pt]
		\item $\rho_N$ solves \eqref{Liouville-Eq} in the sense of distributions, i.e. for any $\varphi \in C^\infty_c([0,T] \times \mathbb{T}^{dN} )$, 
		\begin{equation*}
			\begin{aligned}
				&\int_{ \mathbb{T}^{ dN } }  \rho_N(T) \varphi(T) \d X - 	\int_{ \mathbb{T}^{ dN } }  \rho_N(0) \varphi(0) \d X - \int^t_0 \int_{ \mathbb{T}^{ dN } } \rho_N \partial_t \varphi \d X \d s\\
				& = \dfrac{1}{N} \sum\limits^N_{i=1} \sum \limits_{ k=1 }^N \int^t_0 \int_{ \mathbb{T}^{ dN } }  \rho_N K ( x_i-x_k ) \cdot \nabla_{x_i} \varphi\d X \d s \\
				&+ \dfrac{1}{N^2} \sum^N_{i=1} \sum^d_{\alpha=1} \int^t_0 \int_{ \mathbb{T}^{ dN } }  \rho_N \Big( \sum\limits_{ k=1 }^N \sigma_{\alpha\alpha}( x_i-x_k ) \Big)^2 \partial^2_{ x^\alpha_i } \varphi \d X \d s\,.
			\end{aligned}
		\end{equation*}
		
		\item For a.e. $t\in [0,T]$,
		 \begin{equation}\label{Entropy- dissipation-inequality}
			\begin{aligned}
				\int_{ \mathbb{T}^{ dN } } 
				&\rho_N \log \rho_N \d X 
				\leq \int_{ \mathbb{T}^{ dN } }  \rho^0_N \log \rho^0_N  \d X-\dfrac{1}{N} \sum\limits^N_{i=1} \sum^N_{k=1} \int^t_0 \int_{ \mathbb{T}^{ dN } }  \nabla_{x_i} \cdot K (x_i-x_k)   \rho_N \d X \d s\\
				&+ \frac{1}{ N^2 } \sum^N_{ i=1 }\sum^d_{\alpha=1}\int^t_0\int_{\mathbb{T}^{dN}} \partial^2_{x^\alpha_i}  \Big( \sum_{k = 1}^N \sigma_{\alpha\alpha}(x_i - x_k) \Big)^2   \rho_N  \d X \d s\\
				&- \frac{1}{ N^2 } \sum^N_{ i=1 }\sum^d_{ \alpha = 1 } \int^t_0 \int_{ \mathbb{T}^{ dN } }   \Big( \sum_{k = 1}^N \sigma_{ \alpha\alpha }( x_i - x_k ) \Big)^2  \frac{ ( \partial_{ x^\alpha_i } \rho_N )^2 }{ \rho_N } \d X \d s\,.
			\end{aligned}
		\end{equation}
	\end{itemize}
\end{definition}

\noindent For the limiting equation \eqref{Limit-Eq}, a higher regularity solution is required and we directly consider a  classical solution here.

Now, we state our main convergence result of this paper. Under Assumptions \ref{Ass-A1} and \ref{Ass-A2}, we  prove a key Gronwall inequality for the time derivative of $H_N(t)$, from which we obtain a quantitative rate of convergence of $\rho_N$ to $\bar{\rho}_N$ in the sense of rescaled relative entropy. 
 \begin{theorem}\label{Thm-Rel-Ent-Est} Let $T>0$.
 	If the drift kernel $K$ satisfies Assumption \ref{Ass-A1} and the diffusion kernel $\sigma$  satisfies Assumption \ref{Ass-A2}, $\bar{\rho} \in C^1 ([0,T]; C^2 ( \mathbb{T}^d ) ) $ is a solution the limiting equation \eqref{Limit-Eq}  with $ \inf \bar{ \rho } > 0 $ and $ \int_{ \mathbb{T}^d } \bar{ \rho } \d x = 1 $, and $\rho_N$ is an entropy solution to the Liouville equation \eqref{Liouville-Eq} in the sense of Definition \ref{Def-1-1}, then for any $t\leq T$, the rescaled relative entropy  $H_N$ satisfies the inequality
 	\begin{equation*}
 			H_N ( \rho_N | \bar{ \rho }^{ \otimes N } ) (t) \leq e^{CMt} \Big( H_N ( \rho_N | \bar{ \rho }^{ \otimes N } ) (0) + \frac{1}{N} \Big) \,,
 	\end{equation*}
 	where $C$ is a universal constant and
 		\begin{equation*}
 			\begin{aligned}
 				M & = ( \| K \|_{ L^\infty } + \| \div K \|_{ \dot{W}^{-1,\infty} } ) \frac{ \| \nabla \bar{\rho} \|_{L^\infty} }{ \inf \bar{\rho} } + \frac{d}{ \underline{ \sigma }^2 } \| \div K \|^2_{ \dot{W}^{-1,\infty} } \\
 				& + 12 e^2 \Big(8d \|\sigma\|^2_{ W^{2,\infty} } + 8d \|\sigma\|^2_{ W^{2,\infty} } \frac{ \| \nabla \bar{\rho} \|_{L^\infty} }{ \inf \bar{\rho} } + 2d \|\sigma\|^2_{ W^{2,\infty} } \frac{ \| \nabla^2 \bar{\rho} \|_{L^\infty} }{ \inf \bar{\rho} } \Big)\,.
 			\end{aligned}
 		\end{equation*}
 \end{theorem}

 Theorem \ref{Thm-Rel-Ent-Est} implies a strong form of propagation of chaos for \eqref{Particle-sys}. Indeed, by the classical Csisz\'{a}r-Kullback-Pinsker inequality \cite[Chapter 22]{Vil2009} 
$$ \| \rho^k_N - \bar{\rho}^{ \otimes k } \|_{L^1} \leq \sqrt{ 2k H_k( \rho^k_N | \bar{\rho}^{ \otimes k } ) }$$ and the subadditivity of entropy 
$$  H_k ( \rho^k_N |  \bar{ \rho }^{ \otimes k } ) = \frac{1}{k} \int_{ \mathbb{T}^{ dk } } \rho^k_N \log ( \frac{ \rho_N^k }{ \bar{ \rho }^{ \otimes k } } )\d x_1 \cdots \d x_k \leq H_N ( \rho_N | \bar{ \rho }_N ) $$ 
for any $ k \leq N $, we can obtain
\begin{corollary} Under the assumptions of Theorem \ref{Thm-Rel-Ent-Est}, if additionally  $H_N ( \rho_N | \bar{ \rho }^{ \otimes N } ) (0) \to 0$ as $N\to \infty$, then for all $t\in [0,T]$
		$$H_N ( \rho_N | \bar{ \rho }^{ \otimes N } ) (t) \to 0  \text{ as } N\to \infty\,.$$
Moreover, for any fixed $k\geq 1$
$$ \| \rho^k_N - \bar{\rho}^{ \otimes k } \|_{L^\infty(0,T;L^1(\mathbb{T}^d))} \to 0 \,, \text{ as } N \to \infty \,.$$
\end{corollary}
We remark that in Theorem \ref{Thm-Rel-Ent-Est}, the assumptions on  $\rho_N$ and $\bar{\rho}$ are justifies in Theorems \ref{Thm-entropy-solution-Liou} and \ref{Thm-solution-Limit-Eq}. In fact, the well-posedness of the solutions required for the two equations is established in Theorems \ref{Thm-entropy-solution-Liou} and \ref{Thm-solution-Limit-Eq} below, with proofs provided in Sections \ref{sec-proof-Thm-weak-solution-Liou-Eq} and \ref{sec-proof-Thm-solution-Limit-Eq}, respectively.

\begin{theorem}\label{Thm-entropy-solution-Liou}
	If the drift kernel $K$ satisfies Assumption \ref{Ass-A1} and the diffusion kernel $\sigma$  satisfies Assumption \ref{Ass-A2}, and if the initial data $\rho^0_N$  satisfies $ \rho^0_N \in L^1 \cap L^2 ( \mathbb{T}^{dN} )$ with $ \rho^0_N \geq 0 $, $ \int_{ \mathbb{T}^{dN} } \rho^0_N \d X = 1 $ and $ \int_{ \mathbb{T}^{dN} } \rho^0_N \log \rho^0_N < \infty $, then for any $T>0$, there exists an entropy solution $ \rho_N$ to \eqref{Liouville-Eq} on $[0,T]$ in the sense of Definition \ref{Def-1-1}.
\end{theorem}

\begin{theorem}\label{Thm-solution-Limit-Eq}
	Under Assumption \ref{Ass-A2}, if additionally $K\in L^\infty(\mathbb{T}^d)$ and $\bar{\rho}_0 \in H^{s} ( \mathbb{T}^d )$ with $s > 2 + d/2 $, $\inf \bar{\rho}_0 > 0$ and $\int_{\mathbb{T}^d} \bar{\rho}_0 \d x = 1 $, then for any $T>0$, there exists a unique solution $\bar{\rho} \in C^1 ([0,T]; C^2 ( \mathbb{T}^d ) ) $ to \eqref{Limit-Eq} with $ \inf \bar{\rho} > 0 $ and $\int_{\mathbb{T}^d} \bar{\rho}\d x = 1 $.
\end{theorem}

\begin{remark}
	The condition $\div K\in \dot{W}^{-1,\infty}$ is indeed not used in the proof of Theorem \ref{Thm-solution-Limit-Eq} due to the regularizing effect of convolution; see Sec. \ref{sec-proof-Thm-solution-Limit-Eq} for details.
\end{remark}

\subsection{Method and difficulties}
As noted, the proof of the main result (Theorem \ref{Thm-Rel-Ent-Est}) follows the relative entropy framework of \cite{JW2018}. We begin by deriving the evolution of the relative entropy  between $\rho_N$ and $\bar{ \rho }_N$,  obtaining
\begin{equation*}
	\begin{aligned}
	  H_N (t) \leq & H_N(0) + \underbrace{ \int^t_0\int_{\mathbb{T}^{dN}} \rho_N \dfrac{1}{N^2} \sum\limits^N_{ i,k=1 } \phi_1 (x_i,x_k) }_{=:E_1} + \underbrace{ \int^t_0\int_{\mathbb{T}^{dN}} \rho_N \frac{1}{N^3} \sum\limits^N_{ i,j,k=1 } \phi_2(x_i,x_j,x_k) }_{=:E_2} \\
		&- \underbrace{ \frac{1}{N^3} \sum\limits^N_{ i=1 } \sum^d_{ \alpha=1 } \int^t_0\int_{\mathbb{T}^{dN}} \rho_N ( \sum\limits^N_{ k=1 } \sigma_{\alpha\alpha}(x_i-x_k) )^2 \left( \partial_{ x^\alpha_i } \log \frac{\rho_N}{ \bar{\rho}_N} \right)^2 }_{=:E_3}\,,
	\end{aligned}
\end{equation*}
where $E_1$ and $E_2$ are the error terms associated with the drift kernel $K$ and the diffusion kernel $\sigma$ , respectively, and $E_3$ is an weighted relative Fisher information term (see \eqref{Ent-evo-ineq}). The full derivation is provided in Section \ref{sec-proof-Thm-Rel-Ent-Est}.

The term $E_2$ is indeed the essential difference between our analysis and that of \cite{JW2018}. This distinction arises because the diffusion terms in our equations \eqref{Liouville-Eq} and \eqref{Limit-Eq} have a more complex structure than the constant diffusion coefficients considered in \cite{JW2018}. Consequently, the terms involving $\sigma$ cannot be absorbed into a relative Fisher information term as in \cite{JW2018}, leading directly to the new error term $E_2$. Moreover, the $E_1$ term can be handled directly by applying Lemma 4 of \cite{JW2018}, which is derived from Theorem 3 (an exponential law of large numbers) and Theorem 4 (a large deviation estimate) therein.

We now turn to the treatment of $E_2$. A natural idea is to adapt the proof of Theorem 4 in \cite{JW2018}, where $E_1$ is handled by leveraging two cancellation properties of $\phi_1$ and a combinatorics argument. This motivates a similar strategy for $E_2$ because $E_2$ shares a similar form with $E_1$ and $\phi_2$ satisfies two analogous cancellation properties, detailed in \eqref{Cancel-rules}. However, $\phi_2$ contains nonlinear terms such as $\sigma_{\alpha\alpha}(x_i-x_j) \sigma_{\alpha\alpha}(x_i-x_k)$, a structure absent in $\phi_1$. This nonlinearity disrupts the combinatorics argument from \cite{JW2018}, preventing its direct application and constituting the main difficulty of our proof. 
To overcome it, we develop a dynamic combinatorial counting rule which can effectively handle the nonlinearity in Section \ref{sec-Thm-LDE}. This allows us to prove a new large deviation
estimate for $\sigma$ (see Theorem \ref{Thm-LDE}) and obtain a uniform bound for $E_2$. With these estimates in hand, the main theorem follows by the Gronwall lemma.

\vspace{2mm}
 
\noindent{\bf Organization of the paper:} The next section is devoted to deriving the main result under the assumption of Theorem \ref{Thm-LDE} (the large deviation
estimate for the diffusion kernel $\sigma$). In Section \ref{sec-Thm-LDE}, we detail our combinatorics argument based on two cancellation properties for $\sigma$ and prove Theorem \ref{Thm-LDE}. Concluding the paper, we establish the existence and uniqueness of classical solutions to the limiting Eq. \eqref{Limit-Eq} in Section \ref{sec-proof-Thm-solution-Limit-Eq} and the existence of entropy solutions to the Liouville Eq. \eqref{Liouville-Eq} in Section \ref{sec-proof-Thm-weak-solution-Liou-Eq}.
 
\section{Proof of Theorem \ref{Thm-Rel-Ent-Est}}\label{sec-proof-Thm-Rel-Ent-Est} 

The goal here is to prove the propagation of chaos for the particle system \eqref{Particle-sys}, i.e. Theorem \ref{Thm-Rel-Ent-Est}. As mentioned earlier, we follow the relative entropy method of \cite{JW2016, JW2018} and assume the results of Theorems \ref{Thm-entropy-solution-Liou} and \ref{Thm-solution-Limit-Eq}, whose proofs will be given later. 

\vspace{2mm}
\noindent\textbf{\emph{Step 1: Derive the time evolution of relative entropy.}} By the definition of the rescaled relative entropy, we know 
\begin{equation}\label{HN}
	H_N(t) = \frac{1}{N} \int_{ \mathbb{T}^{ dN } } \rho_N \log (\frac{ \rho_N }{ \bar{\rho}_N } ) = \frac{1}{N} \int_{ \mathbb{T}^{ dN } } \rho_N \log \rho_N -\frac{1}{N} \int_{ \mathbb{T}^{ dN } } \rho_N \log \bar{\rho}_N := \frac{1}{N} ( I_1 - I_2 ) \,.
\end{equation}
 For $I_1$, as $\rho_N$ is an entropy solution of Equation \eqref{Liouville-Eq} satisfying the inequality \eqref{Entropy- dissipation-inequality}, we have 
 \begin{equation}\label{ineq-I1}
	\begin{aligned}
		I_1 \leq & \int_{\mathbb{T}^{ dN } } \rho^0_N \log \rho^0_N \d X - \dfrac{1}{N} \sum \limits^N_{ i=1 } \sum^N_{ k=1 } \int^t_0 \int_{ \mathbb{T}^{ dN } }  \div K (x_i-x_k)   \rho_N \d X \d s \\
		& +
		\frac{1}{ N^2 } \sum\limits^N_{ i=1 } \sum^d_{ \alpha = 1 } \int^t_0 \int_{ \mathbb{T}^{ dN } } \partial^2_{ x^\alpha_i } \Big( \sum\limits_{ k = 1 }^N \sigma_{ \alpha\alpha } ( x_i - x_k ) \Big)^2  \rho_N \d X \d s\\
		& - \frac{1}{N^2} \sum\limits^N_{ i=1 } \sum^d_{ \alpha = 1 } \int^t_0 \int_{ \mathbb{T}^{ dN } } \Big( \sum\limits_{ k = 1 }^N \sigma_{ \alpha\alpha } ( x_i - x_k ) \Big)^2 \frac{ ( \partial_{ x^\alpha_i } \rho_N)^2 }{ \rho_N } \d X \d s \,.
	\end{aligned}
\end{equation}
Note that $\rho_N$ satisfies Equation \eqref{Liouville-Eq} in the sense of distribution. Since $\bar{\rho} \in C^1 (0,T; C^2 ( \mathbb{T}^d ) ) $ and $ \inf \bar{\rho} > 0 $, $ \log \bar{\rho}_N $ can be used as a test function. This yields
\begin{equation*}
	\begin{aligned}
		I_2 = & \int_{ \mathbb{T}^{ dN } } \rho^0_N \log \bar{\rho}^0_N
		+ \int^t_0 \int_{ \mathbb{T}^{ dN } } \rho_N \partial_t \log \bar{\rho}_N + \dfrac{1}{N} \sum\limits^N_{ i=1 } \sum \limits_{ k=1 }^N \int^t_0 \int_{ \mathbb{T}^{ dN } } \rho_N K ( x_i - x_k ) \cdot \nabla_{x_i} \log \bar{\rho}_N  \\
		& + \frac{1}{N^2} \sum\limits^N_{ i=1 } \sum^d_{ \alpha=1 } \int^t_0 \int_{ \mathbb{T}^{ dN } }  \Big( \sum\limits^N_{ k=1 } \sigma_{\alpha\alpha} ( x_i - x_k ) \Big)^2 \partial^2_{ x^\alpha_i } \log \bar{\rho}_N \,.
	\end{aligned}
\end{equation*}
To facilitate later computations, we rewrite $I_2$ as
\begin{equation*}
	\begin{aligned}
		I_2 = & \int_{ \mathbb{T}^{dN} } \rho^0_N \log \bar{\rho}^0_N
		+\int^t_0 \int_{ \mathbb{T}^{dN} } \rho_N \partial_t \log \bar{\rho}_N  +\dfrac{1}{N} \sum\limits^N_{ i=1 }\sum \limits_{ k=1 }^N \int^t_0 \int_{ \mathbb{T}^{dN} } \rho_N  K ( x_i - x_k ) \cdot \nabla_{x_i} \log \bar{\rho}_N \\
		&+ \frac{2}{N^2} \sum\limits^N_{ i=1 } \sum^d_{ \alpha=1 } \int^t_0\int_{ \mathbb{T}^{dN} } \rho_N \Big( \sum\limits^N_{ k=1 } \sigma_{ \alpha\alpha } ( x_i - x_k ) \Big)^2 \partial^2_{ x^\alpha_i } \log \bar{\rho}_N \\
		& -\frac{1}{N^2} \sum\limits^N_{ i=1 }\sum^d_{ \alpha=1 } \int^t_0 \int_{ \mathbb{T}^{dN} } \rho_N \Big( \sum\limits^N_{ k=1 } \sigma_{ \alpha\alpha } ( x_i - x_k ) \Big)^2 \partial^2_{ x^\alpha_i } \log \bar{\rho}_N \,.
	\end{aligned}
\end{equation*}
By applying integration by parts to the fourth term of the above equation and using the equality
\begin{equation*}
	\partial^2_{ x^\alpha_i } \log \bar{\rho}_N =	\frac{ \partial^2_{ x^\alpha_i } \bar{\rho}_N }{ \bar{\rho}_N } - \frac{ | \partial_{ x^\alpha_i } \bar{\rho}_N |^2 }{ \bar{\rho}^2_N }
\end{equation*}
to its last term, we have
 \begin{equation}\label{eq-I2}
 	\begin{aligned}
 		I_2
 		 = & \int_{ \mathbb{T}^{dN} } \rho^0_N \log \bar{\rho}^0_N
 		+ \int^t_0 \int_{ \mathbb{T}^{dN} } \rho_N \partial_t \log \bar{\rho}_N + \dfrac{1}{N} \sum\limits^N_{ i=1 } \sum\limits_{ k=1 }^N \int^t_0 \int_{ \mathbb{T}^{dN} } \rho_N  K ( x_i - x_k ) \cdot \nabla_{x_i} \log \bar{\rho}_N \\
 		& - \frac{2}{N^2} \sum^N_{ i=1 } \sum^d_{ \alpha=1 } \int^t_0 \int_{ \mathbb{T}^{dN} } \rho_N  \partial_{ x^\alpha_i } \Big( \sum^N_{ k=1 } \sigma_{ \alpha\alpha } ( x_i - x_k ) \Big)^2 \partial_{ x^\alpha_i } \log \bar{\rho}_N \\
 		& -  \frac{2}{N^2} \sum^N_{ i=1 } \sum^d_{ \alpha=1 }\int^t_0 \int_{ \mathbb{T}^{dN} } \rho_N  \Big( \sum^N_{ k=1 } \sigma_{ \alpha\alpha } ( x_i - x_k ) \Big)^2 \partial_{ x^\alpha_i } \log \rho_N \partial_{ x^\alpha_i } \log \bar{\rho}_N \\
 		& - \frac{1}{N^2} \sum\limits^N_{ i=1 } \sum^d_{ \alpha=1 } \int^t_0 \int_{ \mathbb{T}^{dN} } \rho_N \Big( \sum\limits^N_{ k=1 } \sigma_{ \alpha\alpha } ( x_i - x_k ) \Big)^2 \frac{ \partial^2_{ x^\alpha_i } \bar{\rho}_N}{ \bar{\rho}_N } \\
 		& + \frac{1}{N^2} \sum\limits^N_{ i=1 } \sum^d_{ \alpha=1 } \int^t_0 \int_{ \mathbb{T}^{dN} } \rho_N \Big( \sum\limits^N_{ k=1 } \sigma_{ \alpha\alpha }( x_i - x_k ) \Big)^2 \frac{ | \partial_{ x^\alpha_i } \bar{\rho}_N |^2 }{ \bar{\rho}^2_N } \,.
 	\end{aligned}
 \end{equation}
Substituting \eqref{ineq-I1} and \eqref{eq-I2} into \eqref{HN} yields
\begin{equation}\label{HN-bar-I}
	H_N(t) \leq H_N(0) + \frac{1}{N} \int^t_0\int_{ \mathbb{T}^{dN} } \rho_N ( \bar{I}_1 -\bar{I}_2 ) \,,
\end{equation}
where
\begin{equation}\label{bar-I1}
	\begin{aligned}
		\bar{I}_1 = &- \dfrac{1}{N} \sum\limits^N_{ i=1 } \sum\limits^N_{ k=1 } \nabla_{x_i} \cdot  K ( x_i - x_k ) + \frac{1}{N^2} \sum\limits^N_{ i=1 } \sum^d_{ \alpha=1 } \partial^2_{ x^\alpha_i } \Big( \sum\limits_{ k=1 }^N \sigma_{ \alpha\alpha } ( x_i - x_k ) \Big)^2 \\
		& -  \frac{1}{N^2} \sum\limits^N_{ i=1 } \sum^d_{ \alpha=1 } \Big(  \sum\limits_{ k=1 }^N \sigma_{ \alpha\alpha } ( x_i - x_k ) \Big)^2 \frac{ | \partial_{x^\alpha_i} \rho_N |^2 }{ \rho^2_N } \\
	\end{aligned}
\end{equation}
and
\begin{equation}\label{bar-I2-1}
	\begin{aligned}
		\bar{I}_2 = & \partial_t \log \bar{\rho}_N + \dfrac{1}{N} \sum\limits^N_{ i=1 } \sum\limits_{ k=1 }^N K ( x_i - x_k ) \cdot \nabla_{x_i} \log \bar{\rho}_N \\ & - \frac{2}{N^2} \sum^N_{ i=1 } \sum^d_{ \alpha=1 } \partial_{ x^\alpha_i } \Big( \sum^N_{ k=1 } \sigma_{ \alpha\alpha } ( x_i - x_k )  \Big)^2 \partial_{ x^\alpha_i } \log \bar{\rho}_N \\
		&- \frac{2}{N^2} \sum^N_{ i=1 } \sum^d_{ \alpha=1 } \Big( \sum^N_{ k=1 } \sigma_{ \alpha\alpha } ( x_i - x_k ) \Big)^2 \partial_{ x^\alpha_i } \log \rho_N \partial_{ x^\alpha_i } \log \bar{\rho}_N \\
		&- \frac{1}{N^2} \sum\limits^N_{ i=1 } \sum^d_{ \alpha=1 } \Big( \sum\limits^N_{ k=1 } \sigma_{ \alpha\alpha } ( x_i - x_k ) \Big)^2 \frac{ \partial^2_{ x^\alpha_i } \bar{\rho}_N }{ \bar{\rho}_N } + \frac{1}{N^2} \sum\limits^N_{ i=1 } \sum^d_{ \alpha=1 } \Big( \sum\limits^N_{ k=1 } \sigma_{ \alpha\alpha } ( x_i - x_k ) \Big)^2 \frac{ | \partial_{ x^\alpha_i } \bar{\rho}_N |^2}{ \bar{\rho}^2_N } \,.
	\end{aligned}
\end{equation}

\noindent\textbf{\emph{Step 2: Reorganize the terms in $ \bar{I}_1 - \bar{I}_2 $.}}  We first rewrite $\bar{I}_2$.
Recalling the equation \eqref{Limit-Eq} and noting that $\partial_t \bar{\rho}_N = \Pi^N_{ j \neq i } \bar{\rho}(x_j) \partial_t \bar{\rho}(x_i) $, we have
\begin{equation*}
	\begin{aligned}
		\partial_t \bar{\rho}_N
		= &
		- \sum^N_{i=1} \nabla_{x_i} \cdot  K*\bar{\rho}(x_i) \bar{\rho}_N - \sum^N_{i=1} K*\bar{\rho}(x_i) \cdot \nabla_{x_i} \bar{\rho}_N + \sum^N_{ i=1 } \sum^d_{ \alpha=1 } \partial^2_{ x^\alpha_i } \big( \sigma_{ \alpha\alpha }*\bar{\rho}(x_i) \big)^2 \bar{\rho}_N \\
		& + 2 \sum^N_{ i=1 } \sum^d_{ \alpha=1 } \partial_{ x^\alpha_i } \big( \sigma_{ \alpha\alpha }*\bar{\rho}(x_i) \big)^2  \partial_{ x^\alpha_i } \bar{\rho}_N + \sum^N_{ i=1 } \sum^d_{ \alpha=1 } \big( \sigma_{ \alpha\alpha }*\bar{\rho}(x_i) \big)^2 \partial^2_{ x^\alpha_i } \bar{\rho}_N \,.
	\end{aligned}
\end{equation*}
 Combined with the relation $\partial_t \log \bar{\rho}_N = \frac{1}{ \bar{\rho}_N } \partial_t \bar{\rho}_N $, it follows from \eqref{bar-I2-1} that
\begin{equation*}
	\begin{aligned}
		\bar{I}_2 = & -\sum^N_{ i=1 } \nabla_{x_i} \cdot K*\bar{\rho}(x_i) - \sum^N_{ i=1 } K*\bar{\rho}(x_i)  \cdot \nabla_{x_i} \log \bar{\rho}_N
		+ \sum^N_{ i=1 } \sum^d_{ \alpha=1 } \partial^2_{ x^\alpha_i } \big( \sigma_{ \alpha\alpha }*\bar{\rho}(x_i) \big)^2 \\
		& + 2 \sum^N_{ i=1 } \sum^d_{ \alpha=1 } \partial_{ x^\alpha_i } \big( \sigma_{ \alpha\alpha }*\bar{\rho}(x_i) \big)^2 \partial_{ x^\alpha_i } \log \bar{\rho}_N + \sum^N_{ i=1 } \sum^d_{ \alpha=1 } \big( \sigma_{ \alpha\alpha }*\bar{\rho}(x_i) \big)^2 \frac{ \partial^2_{ x^\alpha_i } \bar{\rho}_N }{ \bar{\rho}_N } \\
		& + \dfrac{1}{N} \sum\limits^N_{ i=1 } \sum\limits_{ k=1 }^N K( x_i - x_k ) \cdot \nabla_{x_i} \log \bar{\rho}_N - \frac{2}{N^2} \sum^N_{ i=1 } \sum^d_{ \alpha=1 } \partial_{ x^\alpha_i } \Big( \sum^N_{ k=1 } \sigma_{ \alpha\alpha } ( x_i - x_k ) \Big)^2 \partial_{ x^\alpha_i } \log \bar{\rho}_N \\
		&- \frac{1}{N^2} \sum\limits^N_{ i=1 } \sum^d_{ \alpha=1 } \Big( \sum\limits^N_{ k=1 } \sigma_{ \alpha\alpha } ( x_i - x_k ) \Big)^2 \frac{ \partial^2_{ x^\alpha_i } \bar{\rho}_N }{ \bar{\rho}_N } \\
		& + \frac{1}{N^2} \sum\limits^N_{ i=1 } \sum^d_{ \alpha=1 } \Big( \sum\limits^N_{ k=1 } \sigma_{ \alpha\alpha } ( x_i - x_k ) \Big)^2 \Big( - 2 \partial_{ x^\alpha_i } \log \rho_N \partial_{ x^\alpha_i } \log \bar{\rho}_N + \frac{ | \partial_{ x^\alpha_i } \bar{\rho}_N |^2 }{ \bar{\rho}^2_N } \Big) \,.
	\end{aligned}
\end{equation*}
Note that for some specific terms in the above equation, the following relationships hold:
\begin{itemize}
	\item the 2nd and 6th terms: 
	\begin{equation*}
		\begin{aligned}
			- &\sum^N_{i=1} K*\bar{\rho}(x_i) \cdot \nabla_{x_i} \log \bar{\rho}_N +\frac{1}{N} \sum^N_{i=1} \sum\limits_{ k=1 }^N K ( x_i - x_k ) \cdot \nabla_{x_i} \log \bar{\rho}_N \\
			& = \frac{1}{N} \sum^N_{i=1} \sum\limits_{ k=1 }^N \Big( K ( x_i - x_k ) - K*\bar{\rho}(x_i) \Big) \cdot \nabla_{x_i} \log \bar{\rho}_N \,;
		\end{aligned}
	\end{equation*}
	\item the 4th and 7th terms:
	\begin{equation*}
		\begin{aligned}
			& 2 \sum^N_{ i=1 } \sum^d_{ \alpha=1 } \partial_{ x^\alpha_i } ( \sigma_{ \alpha\alpha } *\bar{\rho} (x_i) )^2 \partial_{ x^\alpha_i } \log \bar{\rho}_N  -\frac{2}{N^2} \sum^N_{ i=1 } \sum^d_{ \alpha=1 } \partial_{ x^\alpha_i } \Big( \sum^N_{ k=1 }\sigma_{ \alpha\alpha }( x_i - x_k ) \Big)^2 \partial_{ x^\alpha_i } \log \bar{\rho}_N \\
			= & - \frac{2}{N^2} \sum^N_{ i,j,k=1 } \sum^d_{ \alpha=1 } \partial_{ x^\alpha_i } \Big(  \sigma_{ \alpha\alpha } ( x_i - x_k ) \sigma_{ \alpha\alpha } ( x_i - x_j ) - ( \sigma_{\alpha\alpha}*\bar{\rho} (x_i) )^2 \Big) \partial_{ x^\alpha_i } \log \bar{\rho}_N \,;
		\end{aligned}
	\end{equation*}
	\item the 5th and 8th terms:
	\begin{equation*}
		\begin{aligned}
			& \sum^N_{ i=1 } \sum^d_{ \alpha=1 } \big( \sigma_{ \alpha\alpha }*\bar{\rho}(x_i) \big)^2 \frac{ \partial^2_{ x^\alpha_i } \bar{\rho}_N }{ \bar{\rho}_N } - \frac{1}{N^2} \sum\limits^N_{ i=1 } \sum^d_{ \alpha=1 } \Big( \sum\limits^N_{ k=1 } \sigma_{ \alpha\alpha } ( x_i - x_k ) \Big)^2 \frac{ \partial^2_{ x^\alpha_i } \bar{\rho}_N }{ \bar{\rho}_N } \\
			= & - \frac{1}{N^2} \sum\limits^N_{ i,j,k=1 } \sum^d_{ \alpha=1 } \Big( \sigma_{ \alpha\alpha }( x_i - x_j ) \sigma_{ \alpha\alpha } ( x_i - x_k ) - ( \sigma_{ \alpha\alpha }*\bar{\rho}(x_i) )^2 \Big) \frac{ \partial^2_{ x^\alpha_i } \bar{\rho}_N }{ \bar{\rho}_N }\,.
		\end{aligned}
	\end{equation*}
\end{itemize}
Hence
\begin{equation}\label{bar-I2-2}
	\begin{aligned}
		\bar{I}_2 = & - \sum^N_{ i=1 } \nabla_{x_i} \cdot K*\bar{\rho}(x_i) + \sum^N_{ i=1 } \sum^d_{ \alpha=1 } \partial^2_{ x^\alpha_i } ( \sigma_{\alpha\alpha} *\bar{\rho}(x_i) )^2 \\
        &+ \sum^N_{i=1} \sum \limits_{ k=1 }^N \Big( K ( x_i - x_k ) - K*\bar{\rho}(x_i) \Big) \cdot \nabla_{x_i} \log \bar{\rho}_N \\
		& - \frac{1}{N^2} \sum\limits^N_{ i,j,k=1 } \sum^d_{ \alpha=1 } \Big( \sigma_{ \alpha\alpha } ( x_i - x_j ) \sigma_{\alpha\alpha} ( x_i - x_k ) - ( \sigma_{\alpha\alpha}*\bar{\rho}(x_i) )^2 \Big) \frac{ \partial^2_{ x^\alpha_i } \bar{\rho}_N }{ \bar{\rho}_N } \\
		& - \frac{2}{N^2} \sum^N_{ i,j,k=1 } \sum^d_{ \alpha=1 } \partial_{ x^\alpha_i } \Big(  \sigma_{\alpha\alpha} ( x_i - x_k ) \sigma_{\alpha\alpha} ( x_i - x_j ) - ( \sigma_{\alpha\alpha}*\bar{\rho}(x_i) )^2 \Big) \partial_{ x^\alpha_i } \log \bar{\rho}_N \\
		& + \frac{1}{N^2} \sum\limits^N_{ i=1 } \sum^d_{ \alpha=1 }( \sum\limits^N_{ k=1 } \sigma_{\alpha\alpha} ( x_i - x_k ) )^2 \Big( - 2 \partial_{ x^\alpha_i } \log \rho_N \partial_{ x^\alpha_i } \log \bar{\rho}_N + \frac{ | \partial_{ x^\alpha_i } \bar{\rho}_N |^2}{ \bar{\rho}^2_N } \Big) \,.
	\end{aligned}
\end{equation}
Combining \eqref{bar-I1} and \eqref{bar-I2-2} gives
\begin{equation*}
	\begin{aligned}
		\bar{I}_1 - \bar{I}_2 = & - \Big( \dfrac{1}{N} \sum\limits^N_{ i=1 } \sum\limits^N_{ k=1 }  \nabla_{x_i} \cdot K ( x_i - x_k )  - \sum^N_{ i=1 } \nabla_{x_i} \cdot K*\bar{\rho}(x_i)  \Big) \\
		& - \sum^N_{ i=1 } \sum\limits_{ k=1 }^N \Big( K ( x_i - x_k ) - K*\bar{\rho}(x_i) \Big) \cdot \nabla_{x_i} \log \bar{\rho}_N \\
		&+  \frac{1}{N^2} \sum\limits^N_{ i=1 } \sum^d_{ \alpha=1 } 
		\Big( \partial^2_{x^\alpha_i} \Big( \sum\limits_{ k=1 }^N \sigma_{\alpha\alpha} ( x_i - x_k ) \Big)^2 -  \partial^2_{ x^\alpha_i } ( \sigma_{ \alpha\alpha} *\bar{\rho}(x_i) )^2  \Big) \\
		& +  \frac{1}{N^2} \sum\limits^N_{ i,j,k=1 } \sum^d_{ \alpha=1 } \Big( \sigma_{\alpha\alpha} ( x_i - x_j ) \sigma_{\alpha\alpha} ( x_i - x_k ) - ( \sigma_{\alpha\alpha}*\bar{\rho}(x_i) )^2 \Big) \frac{ \partial^2_{ x^\alpha_i } \bar{\rho}_N }{ \bar{\rho}_N } \\
		& + \frac{2}{N^2} \sum^N_{ i,j,k=1 } \sum^d_{ \alpha=1 } \partial_{ x^\alpha_i } \Big(  \sigma_{\alpha\alpha} ( x_i - x_k ) \sigma_{\alpha\alpha} ( x_i - x_j ) - ( \sigma_{\alpha\alpha}*\bar{\rho}(x_i) )^2 \Big) \partial_{ x^\alpha_i } \log \bar{\rho}_N \\
		&- \frac{1}{N^2} \sum\limits^N_{ i=1 } \sum^d_{ \alpha=1 } \Big( \sum\limits^N_{ k=1 } \sigma_{\alpha\alpha} ( x_i - x_k ) \Big)^2 \Big( \frac{ | \partial_{ x^\alpha_i } \rho_N |^2 }{ \rho^2_N }- 2 \partial_{ x^\alpha_i } \log \rho_N \partial_{ x^\alpha_i } \log \bar{\rho}_N + \frac{ |\partial_{ x^\alpha_i } \bar{\rho}_N |^2}{ \bar{\rho}^2_N } \Big) \,.
	\end{aligned}
\end{equation*}
Note that 
\begin{equation*}
	\frac{ (\partial_{ x^\alpha_i } \rho_N )^2 }{ \rho^2_N }- 2 \partial_{ x^\alpha_i } \log \rho_N \partial_{ x^\alpha_i } \log \bar{\rho}_N + \frac{ ( \partial_{ x^\alpha_i } \bar{\rho}_N )^2}{ \bar{\rho}^2_N } = \left( \partial_{ x^\alpha_i } \log \frac{ \rho_N }{ \bar{\rho}_N } \right)^2 \,.
\end{equation*}
Then we derive
\begin{equation}\label{New-bar-I1-I2}
	\begin{aligned}
		\bar{I}_1 - \bar{I}_2 
		= & \dfrac{1}{N} \sum\limits^N_{ i,k=1 } \phi_1(x_i,x_k) + \frac{1}{N^2} \sum\limits^N_{ i,j,k=1 } \phi_2(x_i,x_j,x_k) \\
		& -  \frac{1}{N^2} \sum\limits^N_{ i=1 } \sum^d_{ \alpha=1 } \Big( \sum\limits^N_{ k=1 } \sigma_{\alpha\alpha}( x_i - x_k ) \Big)^2 \left( \partial_{ x^\alpha_i } \log \frac{\rho_N}{ \bar{\rho}_N }  \right)^2\,,
	\end{aligned}
\end{equation}
where
\begin{equation}\label{phi1-phi2}
	\begin{aligned}
		\phi_1(x_i,x_k) = & - \big( \div K (x_i-x_k)  - \div K*\bar{\rho}(x_i)  \big)
		-\big( K(x_i - x_k) - K*\bar{\rho}(x_i) \big) \cdot \nabla_{x_i} \log \bar{\rho}_N \,,\\
		\phi_2(x_i,x_j,x_k) = & \sum^d_{ \alpha=1 } \partial^2_{ x^\alpha_i } \big(  \sigma_{\alpha\alpha}( x_i - x_j) \sigma_{\alpha\alpha}(x_i - x_k) - ( \sigma_{\alpha\alpha}*\bar{\rho}(x_i) )^2  \big) \\
		& + 2 \sum^d_{ \alpha=1 } \partial_{ x^\alpha_i } \big(  \sigma_{\alpha\alpha}(x_i-x_j) \sigma_{\alpha\alpha}(x_i-x_k) - ( \sigma_{\alpha\alpha}*\bar{\rho}(x_i) )^2 \big) \partial_{ x^\alpha_i } \log \bar{\rho}(x_i)\\
		& + \sum^d_{ \alpha=1 } \big( \sigma_{\alpha\alpha}(x_i-x_j) \sigma_{\alpha\alpha}(x_i-x_k) - ( \sigma_{\alpha\alpha}*\bar{\rho}(x_i) )^2 \big) \frac{ \partial^2_{x^\alpha_i} \bar{\rho}(x_i) }{ \bar{\rho}(x_i) } \,.
	\end{aligned}
\end{equation}
From \eqref{HN-bar-I} and \eqref{New-bar-I1-I2}, we obtain 
\begin{equation}\label{Ent-evo-ineq}
	\begin{aligned}
		H_N(t) \leq & H_N(0) + \int^t_0\int_{\mathbb{T}^{dN}} \rho_N \dfrac{1}{N^2} \sum\limits^N_{ i,k=1 } \phi_1 (x_i,x_k) + \int^t_0\int_{\mathbb{T}^{dN}} \rho_N \frac{1}{N^3} \sum\limits^N_{ i,j,k=1 } \phi_2(x_i,x_j,x_k) \\
		&-  \frac{1}{N^3} \sum\limits^N_{ i=1 } \sum^d_{ \alpha=1 } \int^t_0\int_{\mathbb{T}^{dN}} \rho_N \big( \sum\limits^N_{ k=1 } \sigma_{\alpha\alpha}(x_i-x_k) \big)^2 \left( \partial_{ x^\alpha_i } \log \frac{\rho_N}{ \bar{\rho}_N} \right)^2 \\
		\leq & H_N(0) + \int^t_0\int_{\mathbb{T}^{dN}} \rho_N \dfrac{1}{N^2} \sum\limits^N_{ i,k=1 } \phi_1 (x_i,x_k) + \int^t_0\int_{\mathbb{T}^{dN}} \rho_N \frac{1}{N^3} \sum\limits^N_{ i,j,k=1 } \phi_2(x_i,x_j,x_k) \\
		& -  \frac{ \underline{\sigma}^2 }{N} \sum\limits^N_{ i=1 } \int^t_0\int_{\mathbb{T}^{dN}} \rho_N \left| \nabla_{x_i} \log \frac{\rho_N}{ \bar{\rho}_N }  \right|^2\,,
	\end{aligned}
\end{equation}
where the last inequality is due to Assumption \ref{Ass-A2}, with $\phi_1$ and $\phi_2$ given by \eqref{phi1-phi2}. Then applying Lemma 4 in \cite{JW2018} under Assumption \ref{Ass-A1} gives
\begin{equation*}
	\begin{aligned}
		\int_{\mathbb{T}^{dN}} \rho_N \dfrac{1}{N^2} \sum\limits^N_{ i,k=1 } \phi_1 (x_i,x_k) \leq \frac{ \underline{\sigma}^2 }{4N} \sum\limits^N_{ i=1 } \int_{\mathbb{T}^{dN}} \rho_N \left| \nabla_{x_i} \log \frac{\rho_N}{ \bar{\rho}_N }  \right|^2 + CM_K( H_N(t) + \frac{1}{N} )\,,
	\end{aligned}
\end{equation*}
where $C$ is a universal constant and
\begin{equation*}
	M_K = ( \|K\|_{L^\infty} + \| \div K \|_{ \dot{W}^{-1,\infty} } )\frac{ \|\nabla\bar{\rho}\|_{L^\infty} }{ \inf \bar{\rho} } + \frac{d}{ \underline{\sigma}^2 } \| \div K \|^2_{ \dot{W}^{-1,\infty} } \,.
\end{equation*}
Hence
\begin{equation}\label{HN-phi}
	\begin{aligned}
		H_N(t)
		\leq & H_N(0) + CM_K ( H_N(t) + \frac{1}{N} ) + \int^t_0 \int_{\mathbb{T}^{dN}}\rho_N \frac{1}{N^3} \sum\limits^N_{ i,j,k=1 } \phi_2 (x_i,x_j,x_k) \,.
	\end{aligned}
\end{equation}
The lemma reads as follows
\begin{lemma}\label{L4-JW18}\cite[Lemma 4]{JW2018}
	Assume that $ \bar{\rho} \in W^{1,p} $ for any $ p < \infty $, then for any kernel $K \in L^\infty ( \mathbb{T}^d ) $ with $\div K \in \dot{W}^{-1,\infty}$, one has that
	\begin{equation*}
			\begin{aligned}
				- & \frac{1}{N^2} \sum^N_{ i,j }\int_{ \mathbb{T}^{dN}}  \rho_N \big( K(x_i - x_k) - K*\bar{\rho}(x_i) \big) \cdot \nabla_{x_i} \log \bar{\rho}_N \d X \\
				- & \frac{1}{N^2} \sum^N_{i,j} \int_{\mathbb{T}^{dN}} \rho_N \big( \nabla_{x_i} \cdot K (x_i-x_k) - \nabla_{x_i} \cdot K*\bar{\rho}(x_i) \big) \d X \\
				\leq & \frac{ \underline{\sigma} }{4N} \sum\limits^N_{ i=1 } \int_{\mathbb{T}^{dN}} \rho_N \left| \nabla_{x_i} \log \frac{\rho_N}{ \bar{\rho}_N }  \right|^2 
					d X + CM_K ( H_N( \rho_N | \bar{\rho}_N) + \frac{1}{N} ) \,,
			\end{aligned}
		\end{equation*}
	where $C$ is a universal constant and 
		\begin{equation*}
			M_K = ( \|K\|_{L^\infty} + \| \div K \|_{\dot{W}^{-1,\infty} } )\frac{ \| \nabla \bar{\rho} \|_{L^\infty} }{ \inf \bar{\rho} } + \frac{d}{ \underline{\sigma} } \| \div K\|^2_{ \dot{W}^{-1,\infty} } \,.
		\end{equation*}
\end{lemma}
\begin{remark}
	The $\underline{\sigma}^2$ in our result differs from the $\underline{\sigma}$ in Lemma \ref{L4-JW18} because the diffusion coefficient in \cite{JW2018} is defined as the square root of the corresponding quantity.
\end{remark}
It now remains only to handle the term involving $\phi_2$.

\noindent\textbf{\emph{Step 3: Reexpress \eqref{HN-phi} in terms of $\bar{\rho}_N$.}} As noted in \cite{JW2018}, information about  $\rho_N$ is limited. We therefore translate the calculation in \eqref{HN-phi} involving $\rho_N$ into one for $\bar{\rho}_N$ via 
\begin{lemma}\label{Lem-FN-bar-FN}
	\cite[Lemma 1]{JW2018}
	For any two parbability densities $\rho_N$, $\bar{\rho}_N$ on $\mathbb{T}^{dN}$ and $\Phi\in L^\infty ( \mathbb{T}^{dN}  )$, the following holds for all $\eta>0$
	\begin{equation*}
		\int_{\mathbb{T}^{dN}} \Phi \rho_N \d X\leq \frac{1}{\eta}\Big(H_N(\rho_N|\bar{\rho}_N) +\frac{1}{N}\log \int_{\mathbb{T}^{dN}} \bar{\rho}_Ne^{N\eta\Phi} \d X\Big)\,.
	\end{equation*}
\end{lemma}
\noindent It is easy to verify that $\phi_2\in L^\infty$. Applying Lemma \ref{Lem-FN-bar-FN} to \eqref{HN-phi}, we obtain 
\begin{equation}\label{HN-eta}
	\frac{ \d H_N(t) }{\d t} \leq  C M_K ( H_N(t) + \frac{1}{N}) + \frac{1}{\eta} H_N(t) + \frac{1}{N\eta} \log \Psi \,,
\end{equation}
with
\begin{equation*}
	\begin{aligned}
		\Psi = \int_{\mathbb{T}^{dN}} \bar{\rho}_N \mathrm{exp} \Big( \frac{1}{N^2} \sum^N_{ i,j,k=1 } \eta \phi_2(x_i,x_j,x_k) \Big) \d X\,.
	\end{aligned}
\end{equation*}

\noindent\textbf{\emph{Step 4: Bounding  $\Psi$.}} 
This requires us to establish a new exponential law of large number similar to Theorem 4 in \cite{JW2018}, as follows.

\begin{theorem}\label{Thm-LDE}
	Let $\bar{\rho} \in L^1( \mathbb{T}^d )$  satisfy $\bar{\rho} \geq 0$ and $\int \bar{\rho} \d x = 1 $. For any $\phi (x,z,z') \in L^\infty$ with 
	$$ \sup_{ p \geq 1 } \frac{ \|\sup_{z,z'} | \phi (\cdot,z,z') | \|_{ L^p(\bar{\rho} \d x) } }{p} < \frac{1}{6e^2} \,.$$	
 Assume that $\phi$ satisfies the cancellations \begin{equation}\label{Cancel-rules}
		\int_{\mathbb{T}^d}\phi (x,z,z') \bar{\rho}(x)\d x = 0 \ \forall z,z' \textrm{ and } \int_{\mathbb{T}^{2d}} \phi(x,z,z') \bar{\rho}(z) \bar{\rho}(z') \d z \d z' = 0 \ \forall x \,.
			\end{equation}
			 Then 
	\begin{equation*}
		\int_{\mathbb{T}^{dN} } \bar{\rho}_N \exp \Big( \frac{1}{N^2} \sum^N_{ i,j,k=1 } \phi(x_i,x_j,x_k)  \Big) \d X \leq C = 2 \Big( 1 + \frac{ 4 \alpha }{ ( 1 - \alpha )^3 } + \frac{1}{ 1 - \beta } \Big) < \infty \,,
	\end{equation*}
where 
\begin{equation}\label{alpha-beta}
	\begin{aligned}
		\alpha = & \Big( \sqrt{32e^3} \sup_{ p \geq 1 } \frac{ \| \sup_{z,z'} | \phi (\cdot,z,z') | \|_{ L^p( \bar{\rho} \d x) } }{p} \Big)^2 < 1 \,,\\
		\beta = & \Big( 3e^2 \sup_{ p \geq 1 } \frac{ \| \sup_{z,z'} | \phi (\cdot,z,z') | \|_{ L^p(\bar{\rho}\d x) } }{p} \Big)^2 < 1 \,.
	\end{aligned}
\end{equation}
\end{theorem}
\noindent The proof of Theorem \ref{Thm-LDE} is given in Section \ref{sec-Thm-LDE}. we now define $ \tilde{\phi}_2 = \eta \phi_2$ and note that 
\begin{equation*}
	\begin{aligned}
		|\phi_2 | \leq & 8d \| \sigma \|^2_{ W^{2,\infty} } + 8 \| \sigma \|^2_{ W^{2,\infty} } \sum^d_{ \alpha=1 } \frac{ |\partial_{ x^\alpha_i } \bar{\rho}(x_i) | }{ \bar{\rho}(x_i) } + 2 \| \sigma \|^2_{ W^{2,\infty} } \sum^d_{ \alpha=1 } \frac{ |\partial^2_{ x^\alpha_i } \bar{\rho}(x_i) | }{ \bar{\rho}(x_i) } \\
		\leq &  8d \| \sigma \|^2_{ W^{2,\infty} } + 8d \|\sigma\|^2_{ W^{2,\infty} } \frac{ \| \nabla\bar{\rho} \|_{L^\infty} }{ \inf \bar{\rho} } + 2d \|\sigma\|^2_{ W^{2,\infty} } \frac{ \| \nabla^2 \bar{\rho} \|_{L^\infty} }{ \bar{\inf \bar\rho} } \,.
	\end{aligned}
\end{equation*}
Hence
\begin{equation*}
	\begin{aligned}
		\| \sup_{z,z'} | \phi_2(\cdot,z,z') |\|_{ L^p( \bar{\rho} \d x) }
		\leq & 8d \|\sigma\|^2_{ W^{2,\infty} } + 8d \|\sigma\|^2_{ W^{2,\infty} } \frac{ \| \nabla \bar{\rho} \|_{L^\infty} }{ \inf \bar{\rho} } + 2d \|\sigma\|^2_{ W^{2,\infty} } \frac{ \| \nabla^2 \bar{\rho} \|_{L^\infty} }{ \bar{\inf \bar\rho} } \,.
	\end{aligned}
\end{equation*}
If we choose
$$\eta = \frac{1}{ 12 e^2 \Big( 8d \|\sigma\|^2_{ W^{2,\infty} } + 8d \|\sigma\|^2_{ W^{2,\infty} } \frac{ \| \nabla \bar{\rho} \|_{L^\infty} }{ \inf \bar{\rho} } + 2d \|\sigma\|^2_{ W^{2,\infty} } \frac{ \| \nabla^2 \bar{\rho} \|_{L^\infty} }{ \bar{ \inf\bar\rho } } \Big) } \,,$$
then $\sup_{ p \geq 1 }\frac{ \|\sup_{z,z'} | \tilde{\phi}_2 (\cdot,z,z') | \|_{ L^p( \bar{\rho} \d x) } }{p} < \frac{1}{ 6e^2 } $. Recall that
\begin{equation*}
	\begin{aligned}
		\phi_2(x_i,x_j,x_k) = & \sum^d_{ \alpha=1 } \partial^2_{ x^\alpha_i } \big(  \sigma_{\alpha\alpha}( x_i - x_j ) \sigma_{\alpha\alpha}( x_i - x_k ) - ( \sigma_{\alpha\alpha}*\bar{\rho}(x_i) )^2  \big) \\
		& + 2 \sum^d_{ \alpha=1 } \partial_{ x^\alpha_i } \big(  \sigma_{\alpha\alpha}( x_i - x_j ) \sigma_{\alpha\alpha}( x_i - x_k ) - ( \sigma_{\alpha\alpha}*\bar{\rho}(x_i) )^2 \big) \partial_{ x^\alpha_i } \log \bar{\rho}(x_i) \\
		& + \sum^d_{ \alpha=1 } \big( \sigma_{\alpha\alpha}( x_i - x_j ) \sigma_{\alpha\alpha}( x_i - x_k ) - ( \sigma_{\alpha\alpha}*\bar{\rho}(x_i) )^2 \big) \frac{ \partial^2_{ x^\alpha_i } \bar{\rho}(x_i) }{ \bar{\rho}(x_i) }\,.
	\end{aligned}
\end{equation*}
Direct computation shows that
\begin{equation*}
		\int_{\mathbb{T}^{2d}} \tilde{\phi}_2 \bar{\rho}(z)\bar{\rho}(z') \d z \d z' = \eta \int_{\mathbb{T}^{2d}} \phi_2(x,z,z') \bar{\rho}(z) \bar{\rho}(z') \d z \d z' = 0 \,,
\end{equation*}
while integration by parts gives 
\begin{equation*}
	\int_{\mathbb{T}^d} \tilde{\phi}_2 \bar{\rho}(x) \d x = \eta \int_{\mathbb{T}^d} \phi_2(x,z,z') \bar{\rho}(x) \d x = 0 \,.
\end{equation*}
Hence Theorem \ref{Thm-LDE} gives that
\begin{equation}\label{Bound-Phi}
	\begin{aligned}
		\Psi = \int_{\mathbb{T}^{dN}} \bar{\rho}_N \exp \Big( { \frac{1}{N} \sum_{ i,j,k=1 } \eta \phi_2(x,z,z') } \Big) \d X \leq C = 2 \Big( 1 + \frac{ 4\alpha }{ ( 1 - \alpha )^3 } + \frac{1}{ 1 - \beta } \Big) < \infty \,.
	\end{aligned}
\end{equation}

\noindent\textbf{\emph{Final step: End of the proof.}} From \eqref{HN-eta}-\eqref{Bound-Phi}, it follows that
\begin{equation*}
	\begin{aligned}
	\frac{ \d H_N(t) }{ \d t } \leq & C M_K ( H_N(t) + \frac{1}{N} ) + \frac{1}{\eta} H_N(t) + \frac{1}{N\eta} \log \Psi \\
	\leq & C \Big( M_K + \frac{1}{\eta} \Big) \Big( H_N(t) + \frac{1}{N} \Big) \leq C M \Big( H_N(t) + \frac{1}{N} \Big)\,,
	\end{aligned}
\end{equation*}
where
\begin{equation*}
	\begin{aligned}
		M = & ( \|K\|_{L^\infty} + \| \div K \|_{ \dot{W}^{-1,\infty} } ) \frac{ \|\nabla \bar{\rho} \|_{L^\infty} }{ \inf \bar{\rho} } + \frac{d}{ \underline{\sigma}^2 } \| \div K\|^2_{ \dot{W}^{-1,\infty} } \\
		& + 12e^2 \Big( 8d \|\sigma\|^2_{ W^{2,\infty} } + 8d \|\sigma\|^2_{ W^{2,\infty} } \frac{ \|\nabla \bar{\rho} \|_{L^\infty} }{ \inf \bar{\rho} } + 2 d \|\sigma\|^2_{ W^{2,\infty} } \frac{ \|\nabla^2\bar{\rho} \|_{L^\infty} }{ \bar{ \inf \bar\rho } } \Big) \,.
	\end{aligned}
\end{equation*}
Applying Gronwall lemma yields
\begin{equation*}
	H_N(t) \leq e^{CMt} \Big( H_N(0) + \frac{1}{N} \Big) \,,
\end{equation*}
which completes the proof of Theorem \ref{Thm-Rel-Ent-Est}.
\section{ Proof of Theorem \ref{Thm-LDE} }\label{sec-Thm-LDE} 
This section provides the proof of the exponential law of large number needed for Theorem \ref{Thm-Rel-Ent-Est}, i.e. the proof of Theorem \ref{Thm-LDE}. More precisely, we will prove that
\begin{equation*}
	\int_{ \mathbb{T}^{dN} } \bar{\rho}_N \exp \Big( \frac{1}{N^2} \sum^N_{ i,j,k=1 } \phi_2(x_i,x_j,x_k) \Big)\d X \leq C < \infty \,.
\end{equation*}
Using the series expansion 
$e^x \leq e^x + e^{-x} = 2 \sum^\infty_{ m=0 } \frac{ x^{2m} }{ (2m)! }$, we deduce that
\begin{equation}\label{Goal-integral}
	\begin{aligned}
		&\int_{ \mathbb{T}^{dN} } \bar{\rho}_N \exp \Big( { \frac{1}{N^2} \sum_{ i,j,k=1 } \phi_2(x_i,x_j,x_k) } \Big) \leq \sum^\infty_{ m=0 } \frac{2}{ (2m)! } \int_{ \mathbb{T}^{dN} } \bar{\rho}_N \Big(  \frac{1}{N^2} \sum_{ i,j,k=1 } \phi_2(x_i,x_j,x_k) \Big)^{2m}  \\
		& =\sum^\infty_{ m=0 } \frac{2}{ (2m)! } \frac{1}{N^{4m}} \sum^N_{ \substack{ i_1,j_1,k_1,\cdots,\\i_{2m},j_{2m},k_{2m}=1} } \int_{ \mathbb{T}^{dN} } \phi_2(x_{i_1},x_{j_1},x_{k_1}) \cdots \phi_2(x_{i_{2m}},x_{j_{2m}},x_{k_{2m}}) \bar{\rho}_N := \sum^\infty_{ m=0 } r_m \,.
	\end{aligned}
\end{equation}
Clearly, it suffices to prove the convergence of the series $ \sum^\infty_{ m=0 } r_m $. We split the series into two cases. The case $4m>N$ can be controlled directly by standard combinatorial results. The case $4\leq 4m \leq N$ requires a more delicate combinatorial analysis to reveal that the cancellation rule \eqref{Cancel-rules} implies
\begin{equation*}
\int_{ \mathbb{T}^{dN} } \phi_2(x_{i_1},x_{j_1},x_{k_1}) \cdots \phi_2(x_{i_{2m}},x_{j_{2m}},x_{k_{2m}}) \bar{\rho}_N \d X = 0
\end{equation*}
for many choices of multi-indices $(i_1,j_1,k_1,\cdots,i_{2m},j_{2m},k_{2m})$, thereby bringing this case under control as well. The results for both cases are as follows.
\begin{proposition}\label{Pro-Large-m}
	If $4m>N$, then
	\begin{equation*}
	    \begin{aligned}
     \frac{1}{ (2m)! } \int_{ \mathbb{T}^{dN} } \bar{\rho}_N  \Big(  \frac{1}{N^2} \sum^N_{i,j,k=1} \phi_2(x_i,x_j,x_k) \Big)^{2m} 
     \leq \Big(  3e^2 \sup_{ p \geq 1 } \frac{ \|\sup_{z,z'} | \phi_2(\cdot,z,z') | \|_{L^p( \bar{\rho} \d x) } }{p} \Big)^{2m} \,.
         \end{aligned}
	\end{equation*}
\end{proposition}
\begin{proposition}\label{Pro-Small-m}
	If $ 4 \leq 4m \leq N$, then
	\begin{equation*}
	    \begin{aligned}
	        \frac{1}{ (2m)! } \int_{ \mathbb{T}^{dN} } \bar{\rho}_N \Big(  \frac{1}{N^2} \sum^N_{ i,j,k=1 } \phi_2(x_i,x_j,x_k) \Big)^{2m} \leq 2 m^2 \Big( \sqrt{32e^3} \sup_{ p \geq 1 } \frac{ \|\sup_{z,z'} | \phi_2(\cdot,z,z') | \|_{ L^p( \bar{\rho} \d x) } }{p} \Big)^{2m}\,.
  \end{aligned}
\end{equation*}
\end{proposition}
\noindent By Propositions \ref{Pro-Large-m} and \ref{Pro-Small-m}, we have that
\begin{equation*}
	\begin{aligned}
		\int_{ \mathbb{T}^{dN} } & \bar{\rho}_N \exp \Big( { \frac{1}{N^2} \sum^N_{ i,j,k=1 } \phi_2(x_i,x_j,x_k) } \Big) \d X \\
		\leq & 2 \sum^{ \left\lfloor \frac{N}{4} \right\rfloor}_{ m=0 } \frac{1}{ (2m)! } \int_{ \mathbb{T}^{dN} } \bar{\rho}_N \Big(  \frac{1}{N^2} \sum^N_{ i,j,k=1 } \phi_2(x_i,x_j,x_k)  \Big)^{2m} \d X\\
		& + 2 \sum^\infty_{ m = \left\lfloor \frac{N}{4} \right\rfloor + 1 } \frac{1}{ (2m)! } \int_{ \mathbb{T}^{dN} } \bar{\rho}_N \Big(  \frac{1}{N^2} \sum^N_{ i,j,k=1 } \phi_2(x_i,x_j,x_k) \Big)^{2m} \d X \\
		\leq & 2 + 4 \sum^{ \left\lfloor \frac{N}{4} \right\rfloor }_{ m=1 } m^2 \Big( \sqrt{32e^3} \sup_{ p \geq 1 } \frac{ \|\sup_{z,z'} | \phi_2(\cdot,z,z') | \|_{ L^p( \bar{\rho} \d x ) } }{p}   \Big)^{2m} \\
		& + 2 \sum^\infty_{ m = \left\lfloor \frac{N}{4} \right\rfloor + 1 } \Big(  3e^2 \sup_{ p \geq 1 } \frac{ \|\sup_{z,z'} | \phi_2(\cdot,z,z') | \|_{ L^p( \bar{\rho} \d x ) } }{p}   \Big)^{2m}\,.
	\end{aligned}
\end{equation*}
Recalling \eqref{alpha-beta}, we obtain 
\begin{equation*}
	\begin{aligned}
		\sum^{ \left\lfloor \frac{N}{4} \right\rfloor }_{ m=1 } & m^2 \Big(   \sqrt{32e^3} \sup_{ p \geq 1 } \frac{ \|\sup_{z,z'} | \phi_2(\cdot,z,z') | \|_{ L^p( \bar{\rho} \d x ) } }{p} \Big)^{2m}
		\leq \sum^\infty_{ m=1 } m(m+1)\alpha^m \\
		= & \alpha \sum^\infty_{ m=1 } m(m+1) \alpha^{m-1} = \alpha \frac{\d^2}{ \d \alpha^2 } \Big( \sum^\infty_{ m=0 } \alpha^m \Big) = \alpha \Big( \frac{1}{ 1 - \alpha } \Big)^{''} = \frac{ 2\alpha }{ ( 1 - \alpha )^3 } < \infty
	\end{aligned}
\end{equation*}
and 
\begin{equation*}
	\begin{aligned}
	\sum^\infty_{ m = \left\lfloor \frac{N}{4} \right\rfloor + 1 } \Big( 3e^2 \sup_{ p \geq 1 } \frac{ \| \sup_{z,z'} | \phi_2(\cdot,z,z') | \|_{L^p( \bar{\rho} \d x ) } }{p} \Big)^{2m}
	\leq \sum^\infty_{ m=0 } \beta^m = \frac{1}{ 1 - \beta } < \infty \,.
	\end{aligned}
\end{equation*}
Therefore
	\begin{equation*}
		\int_{ \mathbb{T}^{dN} } \bar{\rho}_N \exp \Big( { \frac{1}{N^2} \sum^N_{ i,j,k=1 } \phi_2(x_i,x_j,x_k) } \Big) \d X \leq 2 + \frac{ 8 \alpha }{ ( 1 - \alpha )^3 } + \frac{2}{ 1 - \beta } < \infty \,.
	\end{equation*}
This completes the proof of Theorem \ref{Thm-LDE}. 
\vspace{2mm}

We next introduce some notations and then give the proofs of Propositions \ref{Pro-Large-m} and \ref{Pro-Small-m}.

Let $ I_{2m} = (i_1,i_2,...,i_{2m}) $, $ J_{2m} = (j_1,j_2,...,j_{2m}) $, $ K_{2m} = (k_1,k_2,...,k_{2m}) $ be the $i$-, $j$- and $k$-indices, where $i_l,j_l,k_l \in \{ 1,2,3,...,N \}$ for $1 \leq l \leq 2m$. Let $ A_N = (a_1,...,a_N) $ be the multiplicities of $I_{2m}$, where 
$$a_t = | \{l: 1\leq l \leq 2m, i_l=t \} | \,, \ \ t=1,2,3,...,N \,.$$

\subsection{Proof of Proposition \ref{Pro-Large-m}}
Recalling \eqref{Goal-integral}, we know
\begin{equation*}
	\begin{aligned}
		\frac{1}{ (2m)! } & \int_{ \mathbb{T}^{dN} } \bar{\rho}_N \Big( \frac{1}{N^2} \sum^N_{ i,j,k=1 } \phi_2(x_i,x_j,x_k) \Big)^{2m} \d X \\
		& \leq \frac{1}{(2m)!} \frac{1}{ N^{4m} } \sum^N_{ \substack{i_1,j_1,k_1,\cdots,\\i_{2m},j_{2m},k_{2m}=1} } \int_{ \mathbb{T}^{dN} } \phi_2(x_{i_1},x_{j_1},x_{k_1}) \cdots \phi_2(x_{i_{2m}},x_{j_{2m}},x_{k_{2m}}) \bar{\rho}_N \d X \,.
	\end{aligned}
\end{equation*}
For a fixed triple of multi-indices $(I_{2m},J_{2m},K_{2m})$, the above integral term satisfies 
\begin{equation*}
	\begin{aligned}
		\int_{ \mathbb{T}^{dN} } & \phi_2(x_{i_1},x_{j_1},x_{k_1})  \cdots \phi_2(x_{i_{2m}},x_{j_{2m}},x_{k_{2m}}) \bar{\rho}_N \d X\\
		& \leq \int_{ \mathbb{T}^{dN} } \sup_{z,z'} | \phi_2(x_{i_1},z,z') | \cdots \sup_{z,z'} | \phi_2(x_{i_{2m}},z,z') | \bar{\rho}_N \d X \,.
	\end{aligned}
\end{equation*}
Hence
\begin{equation*}
	\begin{aligned}
		& \frac{1}{(2m)!} \int_{ \mathbb{T}^{dN} } \bar{\rho}_N \Big(  \frac{1}{N^2} \sum^N_{i,j,k=1} \phi_2(x_i,x_j,x_k) \Big)^{2m} \d X \\
		& \leq \frac{1}{(2m)!} \frac{1}{N^{4m}} \sum^N_{ \substack{i_1,j_1,k_1,\cdots,\\i_{2m},j_{2m},k_{2m}=1} } \int_{ \mathbb{T}^{dN} } \sup_{z,z'} | \phi_2(x_{i_1},z,z') | \cdots \sup_{z,z'} | \phi_2(x_{i_{2m}},z,z') | \bar{\rho}_N \d X \\
		& = \frac{1}{(2m)!} \sum^N_{ i_1,\cdots, i_{2m}=1 } \int_{ \mathbb{T}^{dN} } \sup_{z,z'} | \phi_2(x_{i_1},z,z') | \cdots \sup_{z,z'} | \phi_2(x_{i_{2m}},z,z') | \bar{\rho}_N \d X \\
		& = \frac{1}{(2m)!} \int_{ \mathbb{T}^{dN} } \Big( \sum^N_{i=1} \sup_{z,z'} | \phi_2(x_i,z,z') | \Big)^{2m} \bar{\rho}_N \d X \,.
	\end{aligned}
\end{equation*}
By the multinomial theorem \cite[Theorem 5.4.1]{Bru2010}
\begin{equation*}
	\begin{aligned}
		\int_{ \mathbb{T}^{dN} }& \Big( \sum^N_{i=1} \sup_{z,z'} | \phi_2(x_i,z,z') | \Big)^{2m} \bar{\rho}_N \\
		& = \sum_{ \substack{ a_1 + \cdots + a_N = 2m \\ a_1 \geq 0,\cdots,a_N} } \frac{ (2m)! }{ (a_1)!\cdots(a_N)! } \int_{ \mathbb{T}^{dN} } \big( \sup_{z,z'} | \phi_2(x_{i_1},z,z') | \big)^{a_1} \cdots \big( \sup_{z,z'} | \phi_2(x_{i_{2m}},z,z') | \big)^{a_N} \bar{\rho}_N \\
		& := \sum_{ \substack{ a_1 + \cdots + a_N = 2m \\ a_1 \geq 0,\cdots,a_N} } \frac{ (2m)! }{ (a_1)!\cdots(a_N)! } M^{a_1}_{a_1} \cdots M^{a_N}_{a_N} \,,
	\end{aligned}
\end{equation*}
where 
\begin{equation*}
	M_{a_i} = \Big( \int_{\mathbb{T}^d}\sup_{z,z'} | \phi_2(x,z,z') |^{a_i} \bar{\rho}(x) \d x \Big)^{ \frac{1}{a_i} } \,.
\end{equation*}
Using the bound $
M_{a_i} \leq a_i \sup_{ p \geq 1 } \Big( \frac{M_p}{p} \Big)$ and the inequality $ n^n \leq e^n n! $, we deduce that
\begin{equation*}
	M^{a_i}_{a_i} \leq e^{a_i}(a_i)! \Big( \sup_{ p \geq 1 } \frac{M_p}{p} \Big)^{a_i} \,.
\end{equation*}
Hence
\begin{equation*}
	\begin{aligned}
		& \frac{1}{ (2m)! } \int_{ \mathbb{T}^{dN} } \bar{\rho}_N \Big( \frac{1}{N^2} \sum^N_{ i,j,k=1 } \phi_2(x_i,x_j,x_k) \Big)^{2m} \d X \\
		& \leq \frac{1}{ (2m)! } \sum_{ \substack{ a_1 + \cdots + a_N = 2m\\ a_1 \geq 0,\cdots,a_N\geq 0}}\frac{ (2m)! }{(a_1)!\cdots(a_N)!} e^{2m}(a_1)! \cdots (a_N)! \Big( \sup_{ p \geq 1 } \frac{M_p}{p} \Big)^{2m} \\
		& = e^{2m} \Big( \sup_{ p \geq 1} \frac{M_p}{p} \Big)^{2m} \sum_{ \substack{ a_1 + \cdots + a_N = 2m \\ a_1 \geq 0,\cdots,a_N\geq 0} } 1 \,.
	\end{aligned}
\end{equation*}
For $ \sum_{ \substack{ a_1 + \cdots + a_N = 2m \\ a_1 \geq 0, \cdots, a_N \geq 0 } } 1$, its value equals the number of nonnegative integer solutions to the equation $ a_1 + \cdots + a_N = 2m $, i.e.
\begin{equation*}
	\sum_{ \substack{ a_1 + \cdots + a_N = 2m \\ a_1 \geq 0, \cdots, a_N \geq 0 } } 1 = | \{ (a_1,\cdots,a_N) | a_1 + \cdots + a_N = 2m, a_i \geq 0 \textrm{ for } 1 \leq i \leq N \} |\,.
\end{equation*}
From the classical combinatorics result \cite[Theorem 2.5.1]{Bru2010} and its proof, we obtain
\begin{equation}\label{comb-result-1}
		\sum_{ \substack{ a_1 + \cdots + a_N = 2m \\ a_1 \geq 0, \cdots, a_N \geq 0 } } 1 = \binom{ 2m + N-1 }{N-1} \,. 
\end{equation}
By Stirling's formula 
\begin{equation*}
	m! = \lambda_m \sqrt{ 2 \pi m} \left( \frac{m}{e} \right)^m
\end{equation*}
with $ 1 \leq \lambda_m \leq \frac{11}{10} $ and $ \lambda_m \to 1 $ as $n \to \infty$, we know that if $2a \geq b$, then
\begin{equation*}
	\begin{aligned}
		\binom{a+b}{b} =&\frac{ \lambda_{a+b} \sqrt{ 2 \pi(a+b) } \big( \frac{a+b}{e} \big)^{a+b} }{ \lambda_a \sqrt{ 2 \pi a} \big( \frac{a}{e} \big)^{a} \lambda_b \sqrt{ 2 \pi b } \big( \frac{b}{e} \big)^{b} } = \frac{ \lambda_{a+b} }{ \lambda_a \lambda_b } \frac{1}{ \sqrt{2\pi} }\sqrt{ \frac{a+b}{ab} }\frac{ (a+b)^{a+b} }{ a^ab^b } \\
		\leq & \big( 1 + \frac{b}{a} \big)^a \big( 1 + \frac{a}{b} \big)^b \leq 3^a \cdot \big( 1 + \frac{1}{b/a} \big)^{ \frac{b}{a} \cdot \frac{a}{b} \cdot b} \leq (3e)^a \,,
	\end{aligned}
\end{equation*}
where the last inequality follows from $ (1+1/n)^n < e $ for all $ n > 0 $. Since $ 4m > N $, applying the above inequality to \eqref{comb-result-1} yields
\begin{equation*}
	\binom{ 2m+N-1 }{ N-1 } \leq (3e)^{2m} \,.
\end{equation*}
Therefore
\begin{equation*}
	\begin{aligned}
		&\frac{1}{ (2m)! } \int_{ \mathbb{T}^{dN} } \bar{\rho}_N \Big( \frac{1}{N^2} \sum^N_{ i,j,k=1 } \phi_2(x_i,x_j,x_k) \Big)^{2m} \d X \leq \Big( 3e^2 \sup_{ p \geq 1} \frac{M_p}{p} \Big)^{2m} \,.
	\end{aligned}
\end{equation*}

\subsection{Proof of Proposition \ref{Pro-Small-m}.} We again work with the expansion
\begin{equation*}
	\begin{aligned}
		\frac{1}{(2m)!} & \int_{ \mathbb{T}^{dN} } \bar{\rho}_N \Big(  \frac{1}{N^2} \sum^N_{ i,j,k=1 } \phi_2(x_i,x_j,x_k) \Big)^{2m} \d X\\
		& \leq \frac{1}{(2m)!} \frac{1}{N^{4m}} \sum^N_{ \substack{ i_1,j_1,k_1,\cdots,\\i_{2m},j_{2m},k_{2m}=1 } } \int_{ \mathbb{T}^{dN} } \phi_2(x_{i_1},x_{j_1},x_{k_1}) \cdots  \phi_2( x_{i_{2m}},x_{j_{2m}},x_{k_{2m}} ) \bar{\rho}_N \d X \,.
	\end{aligned}
\end{equation*}
As previously noted, the key is using the cancellation rule \eqref{Cancel-rules} to count non-vanishing integral terms in the $6m$-fold sum. For easier application, We restate the rule below.
\begin{lemma}\label{Lem-Condition-IJK}
If a triple of multi-indices $(I_{2m}, J_{2m}, K_{2m})$ satisfies neither
	\begin{enumerate}
	\item there exists $l$ such that $i_l \notin \{ i_1,i_2,...,i_{l-1},i_{l+1},...,i_{2m},j_1,\cdots,j_{2m},k_1,\cdots,k_{2m} \}$ nor
	\item there exists $l$ such that $j_l \neq k_l$ and $j_l,k_l \notin \{i_1,i_2,...,i_{2m}\} \cup \{ j_1,...,j_{l-1},j_{l+1},...,j_{2m}\} \cup \{k_1,...,k_{l-1},k_{l+1},...,k_{2m} \}$, then
\end{enumerate}
\begin{equation}\label{Non-0-Integral}
	\int_{ \mathbb{T}^{dN} } \phi_2(x_{i_1},x_{j_1},x_{k_1}) \cdots \phi_2(x_{i_{2m}},x_{j_{2m}},x_{k_{2m}}) \bar{\rho}_N \d X \neq 0 \,.
\end{equation}
\end{lemma}
\begin{proof}
	It suffices to prove that whenever condition (1) or (2) holds,
	$$ 	\int_{ \mathbb{T}^{dN} } \phi_2(x_{i_1},x_{j_1},x_{k_1}) \cdots \phi_2(x_{i_{2m}},x_{j_{2m}},x_{k_{2m}}) \bar{\rho}_N \d X = 0\,. $$
If Condition (1) holds, we may assume $i_l=p$. Then the cancellation rule \eqref{Cancel-rules} yields 
	\begin{equation*}
		\int_{ \mathbb{T}^{d(N-1)} } \Big( \int_{\mathbb{T}^d} \phi_2(x_{i_l},x_{j_l},x_{k_l}) \, \d x_{i_l} \Big)  \prod_{t \neq l} \phi_2(x_{i_t},x_{j_t},x_{k_t})  \prod_{s \neq p}\bar{\rho}(x_s) \d x_s=0\,.
	\end{equation*}
If Condition (2) holds, we may assume  $j_l=p$ and $k_l=q$ with $p\neq q$. Then \eqref{Cancel-rules} yields 
	\begin{equation*}
		\int_{ \mathbb{T}^{d(N-2)} } \Big( \int_{\mathbb{T}^{2d}} \phi_2(x_{i_l},x_{j_l},x_{k_l}) \, \d x_{j_l} \d x_{k_l} \Big)  \prod_{t \neq l} \phi_2(x_{i_t},x_{j_t},x_{k_t})  \prod_{s \neq p,q}\bar{\rho}(x_s) \d x_s=0\,.
	\end{equation*}
\end{proof}

\vspace{2mm}

Let $\mathcal{T}_{all}$ denote the set of  all multi-indices $(I_{2m},J_{2m},K_{2m})$ for which \eqref{Non-0-Integral} holds. We next calculate $| \mathcal{T}_{all} |$
systematically using the basic counting principles in combinatorics.
\vspace{2mm}

\textbf{\emph{Step 1: Define the classification criteria for $I_{2m}$.}}
Let $S_{I_{2m}}$ denote the set of components of the multi-index $I_{2m}$, i.e.
$$S_{I_{2m}} = \{ i_1,i_2,\cdots,i_{2m} \}\,.$$
Then the cardinality of $S_{I_{2m}}(i.e. |S_{I_{2m}}|)$ equals the number of distinct integers in $I_{2m}$. From (1) in Lemma \ref{Lem-Condition-IJK}, the multiplicity of each component of $I_{2m}$ cannot be one. Then $$1 \leq s := |S_{I_{2m}}| \leq m \,.$$ Since the choices of $I_{2m}$ are disjoint for different $s$, we can partition all choices of $I_{2m}$ by $s$.

\textbf{\emph{Step 2: Count the total number of possible $I_{2m}$ for fixed $s$.}} 

\textbf{\emph{Substep 2.1: Choose $s$ distinct integers.}} This is equivalent to an unordered selection of $s$ elements from the set $\{ 1,2,...,N\}$, which has $\binom{N}{s}$ choices.
 
{\textbf{\emph{Substep 2.2: Count the total choices of $I_{2m}$ for fixed $s$ distinct integers.}} Without loss of generality, let $S_{I_{2m}}=\{1,2,\cdots,s\}$, then the multiplicities $A_N$ of $I_{2m}$ satisfies
 $$\begin{cases}
 	a_t \geq 2 , \ 1 \leq t \leq s,\\
 	a_t = 0,\ s < t \leq N, \\
 	a_1+\cdots+a_s = 2m.
 \end{cases}$$
  From \cite[Theorem 2.4.2]{Bru2010}, the total number of choices for any fixed $A_N = (a_1,\cdots,a_s,0,\cdots,0)$ is equivalent to the number of permutations of the multiset $\{ a_1\cdot 1,a_2 \cdot 2,\cdots,a_s \cdot s,0 \cdot (s+1),\cdots,0 \cdot N \}$, i.e.
 $\frac{ (2m)! }{ (a_1)!\cdots(a_s)! }$.
Hence, the total number of choices  equals
 $$\sum_{ \substack{ a_1+\cdots+a_s = 2m,\\ a_1 \geq 2,\cdots,a_s \geq 2 } }\frac{(2m)!}{ (a_1)!\cdots(a_s)! }\,.$$

\textbf{\emph{Step 3: Count the total number of possible $(J_{2m},K_{2m})$ for any $I_{2m}$ with $S_{I_{2m}}=\{1,\cdots,s\}$.}}  Let $\mathcal{F}^s$ denote the set of these $(J_{2m},K_{2m})$. We first determine a classification of $\mathcal{F}^s$.  For each class, by the multiplication principle, we determine $(j_1,k_1),\cdots,(j_{2m},k_{2m})$ in order, compute the number of possible $(j_l,k_l)$ at each step, and multiply them to obtain the total number of possible $(J_{2m},K_{2m})$ in that class. In the following, we determine the classification of $\mathcal{F}^s$ by analyzing the number of choices for each pair $(j_l.k_l)$.

\textbf{\emph{Substep 3.1: Define the classification criteria for $\mathcal{F}^s$.}}
Let $S_r=\{ j_1,k_1,\cdots, j_r,k_r\}$ be the set of components in the first $r$ pairs of positions of  $(J_{2m},K_{2m})$, where $0\leq r\leq 2m$. Since we choose each pair $(j_l,k_l)$ in order from $l=1$ to $2m$, Lemma \ref{Lem-Condition-IJK} (2) implies that for every such pair we must have either $j_l=k_l$ or $j_l\neq k_l$ with at least one of them already in set $S_{I_{2m}} \cup S_{l-1}$. Thus the choice of $(j_l,k_l)$ adds at most one element outside of $S_{I_{2m}} \cup S_{l-1}$ to $S_l$, i.e. $0 \leq | S_l-S_{I_{2m}} \cup S_{l-1} | \leq 1$. Accordingly, the number of possible $(j_l,k_l)$ falls into three cases:
\begin{itemize}
	\item If both $j_l$ and $k_l$ are in set $S_{I_{2m}} \cup S_{l-1}$, then there are $|S_{I_{2m}} \cup S_{l-1}|^2$  choices;
	\item if $j_l$ is in $S_{I_{2m}} \cup S_{l-1}$ but $k_l$ is not, there are $|S_{I_{2m}} \cup S_{l-1} | | ( N- |S_{ I_{2m} } \cup S_{l-1} |)$ choices; 
	\item if $j_l$ is not in $S_{I_{2m}} \cup S_{l-1}$ but $k_l$ is, there are $(N- | S_{I_{2m}} \cup S_{l-1} | )( | S_{I_{2m}} \cup S_{l-1} | + 1)$ choices.
\end{itemize}
 Hence the number of choices for each pair $(j_l,k_l)$ depends only on $| S_{I_{2m}} \cup S_{l-1}|$. Note that $| S_{I_{2m}} \cup S_{l-1}| = |S_{I_{2m}}| + | S_{l-1} - S_{I_{2m}} | = s + |S_{l-1} - S_{I_{2m}}
| $ and $0 \leq | S_l - S_{I_{2m} } \cup S_{l-1} | \leq 1$ implies $ 0 \leq | S_l - S_{I_{2m}} | \leq l$. This shows that when choosing values for the first $r$ pairs of positions in $(J_{2m},K_{2m})$, the corresponding $S_r$ introduces at most $r$ new elements from $\{1,\cdots, N\}$ that are not in $S_{I_{2m}}$, where $0\leq r\leq 2m$. Note that each $(J_{2m},K_{2m})$ uniquely determines an integer $r$, namely the number of distinct values that lie outside $S_{I_{2m}}$ among the values taken by the components of $(J_{2m},K_{2m})$.  We classify $\mathcal{F}^s$ by $r$, which yields a partition of $\mathcal{F}^s$. 

Let $\mathcal{F}_r=\big\{(J_{2m},K_{2m}): | S_{2m}-S_{I_{2m}} |=r\big\}$ for $r=0,1,\cdots, 2m$. Then
\begin{equation*}
	|\mathcal{F}^s|=\sum^{2m}_{r=0}|\mathcal{F}_r|\,.
\end{equation*}

\textbf{\emph{Substep 3.2: Deriving an upper bound for  $|\mathcal{F}^s|$.}} We first analyze $|\mathcal{F}_r|$ for a fixed $r$. For $r=0$, Substep 3.1 implies that every component of each pair $(j_l,k_l)$ is chosen from $S_{I_{2m}}$. Hence $|\mathcal{F}_0|=s^{4m}$. For $r\geq 1$, note that these $r$ distinct values must occupy at least $r$ distinct pairs of positions. We derive an upper bound for $|\mathcal{F}_r|$ as follows. 
\begin{itemize}
\item Choose $r$ pairs from the $2m$ pairs, and for each chosen pair, decide whether the element from $\{1,\cdots,N\} \setminus S_{I_{2m}}$ is placed in the left or right component. This gives $\binom{2m}{r}\cdot 2^r$ choices in total. 

\item For each of the $\binom{2m}{r}\cdot 2^r$ choices, we follow the rule from Substep 3.1 to determine the value assignment of each pair $(j_l,k_l)$ in order $l = 1,\cdots, 2m$. We then multiply the numbers of possibilities for all pairs to obtain the total number of possible $(J_{2m},K_{2m})$ for that choice.
\item Let $\mathcal{N}_r$ be the sum over all $\binom{2m}{r}\cdot 2^r$ choices of the number of possible $(J_{2m},K_{2m})$ generated by that choice. Note that different choices for the same $r$ may lead to the same $(J_{2m},K_{2m})$. Hence $\mathcal{N}_r$ is an upper bound for $|\mathcal{F}_r|$. 
\end{itemize}

From the above analysis, even after fixing the $r$ positions that contain elements outside $S_{I_{2m}}$, the number of possibilities for each of the $\binom{2m}{r}\cdot 2^r$ choices is determined by applying the rule in Substep 3.1 to each pair $(j_l,k_l)$. To avoid enumerating the $ \binom{2m}{r} \cdot 2^r$ cases one by one to obtain $\mathcal{N}_r$,  we next derive a uniform upper bound for all cases, denoted by $\mathcal{T}_r$. 

Recall the three cases in Substep 3.1 for the number of possible choices of each pair  $(j_l,k_l)$. The number of choices in each case depends on $| S_{I_{2m}} \cup S_{l-1} |$. Thus we observe: 
\begin{itemize}
	\item The more elements outside $S_{I_{2m}}$ are assigned to $(j_1,k_1),\cdots,(j_{i-1},k_{i-1})$, the larger $| S_{I_{2m}} \cup S_{l-1} |$;

	\item The number of choices for  $(j_l,k_l)$ is maximized when $j_l$ is not in $S_{I_{2m}}\cup S_{l-1}$ but $k_l$ is.
\end{itemize}

Hence, among all $ \binom{2m}{r} \cdot 2^r$ possible cases, the number of choices for each pair $(j_l,k_l)$ that maximizes the total number of choices of $(J_{2m},K_{2m})$ is shown below.

\begin{table}[h!]
	\centering
	\renewcommand{\arraystretch}{1.5}
	\scalebox{0.75}{
\begin{tabular}{c|cc|cc|c|cc|cc|c|cc}
\hline Index &$j_1$ &$k_1$ &$j_2$ &$k_2$ & $ \cdots $ &$j_r$& $k_r$ & $j_{r+1}$ &$k_{r+1}$ & $ \cdots $ & $j_{2m}$ &$k_{2m}$ \\
\hline Number of choices &$N-s$ &$s+1$ &$N-(s+1)$ &$s+2$ & $\cdots $ &$N-(s+r-1)$& $s+r$ & $s+r$ &$s+r$ & $\cdots $ & $s+r$ &$s+r$ \\
\hline
\end{tabular}
}
\end{table}

\noindent From the table, we see that
\begin{equation*}
	\begin{aligned}
		\mathcal{T}_r \leq &(N-s)(s+1)( N - (s+1) )(s+2) \cdots ( N - (s+r-1))(s+r) \\
		&\cdot\underbrace{ (s+r)(s+r) \cdots (s+r)(s+r)}_{4m-2r} \\
		= & \frac{ (N-s)! }{ (N-(s+r))! } (s+1)(s+2) \cdots (s+r)(s+r)^{4m-2r} \,.
	\end{aligned}
\end{equation*}
Note that the above equation also holds for $r=0$. Then we have
\begin{equation*}
|\mathcal{F}^s|=\sum^{2m}_{r=0}|\mathcal{F}_r|\leq \sum^{2m}_{r=0}\mathcal{N}_r\leq \sum^{2m}_{r=0} \binom{2m}{r} 2^r\mathcal{T}_r\,.
\end{equation*}

\textbf{\emph{Step 4: Calculate $|\mathcal{T}_{all}|$.}} Based on the above three steps, we have
\begin{equation*}
	\begin{aligned}
		| \mathcal{T}_{all} |=& \sum^m_{ s=1 } \binom{N}{s} \sum_{ \substack{ a_1+\cdots+a_s = 2m,\\ a_1 \geq 2,\cdots,a_s \geq 2 } } \frac{(2m)!}{ (a_1)!\cdots(a_s)! } \cdot |\mathcal{F}^s|\\
		\leq & \sum^m_{ s=1 } \binom{N}{s} \sum_{ \substack{ a_1+\cdots+a_s = 2m,\\ a_1 \geq 2,\cdots,a_s \geq 2 } } \frac{(2m)!}{ (a_1)!\cdots(a_s)! } \cdot \sum^{2m}_{r=0} \binom{2m}{r} 2^r \mathcal{T}_r \,.
	\end{aligned}
\end{equation*}
Given the bounds $\binom{a}{b} \leq  e^ba^bb^{-b} $, $ \binom{a}{b} \leq 2^a $ for $ 0 \leq b \leq a$, and $ 2^r \leq 2^{2m} $, we deduce that
\begin{equation}\label{Bound-T}
	\begin{aligned}
		| \mathcal{T}_{all} | \leq e^m 2^{4m} \tilde{ \mathcal{T} }\,,
	\end{aligned}
\end{equation}
where
\begin{equation*}
	\tilde{\mathcal{T}} := \sum^m_{ s=1 } \sum^{2m}_{r=0} \sum_{ \substack{ a_1+\cdots+a_s = 2m,\\ a_1 \geq 2,\cdots,a_s \geq 2 } } \frac{(2m)!}{ (a_1)!\cdots(a_s)! } N^ss^{-s} \frac{ (N-s)! }{ (N-(s+r))! } (s+1)(s+2) \cdots (s+r)(s+r)^{4m-2r} \,.
\end{equation*}
Divide $\tilde{ \mathcal{T} }$ into 
\begin{equation*}
	\begin{aligned}
		\tilde{ \mathcal{T} }
		= \sum^m_{s=1} \sum^s_{r=0} \sum_{ \substack{ a_1+\cdots+a_s=2m,\\a_1\geq 2,\cdots,a_s\geq 2 } } \frac{(2m)!}{ (a_1)!\cdots(a_s)! } \tilde{ \mathcal{T} }_1 + \sum^m_{s=1} \sum^{2m}_{r=s+1} \sum_{ \substack{ a_1+\cdots+a_s=2m,\\a_1 \geq 2,\cdots,a_s \geq 2 } } \frac{(2m)!}{ (a_1)!\cdots(a_s)! } \tilde{ \mathcal{T} }_2 \,,
	\end{aligned}
\end{equation*}
where
\begin{equation*}
	\begin{aligned}
		\tilde{ \mathcal{T} }_1 := & N^s s^{-s} \frac{(N-s)!}{ ( N - (s+r) )! } \underbrace{ (s+1) \cdots (s+r) }_r (s+r)^{4m-2r} \,,\\
		\tilde{ \mathcal{T} }_2 := &  N^s s^{-s} \frac{ (N-s)! }{ ( N - (s+r) )! } \underbrace{ (s+1) \cdots (s+r) }_r (s+r)^{4m-2r} \,.
	\end{aligned}
\end{equation*}
Note that $4 \leq 4m \leq N$. When $0 \leq r \leq s \leq m$, one has
\begin{equation*}
	\begin{aligned}
		\tilde{ \mathcal{T} }_1 \leq & N^s s^{-s} N^r (s+r)^r (s+r)^{4m-2r} \\
		= & N^{s+r} (s+r)^{4m-r} s^{-s} \\
		= & N^{s+r} (s+r)^{4m-r-s} \left( \frac{s+r}{s} \right)^s \\
		\leq & N^{s+r} N^{4m-r-s} 2^s \leq N^{4m} 2^m \,.
	\end{aligned}
\end{equation*}
When $s < r \leq 2m$, one has
\begin{equation*}
	\begin{aligned}
		\tilde{ \mathcal{T} }_2 \leq & N^s s^{-s} N^r \underbrace{ (s+1)(s+2) \cdots(s+s) }_s \underbrace{ \cdots
		(s+r) }_{r-s} (s+r)^{4m-2r} \\
		\leq & N^{s+r} s^{-s} (s+s)^s (s+r)^{r-s} (s+r)^{4m-2r} \\
		\leq & N^{s+r} s^{-s} (2s)^s (s+r)^{r-s} (s+r)^{4m-2r} \\
		\leq & N^{s+r} 2^sN^{r-s} N^{4m-2r} \\
		= & N^{ (s+r) + (r-s) + (4m-2r) } 2^s \leq N^{4m} 2^m = N^{4m} 2^m \,.
	\end{aligned}
\end{equation*}
Hence $	\{ \tilde{ \mathcal{T} }_1, \tilde{ \mathcal{T} }_2 \}\leq N^{4m} 2^m $, which implies that
\begin{equation*}
	\tilde{ \mathcal{T} } \leq \sum^m_{s=1} \sum^{2m}_{r=0} \sum_{ \substack{ a_1+\cdots+a_s=2m,\\a_1 \geq 2,\cdots,a_s \geq 2 } } \frac{(2m)!}{ (a_1)!\cdots(a_s)! } N^{4m} 2^m \leq \sum^m_{s=1} \sum_{ \substack{ a_1+\cdots+a_s=2m,\\a_1\geq 2,\cdots,a_s\geq 2  } } \frac{(2m)!}{ (a_1)!\cdots(a_s)! } 2m 2^m N^{4m} \,.
\end{equation*}
Combining this with \eqref{Bound-T} gives
\begin{equation}\label{T-all}
	\begin{aligned}
		| \mathcal{T}_{all} | \leq & e^m 2^{2m} \sum^m_{s=1} \sum_{ \substack{ a_1+\cdots+a_s=2m,\\a_1\geq 2,\cdots,a_s\geq 2 } } \frac{(2m)!}{ (a_1)!\cdots(a_s)! } 2m 2^m N^{4m} \\
		= & 2m (8e)^m N^{4m} \sum^m_{s=1} \sum_{ \substack{ a_1+\cdots+a_s=2m,\\a_1 \geq 2,\cdots,a_s \geq 2 } } \frac{(2m)!}{ (a_1)!\cdots(a_s)! } \,.
	\end{aligned}
\end{equation} 
Return to the sum-integral
$$ \sum^N_{ \substack{ i_1,j_1,k_1,\cdots,\\i_{2m},j_{2m},k_{2m}=1 } } \int_{ \mathbb{T}^{dN} } \phi_2(x_{i_1},x_{j_1},x_{k_1}) \cdots \phi_2( x_{i_{2m}},x_{j_{2m}},x_{k_{2m}} ) \bar{\rho}_N \d X \,.$$
For any fixed multi-indices $I_{2m}$ with multiplicities $A_N = (a_1,...,a_s,0,\cdots,0)$, $J_{2m}$ and $K_{2m}$, the proof of Proposition \ref{Pro-Large-m} shows that 
\begin{equation}\label{Integral-term}
	\begin{aligned}
		\int_{ \mathbb{T}^{dN} } & \phi_2( x_{i_1},x_{j_1},x_{k_1} ) \cdots \phi_2(x_{i_{2m}},x_{j_{2m}},x_{k_{2m}}) \bar{\rho}_N \d X \\
		\leq & \int_{ \mathbb{T}^{dN} } \sup_{z,z'}| \phi_2(x_{i_1},z,z') | \cdots \sup_{z,z'} | \phi_2(x_{i_{2m}},z,z') | \bar{\rho}_N \d X 
		\leq e^{2m} (a_1)! \cdots(a_s)! \Big( \sup_{ p \geq 1 } \frac{M_p}{p} \Big)^{2m} \,.
	\end{aligned}
\end{equation}
where 
\begin{equation*}
	M_{p} = \Big( \int_{ \mathbb{T}^d }\sup_{z,z'} | \phi_2(x,z,z') |^{p} \bar{\rho}(x) \d x \Big)^{ \frac{1}{p} } \,.
\end{equation*}
Thus from \eqref{T-all} and \eqref{Integral-term}, we obtain
\begin{equation*}
	\begin{aligned}
		\frac{1}{ (2m)! } & \int_{ \mathbb{T}^{dN} }  \Big( \frac{1}{N^2}\sum^N_{i,j,k=1} \phi_2(x_i,x_j,x_k) \Big)^{2m}\bar{\rho}_N \d X \\
		\leq & \frac{1}{(2m)!} \frac{1}{ N^{4m} } \sum^N_{ \substack{ i_1,j_1,k_1,\cdots,\\i_{2m},j_{2m},k_{2m}=1} } \int_{ \mathbb{T}^{dN} } \phi_2(x_{i_1},x_{j_1},x_{k_1}) \cdots \phi_2(x_{i_{2m}},x_{j_{2m}},x_{k_{2m}}) \bar{\rho}_N \d X \\
		\leq & \frac{1}{(2m)!} \frac{1}{ N^{4m} } \cdot | \mathcal{T}_{all} | \cdot e^{2m} (a_1)! \cdots(a_s)! \Big( \sup_{ p \geq 1} \frac{M_p}{p} \Big)^{2m} \\
		= & \frac{1}{(2m)!} \frac{1}{ N^{4m} } \cdot 2m (8e)^m N^{4m} \sum^m_{s=1} \sum_{ \substack{ a_1+\cdots+a_s=2m,\\a_1\geq 2,\cdots,a_s \geq 2 } } (2m)! \cdot e^{2m} \Big( \sup_{ p \geq 1} \frac{M_p}{p} \Big)^{2m} \\
		= & 2m (8e^3)^m \Big (\sup_{ p \geq 1 } \frac{M_p}{p} \Big)^{2m} \sum^m_{s=1} \sum_{ \substack{ a_1+\cdots+a_s=2m,\\a_1 \geq 2,\cdots,a_s\geq 2 } }  1 \,.
	\end{aligned}
\end{equation*}
Note that 
\begin{equation*}
	\begin{aligned}
		\sum_{ \substack{ a_1+\cdots+a_s=2m \\ a_1 \geq 2,\cdots,a_N \geq 2} } 1 = &| \{(a_1,\cdots,a_s) | a_1+\cdots+a_s=2m, a_i \geq 2 \textrm{ for } 1 \leq i \leq s \} | \\
		= & | \{ ( \bar{a}_1,\cdots,\bar{a}_s) | \bar{a}_1 + \cdots + \bar{a}_s = 2m-2s, \bar{a}_i \geq 0 \textrm{ for } 1 \leq i \leq s \} | \\
		= &	\sum_{ \substack{ \bar{a}_1 + \cdots + \bar{a}_s = 2m-2s \\ \bar{a}_1 \geq 0,\cdots,\bar{a}_s \geq 0 } } 1\,.
	\end{aligned}
\end{equation*}
Similar to \eqref{comb-result-1}, 
\begin{equation*}
	\sum_{ \substack{ \bar{a}_1 + \cdots + \bar{a}_s = 2m-2s \\ \bar{a}_1 \geq 0,\cdots,\bar{a}_s \geq 0} } 1 = \binom{2m-2s-1}{s-1} \,. 
\end{equation*}
Since $\binom{2m-2s-1}{s-1} \leq 2^{2m} $, we conclude that
\begin{equation*}
	\begin{aligned}
		\frac{1}{(2m)!} & \int_{ \mathbb{T}^{dN} } \bar{\rho}_N \Big(  \frac{1}{N^2} \sum_{i,j,k=1} \phi_2(x_i,x_j,x_k) \Big)^{2m} \d X 
		\leq 2m (8e^3)^m \Big( \sup_{ p \geq 1} \frac{M_p}{p} \Big)^{2m} m \cdot 2^{2m} \\
		= & 2 m^2 (32e^3)^m \Big( \sup_{ p \geq 1} \frac{M_p}{p} \Big)^{2m} = 2m^2 \Big( \sqrt{32e^3} \sup_{ p \geq 1 } \frac{M_p}{p} \Big)^{2m} \,.
	\end{aligned}
\end{equation*}

\section{ proof of Theorem \ref{Thm-solution-Limit-Eq} }\label{sec-proof-Thm-solution-Limit-Eq}
This section is devoted to the well-posedness of the limiting equation \eqref{Limit-Eq}. We establish a priori estimates for the solution and present the proof via a classical iterative method; see for instance \cite{Deg1986,GLM2025, MB2002}.

\subsection{A priori estimates}  
We next derive a priori bounds on the $H^s$-norm of the solution of Equation \eqref{Limit-Eq}; see Lemma \ref{A-priori-H^s-estimates} below. 
\begin{lemma}\label{A-priori-H^s-estimates}
	Let $T \in ( 0, \infty ]$ and $\bar{\rho}$ with $\inf \bar{\rho} > 0$ be a classical solution to Eq. \eqref{Limit-Eq} in the time interval $[0, T)$. Then there is constant $C= C(s,d,\underline{\sigma},\|K\|_{L^\infty},\|\sigma\|_{W^{2,\infty}})$ such that 
	\begin{equation*}
		\| \bar{\rho} \|^2_{H^s} \leq C e^{C t} \ |\bar{\rho}_0 \|^2_{H^s} \,, \quad t \in [ 0, T) \,.
	\end{equation*}
\end{lemma}
\begin{proof}
First, we easily obtain $K*\bar{\rho}$ and $\sigma_{\alpha\alpha}*\bar{\rho}$ for all $\alpha$ are in $C(0,T; C^2( \mathbb{T}^d ) )$ since $K \in L^\infty$, $\sigma \in W^{2,\infty} $ and $\mathbb{T}^d=[0,1]^d$ is compact. The subsequent proof follows by induction on the order of derivatives.

\noindent{\bf\em Zeroth-order estimate.} We have 
\begin{equation*}
	\begin{aligned}
		\frac{\d}{\d t} \| \bar{\rho} \|^2_{L^2} = & \int_{ \mathbb{T}^d } 2\bar{\rho} \partial_t \bar{\rho} 
		= - 2 \int_{ \mathbb{T}^d } \bar{\rho} \nabla \cdot( ( K*\bar{\rho} ) \bar{\rho} ) + 2 \int_{ \mathbb{T}^d } \bar{\rho} \sum^d_{\alpha=1} \partial^2_{x^\alpha} \Big( \bar{\rho} \big( \sigma_{\alpha\alpha}*\bar{\rho} (x) \big)^2  \Big)\\
		=&2\int_{ \mathbb{T}^d } \bar{\rho} K*\rho \cdot \nabla \bar{\rho} - 2 \sum^d_{\alpha=1} \int_{ \mathbb{T}^d } \partial_{x^\alpha} \bar{\rho} \partial_{x^\alpha} \Big( \bar{\rho} \big( \sigma_{\alpha\alpha}*\bar{\rho}(x) \big)^2  \Big)\\	
		\leq & 2 (4\epsilon)^{-1} \| K*\bar{\rho} \|^2_{L^\infty} \| \bar{\rho} \|^2_{L^2} + 2 \epsilon \| \nabla \bar{\rho}\|^2_{L^2} - 2 \sum^d_{\alpha=1} \int_{ \mathbb{T}^d } \partial_{x^\alpha} \bar{\rho} \partial_{x^\alpha} \Big( \bar{\rho} \big( \sigma_{\alpha\alpha}*\bar{\rho}(x) \big)^2  \Big) \,,
	\end{aligned}
\end{equation*}
 where we have used Young’s inequality with $\epsilon$.
Note that 
\begin{equation*}
	\begin{aligned}
		- 2 \int_{ \mathbb{T}^d } & \partial_{x^\alpha} \bar{\rho} \partial_{x^\alpha} \Big( \bar{\rho} \big( \sigma_{\alpha\alpha}*\bar{\rho}(x) \big)^2  \Big) \\
		=& - 2 \int_{ \mathbb{T}^d } ( \partial_{x^\alpha} \bar{\rho} )^2 \big( \sigma_{\alpha\alpha}*\bar{\rho}(x) \big)^2	- 2 \int_{ \mathbb{T}^d } \partial_{x^\alpha} \bar{\rho} \partial_{x^\alpha} \big( \sigma_{\alpha\alpha}*\bar{\rho}(x) \big)^2 \bar{\rho} \\
		= & - 2 \int_{ \mathbb{T}^d } ( \partial_{x^\alpha} \bar{\rho})^2 \big( \sigma_{\alpha\alpha}*\bar{\rho}(x) \big)^2	- \int_{ \mathbb{T}^d } \partial_{x^\alpha} \bar{\rho}^2 \partial_{x^\alpha} \big( \sigma_{\alpha\alpha}*\bar{\rho}(x) \big)^2  \\
		= & - 2 \int_{ \mathbb{T}^d } ( \partial_{x^\alpha} \bar{\rho})^2 \big( \sigma_{\alpha\alpha}*\bar{\rho}(x) \big)^2	+ \int_{ \mathbb{T}^d } \bar{\rho}^2 \partial^2_{x^\alpha} \big( \sigma_{\alpha\alpha}*\bar{\rho}(x) \big)^2 \,.
	\end{aligned}
\end{equation*}
Hence
\begin{equation*}
	\begin{aligned}
		\frac{\d}{\d t} \| \bar{\rho}\|^2_{L^2} & + 2 \sum^d_{\alpha=1} \int_{ \mathbb{T}^d } (\partial_{x^\alpha} \bar{\rho})^2 \big( \sigma_{\alpha\alpha}*\bar{\rho}(x) \big)^2	\\
		= & 2 ( 4 \epsilon )^{-1} \| K*\bar{\rho} \|^2_{L^\infty} \| \bar{\rho} \|^2_{L^2} + 2 \epsilon \| \nabla \bar{\rho} \|^2_{L^2} + \sum^d_{\alpha=1} \int_{ \mathbb{T}^d } \bar{\rho}^2 \partial^2_{x^\alpha} \big( \sigma_{\alpha\alpha}*\bar{\rho}(x) \big)^2 \\
		\leq & 2 (4\epsilon)^{-1} \| K*\bar{\rho} \|^2_{L^\infty} \| \bar{\rho} \|^2_{L^2} + 2 \epsilon \| \nabla \bar{\rho} \|^2_{L^2} + \sum^d_{\alpha=1} \| \partial^2_{x^\alpha} \big( \sigma_{\alpha\alpha}*\bar{\rho}(x) \big)^2 \|_{L^\infty} \| \bar{\rho} \|^2_{L^2} \\
		\leq & 2 (4\epsilon)^{-1} \| K \|^2_{L^\infty} \| \bar{\rho} \|^2_{L^2} + 2 \epsilon \| \nabla \bar{\rho} \|^2_{L^2} + 4 d \| \sigma\|^2_{ W^{2,\infty} } \| \bar{\rho} \|^2_{L^2} \,,
	\end{aligned}
\end{equation*}
where
\begin{equation*}
	\begin{aligned}
		\| \partial^2_{x^\alpha} \big( \sigma_{\alpha\alpha}*\bar{\rho} \big)^2\|_{L^\infty} = & 2 \big\| ( \partial_{x^\alpha} \sigma_{\alpha\alpha}*\bar{\rho} )^2 + \sigma_{\alpha\alpha}*\bar{\rho} \partial^2_{x^\alpha} \sigma_{\alpha\alpha}*\bar{\rho} \big\|_{L^\infty} \\
		\leq & 2 \| \partial_{x^\alpha} \sigma_{\alpha\alpha} \|^2_{L^\infty} + 2 \| \sigma_{\alpha\alpha} \|_{L^\infty} \| \partial^2_{x^\alpha} \sigma_{\alpha\alpha} \|_{L^\infty}
		\leq 4 \|\sigma\|^2_{ W^{2,\infty} } \,.
	\end{aligned}
\end{equation*}
By Assumption \ref{Ass-A2}, we have
\begin{equation*}
	\begin{aligned}
		\frac{\d}{\d t} \| \bar{\rho}\|^2_{L^2} + ( 2(\underline{\sigma})^2 -2\epsilon) \|\nabla\bar{\rho}\|^2 _{L^2}
		\leq & \Big( 2 (4\epsilon)^{-1} \| K \|^2_{L^\infty} + 4d \| \sigma \|^2_{ W^{2,\infty} }  \Big) \| \bar{\rho} \|^2_{L^2} \leq C \| \bar{\rho} \|^2_{L^2} \,,
	\end{aligned}
\end{equation*}
where we  choose $\epsilon$ such that $\underline{\sigma}^2 > 2 \epsilon $.
Integrating the inequality above over $[0,t](t \leq T)$ yields that
\begin{equation}\label{Zeroth-order-estimate}
	\begin{aligned}
		\| \bar{\rho} \|^2_{L^2} \leq \| \bar{\rho}(0) \|^2_{L^2} e^{Ct} \leq \| \bar{\rho}_0 \|^2_{L^2} e^{Ct} \leq C_0 e^{C t} 
	\end{aligned}
\end{equation}
and
\begin{equation*}
	\begin{aligned}
		(  2( \underline{\sigma} )^2 - 2 \epsilon ) \int^t_0 \| \nabla \bar{\rho} \|^2_{L^2} \leq C \int^t_0 \| \bar{\rho}(s) \|^2_{L^2} \d s + \| \bar{\rho}(0) \|^2_{L^2} = \| \bar{\rho}_0 \|^2_{L^2} e^{C t} \,,
	\end{aligned}
\end{equation*}
i.e. 
\begin{equation}\label{Dissipation-zeroth-order-estimate}
	\begin{aligned}
		\int^t_0 \| \nabla \bar{\rho} \|^2_{L^2} \leq \frac{1}{ ( 2 \underline{ \sigma }^2 - 2 \epsilon ) } \| \bar{\rho}_0 \|^2_{L^2} e^{Ct} \leq C^d_0 e^{C t} \,.
	\end{aligned}
\end{equation}
\noindent{\bf\em First-order estimate.} We have 
\begin{equation*}
	\begin{aligned}
		\frac{\d}{\d t} & \| \partial_{x^i} \bar{\rho} \|^2_{L^2} = \int_{ \mathbb{T}^d } 2 \partial_{x^i} \bar{\rho} \partial_t \partial_{x^i} \bar{\rho} 
		\\
		= & - 2 \int_{ \mathbb{T}^d } \partial_{x^i} \bar{\rho} \nabla \cdot \partial_{x^i} ( ( K*\bar{\rho} ) \bar{\rho} ) +  2 \int_{ \mathbb{T}^d } \partial_{x^i} \bar{\rho} \sum^d_{\alpha=1} \partial^2_{x^\alpha} \partial_{x^i} \Big( \bar{\rho} \big( \sigma_{\alpha\alpha}*\bar{\rho}(x) \big)^2 \Big) \\
		= & 2 \int_{ \mathbb{T}^d } \partial_{x^i} \nabla \bar{\rho} \cdot \partial_{x^i} ( K*\bar{\rho}\bar{\rho} ) - 2 \int_{ \mathbb{T}^d } \partial_{x^i} \partial_{x^\alpha} \bar{\rho} \sum^d_{\alpha=1} \partial_{x^\alpha} \partial_{x^i} \Big( \bar{\rho} \big( \sigma_{\alpha\alpha}*\bar{\rho}(x) \big)^2 \Big) \\
		\leq & 2 \epsilon \int_{ \mathbb{T}^d }  | \partial_{x^i} \nabla \bar{\rho} |^2 + 2 ( 4 \epsilon)^{-1} \int_{ \mathbb{T}^d } | \partial_{x^i}( K*\bar{\rho} \bar{\rho} ) |^2 - 2 \sum^d_{\alpha=1} \int_{\mathbb{T}^d} \partial_{x^i} \partial_{x^\alpha} \bar{\rho}  \partial_{x^\alpha} \partial_{x^i} \Big( \bar{\rho} \big( \sigma_{\alpha\alpha}*\bar{\rho}(x) \big)^2 \Big) \,,
	\end{aligned}
\end{equation*}
 where we have used Young’s inequality with $\epsilon$. Note that
\begin{equation*}
	\begin{aligned}
		&\int_{ \mathbb{T}^d } | \partial_{x^i}(  K*\bar{\rho} \bar{\rho} ) |^2 =  \int_{ \mathbb{T}^d } | K*\partial_{x^i} \bar{\rho} \bar{\rho} + K*\bar{\rho} \partial_{x^i} \bar{\rho}|^2
		\leq 2 \int_{ \mathbb{T}^d } \Big( | K*\partial_{x^i} \bar{\rho} \bar{\rho} |^2 + |K*\bar{\rho} \partial_{x^i} \bar{\rho}  |^2  \Big)\\
		&\leq  2 \| K*\partial_{x^i} \bar{\rho} \|^2_{L^\infty} \| \bar{\rho} \|^2_{L^2} + 2 \| K*\bar{\rho} \|^2_{L^\infty} \|\partial_{x^i} \bar{\rho} \|^2_{L^2} 
		\leq  2 \| K \|^2_{L^\infty} \| \partial_{x^i} \bar{\rho} \|^2_{L^2} \| \bar{\rho} \|^2_{L^2} + 2 \| K \|^2_{L^\infty} \| \nabla \bar{\rho} \|^2_{L^2}
	\end{aligned}
\end{equation*}
and
\begin{equation*}
	\begin{aligned}
		\partial_{x^i} \partial_{x^\alpha} \bar{\rho} & \partial_{x^\alpha} \partial_{x^i} \Big( \bar{\rho} \big( \sigma_{\alpha\alpha}*\bar{\rho}(x) \big)^2 \Big) = ( \partial_{x^i} \partial_{x^\alpha}\bar{\rho} )^2 \big( \sigma_{\alpha\alpha}*\bar{\rho}(x) \big)^2 + \partial_{x^i} \partial_{x^\alpha} \bar{\rho} \partial_{x^\alpha} \bar{\rho} \partial_{x^i} \big( \sigma_{\alpha\alpha}*\bar{\rho}(x) \big)^2 \\
		& +  \partial_{x^i} \partial_{x^\alpha} \bar{\rho} \partial_{x^i}\bar{\rho} \partial_{x^\alpha} \big( \sigma_{\alpha\alpha}*\bar{\rho}(x) \big)^2 + \partial_{x^i} \partial_{x^\alpha} \bar{\rho} \bar{\rho} \partial_{x^\alpha} \partial_{x^i} \big( \sigma_{\alpha\alpha}*\bar{\rho}(x) \big)^2 \\
		& := | \partial_{x^i} \partial_{x^\alpha} \bar{\rho} |^2 \big( \sigma_{\alpha\alpha}*\bar{\rho}(x) \big)^2 + \partial_{x^i} \partial_{x^\alpha} \bar{\rho}(I_1) \,,
	\end{aligned}
\end{equation*}
where
\begin{equation*}
	\begin{aligned}
		I_1 := \partial_{x^\alpha} \bar{\rho} \partial_{x^i} \big( \sigma_{\alpha\alpha}*\bar{\rho}(x) \big)^2
		+ \partial_{x^i} \bar{\rho} \partial_{x^\alpha} \big( \sigma_{\alpha\alpha}*\bar{\rho}(x) \big)^2 + \bar{\rho} \partial_{x^\alpha} \partial_{x^i} \big( \sigma_{\alpha\alpha}*\bar{\rho}(x) \big)^2\,.
	\end{aligned}
\end{equation*}
Hence
\begin{equation*}
	\begin{aligned}
		\frac{\d}{\d t} & \| \partial_{x^i} \bar{\rho} \|^2_{L^2}
		+ 2 \sum^d_{\alpha=1} \int_{ \mathbb{T}^d } | \partial_{x^i} \partial_{x^\alpha} \bar{\rho} |^2 \big( \sigma_{\alpha\alpha}*\bar{\rho}(x) \big)^2 \\
		\leq & 2 \epsilon \int_{ \mathbb{T}^d }  | \partial_{x^i} \nabla \bar{\rho}|^2 + 4 (4\epsilon)^{-1} \| K \|^2_{L^\infty} \Big( \|\partial_{x^i} \bar{\rho}\|^2_{L^2} \| \bar{\rho} \|^2_{L^2} + \| \nabla \bar{\rho} \|^2_{L^2} \Big) + 2 \sum^d_{\alpha=1} \int_{\mathbb{T}^d} | \partial_{x^i} \partial_{x^\alpha} \bar{\rho} (I_1) |\\
		\leq & 2\epsilon \int_{ \mathbb{T}^d }  | \partial_{x^i} \nabla \bar{\rho} |^2 + 4 (4\epsilon)^{-1} \| K \|^2_{L^\infty} \Big( \| \partial_{x^i} \bar{\rho} \|^2_{L^2} \| \bar{\rho} \|^2_{L^2} + \| \nabla \bar{\rho} \|^2_{L^2} \Big) + 2 \epsilon \sum^d_{\alpha=1} \int_{\mathbb{T}^d } | \partial_{x^i,x^\alpha}\bar{\rho} |^2  \\
		& + 2(4\epsilon)^{-1} \sum^d_{\alpha=1} \int_{ \mathbb{T}^d } | I_1 |^2 \,.
	\end{aligned}
\end{equation*}
Note that
\begin{equation*}
	\begin{aligned}
		\int_{ \mathbb{T}^d } |I_1|^2
		\leq & 3\Big( \| \partial_{x^i} \big( \sigma_{\alpha\alpha}*\bar{\rho} \big)^2 \|^2_{L^\infty} \| \partial_{x^\alpha} \bar{\rho} \|^2_{L^2} + \| \partial_{x^\alpha} \big( \sigma_{\alpha\alpha}*\bar{\rho} \big)^2 \|^2_{L^\infty} \| \partial_{x^i} \bar{\rho} \|^2_{L^2} \\
		& \quad + \| \partial_{x^i, x^\alpha} \big( \sigma_{\alpha\alpha}*\bar{\rho} \big)^2 \|^2_{L^\infty} \| \bar{\rho} \|^2_{L^2} \Big) \\
		\leq & 3 \Big( 4 \| \sigma_{\alpha\alpha}\|^2_{L^\infty} \| \partial_{x^i} \sigma_{\alpha\alpha} \|^2_{L^\infty} \| \partial_{x^\alpha} \bar{\rho} \|^2_{L^2} + 4 \| \sigma_{\alpha\alpha} \|^2_{L^\infty} \| \partial_{x^\alpha} \sigma_{\alpha\alpha} \|^2_{L^\infty} \| \partial_{x^i} \bar{\rho} \|^2_{L^2} \\
		&+ 8 \big( \|\partial_{x^\alpha} \sigma_{\alpha\alpha} \|^2_{L^\infty} \| \partial_{x^i} \sigma_{\alpha\alpha} \|^2_{L^\infty} + \| \sigma_{\alpha\alpha} \|^2_{L^\infty} \| \partial_{x^\alpha} \sigma_{\alpha\alpha} \|^2_{L^\infty}  \big) \| \bar{\rho} \|^2_{L^2} \Big) \\
		\leq & 3 \Big( 4 \|\sigma\|^4_{ W^{2,\infty} }  \| \partial_{x^\alpha} \bar{\rho} \|^2_{L^2} + 4\|\sigma\|^4_{ W^{2,\infty} } \|\partial_{x^i} \bar{\rho} \|^2_{L^2} + 16 \|\sigma\|^4_{ W^{2,\infty} } \| \bar{\rho} \|^2_{L^2}  \Big) \\
		\leq & 48 \|\sigma\|^4_{ W^{2,\infty} } \Big( \|\partial_{x^\alpha} \bar{\rho} \|^2_{L^2} + \|\partial_{x^i} \bar{\rho} \|^2_{L^2} + \| \bar{\rho} \|^2_{L^2}  \Big)\,.
	\end{aligned}
\end{equation*}
Hence
\begin{equation*}
	\begin{aligned}
		\frac{\d}{\d t} &\| \partial_{x^i} \bar{\rho}\|^2_{L^2} + 2 \sum^d_{\alpha=1} \int_{\mathbb{T}^d} ( \partial_{x^i, x^\alpha} \bar{\rho} )^2 \big( \sigma_{\alpha\alpha}*\bar{\rho} \big)^2\\ 
		\leq & 4\epsilon \sum^d_{\alpha=1} \int_{\mathbb{T}^d} | \partial_{x^i, x^\alpha} \bar{\rho}|^2
		+ 4 (4\epsilon)^{-1} \| K \|^2_{L^\infty} \Big( \| \partial_{x^i} \bar{\rho} \|^2_{L^2} \| \bar{\rho} \|^2_{L^2} + \| \nabla \bar{\rho} \|^2_{L^2} \Big)\\
		& + 96 d (4\epsilon)^{-1} \| \sigma \|^4_{ W^{2,\infty}} \Big( \| \nabla\bar{\rho} \|^2_{L^2} + \|\partial_{x^i} \bar{\rho} \|^2_{L^2} + \| \bar{\rho} \|^2_{L^2}  \Big)\,.
	\end{aligned}
\end{equation*}
By Assumption \ref{Ass-A2}, we have
\begin{equation*}
	\begin{aligned}
		\frac{\d}{\d t} \| \partial_{x^i} \bar{\rho} \|^2_{L^2}  +   (  2 ( \underline{\sigma} )^2 - 4 \epsilon ) & \| \partial_{x^i} \nabla \bar{\rho} \|^2 _{L^2} 
		\leq  C_1 \| \partial_{x^i} \bar{\rho} \|^2_{L^2} ( \| \bar{\rho} \|^2_{L^2} + 1 ) + C_1 \| \nabla \bar{\rho} \|^2_{L^2} + C_1 \| \bar{\rho} \|^2_{L^2} \\
		\leq & C_1 \| \nabla\bar{\rho} \|^2_{L^2} \| \bar{\rho} \|^2_{L^2} + C_1 \| \nabla \bar{\rho} \|^2_{L^2} + C_1 \| \bar{\rho} \|^2_{L^2} \,,
	\end{aligned}
\end{equation*}
where we choose $\epsilon$ such that $\underline{\sigma}^2>2\epsilon$. By integrating the inequality above and using the bounds \eqref{Zeroth-order-estimate} and \eqref{Dissipation-zeroth-order-estimate}, we deduce
\begin{equation*}
	\begin{aligned}
		\| \partial_{x^i} \bar{\rho} \|^2_{L^2} & - \|\partial_{x^i} \bar{\rho}(0) \|^2_{L^2} + ( 2( \underline{\sigma} )^2 - 4\epsilon ) \int^t_0 \| \partial_{x^i} \nabla \bar{\rho} \|^2_{L^2} 
		\leq C_1 \| \bar{\rho}_0 \|^4_{L^2} e^{2C t} + C_1 \| \bar{\rho}_0 \|^2_{L^2} e^{C t}\,.
	\end{aligned}
\end{equation*}
 Hence
\begin{equation}\label{First-order-estimate}
	\begin{aligned}
		\| \partial_{x^i} \bar{\rho} \|^2_{L^2} \leq C_1 \| \bar{\rho}_0 \|^4_{L^2} e^{2Ct} + C_1 \| \bar{\rho}_0 \|^2_{L^2} e^{Ct} + \| \partial_{x^i} \bar{\rho}(0) \|^2_{L^2} \leq C_1 e^{2C t}
	\end{aligned}
\end{equation}
and
\begin{equation}\label{Dissipation-first-order-estimate}
	\begin{aligned}
		\int^t_0 \| \partial_{x^i}\nabla\bar{\rho} \|^2_{L^2} \leq \frac{1}{ ( 2( \underline{\sigma} )^2 - 4\epsilon ) } C_1 e^{2Ct} \leq C^d_1 e^{2Ct} \,.
	\end{aligned}
\end{equation}
\\
\noindent{\bf\em Second-order estimate.} We have 
\begin{equation*}
	\begin{aligned}
		\frac{\d}{\d t} & \| \partial_{x^i,x^j} \bar{\rho} \|^2_{L^2} =  \int_{ \mathbb{T}^d } 2 \partial_{x^i,x^j} \bar{\rho} \partial_t \partial_{x^i,x^j} \bar{\rho} \\
		= & - 2 \int_{ \mathbb{T}^d } \partial_{x^i,x^j} \bar{\rho} \nabla \cdot \partial_{x^i,x^j}( (K*\bar{\rho}) \bar{\rho} )+ 2\int_{ \mathbb{T}^d } \partial_{x^i,x^j} \bar{\rho} \sum^d_{\alpha=1} \partial_{x^i,x^j} \partial^2_{x^\alpha} \Big( \bar{\rho} \big( \sigma_{\alpha\alpha}*\bar{\rho}(x) \big)^2 \Big) \\
		= & 2 \int_{ \mathbb{T}^d } \partial_{x^i,x^j} \nabla\bar{\rho} \cdot  \partial_{x^i,x^j} ( K*\bar{\rho} \bar{\rho} )- 2 \int_{ \mathbb{T}^d } \partial_{x^i,x^j,x^\alpha} \bar{\rho} \sum^d_{\alpha=1} \partial_{x^i,x^j,x^\alpha} \Big( \bar{\rho} \big( \sigma_{\alpha\alpha}*\bar{\rho}(x) \big)^2 \Big) \\
		\leq & 2\epsilon \int_{ \mathbb{T}^d } | \partial_{x^i,x^j}\nabla\bar{\rho} |^2 + 2 (4\epsilon)^{-1} \int_{ \mathbb{T}^d } |  \partial_{x^i,x^j} ( K*\bar{\rho}\bar{\rho} ) |^2 \\
		&- 2 \sum^d_{\alpha=1} \int_{ \mathbb{T}^d } \partial_{x^i,x^j,x^\alpha} \bar{\rho}  \partial_{x^i,x^j,x^\alpha} \Big( \bar{\rho} \big( \sigma_{\alpha\alpha}*\bar{\rho}(x) \big)^2 \Big)\,,
	\end{aligned}
\end{equation*}
where we have used Young’s inequality with $\epsilon$. Note that
\begin{equation*}
	\begin{aligned}
		\int_{ \mathbb{T}^d }& | \partial_{x^i,x^j} ( K*\bar{\rho}\bar{\rho} )|^2 = \int_{ \mathbb{T}^d } | K*\partial_{x^i,x^j} \bar{\rho} \bar{\rho} + K*\partial_{x^j}\bar{\rho} \partial_{x^i} \bar{\rho} + K*\partial_{x^j} \bar{\rho}\partial_{x^i} \bar{\rho} + K*\bar{\rho} \partial_{x^i,x^j}\bar{\rho}  |^2 \\
		\leq & 4 \int_{ \mathbb{T}^d } \Big( | K*\partial_{x^i,x^j} \bar{\rho} \bar{\rho} |^2 + | K*\partial_{x^j} \bar{\rho} \partial_{x^i} \bar{\rho} |^2 + | K*\partial_{x^j} \bar{\rho} \partial_{x^i} \bar{\rho} |^2 + |K*\bar{\rho} \partial_{x^i,x^j}\bar{\rho} |^2 \Big) \\
		\leq & 4 \Big( \| K\|^2_{L^\infty} \|\partial_{x^i,x^j} \bar{\rho}\|^2_{L^1} \|\bar{\rho}\|^2_{L^2} + \|K\|^2_{L^\infty} \| \partial_{x^j} \bar{\rho} \|^2_{L^1} \| \partial_{x^i} \bar{\rho} \|^2_{L^2} \\
		&+ \| K \|^2_{L^\infty} \| \partial_{x^j} \bar{\rho} \|^2_{L^1} \| \partial_{x^i} \bar{\rho} \|^2_{L^2} + \|K\|^2_{L^\infty} \|\partial_{x^i,x^j} \bar{\rho}\|^2_{L^2} \Big) \\
		\leq & 4 \|K\|^2_{L^\infty} \Big( \| \partial_{x^i,x^j} \bar{\rho} \|^2_{L^2} \|\bar{\rho}\|^2_{L^2} + \| \partial_{x^j} \bar{\rho} \|^2_{L^2} \| \partial_{x^i} \bar{\rho} \|^2_{L^2} + \| \partial_{x^j} \bar{\rho}\|^2_{L^2} \| \partial_{x^i} \bar{\rho} \|^2_{L^2} + \| \partial_{x^i,x^j} \bar{\rho} \|^2_{L^2} \Big)\,.
	\end{aligned}
\end{equation*}
Note also that
\begin{equation*}
	\begin{aligned}
		& \partial_{x^i,x^j,x^\alpha} \bar{\rho}  \partial_{x^i,x^j,x^\alpha} \Big( \bar{\rho} \big( \sigma_{\alpha\alpha}*\bar{\rho}(x) \big)^2 \Big)\\
		= & | \partial_{x^i,x^j,x^\alpha}\bar{\rho} |^2 \big( \sigma_{\alpha\alpha}*\bar{\rho}(x) \big)^2 + \partial_{x^i,x^j,x^\alpha} \bar{\rho} \Big( \partial_{x^j,x^\alpha}\bar{\rho} \partial_{x^i} \big( \sigma_{\alpha\alpha}*\bar{\rho}(x) \big)^2 + \partial_{x^i,x^\alpha} \bar{\rho} \partial_{x^j} \big( \sigma_{\alpha\alpha}*\bar{\rho}(x) \big)^2 \\
		& + \partial_{x^\alpha} \bar{\rho} \partial_{x^i, x^j} \big( \sigma_{\alpha\alpha}*\bar{\rho}(x) \big)^2 +\partial_{x^i,x^j} \bar{\rho} \partial_{x^\alpha} \big( \sigma_{\alpha\alpha}*\bar{\rho}(x) \big)^2 +  \partial_{x^j} \bar{\rho} \partial_{x^i, x^\alpha} \big( \sigma_{\alpha\alpha}*\bar{\rho}(x) \big)^2\\
		&+ \partial_{x^i} \bar{\rho} \partial_{ x^j,x^\alpha} \big( \sigma_{\alpha\alpha}*\bar{\rho}(x) \big)^2 + \bar{\rho} \partial_{x^i, x^j,x^\alpha} \big( \sigma_{\alpha\alpha}*\bar{\rho}(x) \big)^2 \Big) \\
		:= & | \partial_{x^i,x^j,x^\alpha}\bar{\rho} |^2 \big( \sigma_{\alpha\alpha}*\bar{\rho}(x) \big)^2 + \partial_{x^i,x^j,x^\alpha} \bar{\rho} (I_2)\,,
	\end{aligned}
\end{equation*}
where
\begin{equation*}
	\begin{aligned}
		I_2 := & \partial_{x^j,x^\alpha} \bar{\rho} \partial_{x^i} \big( \sigma_{\alpha\alpha}*\bar{\rho}(x) \big)^2 + \partial_{x^i,x^\alpha} \bar{\rho} \partial_{x^j} \big( \sigma_{\alpha\alpha}*\bar{\rho}(x) \big)^2 + \partial_{x^\alpha} \bar{\rho} \partial_{x^i, x^j} \big( \sigma_{\alpha\alpha}*\bar{\rho}(x) \big)^2 \\
		&+\partial_{x^i,x^j} \bar{\rho} \partial_{x^\alpha} \big( \sigma_{\alpha\alpha}*\bar{\rho}(x) \big)^2 +  \partial_{x^j} \bar{\rho} \partial_{x^i, x^\alpha} \big( \sigma_{\alpha\alpha}*\bar{\rho}(x) \big)^2 \\
		&+ \partial_{x^i} \bar{\rho} \partial_{ x^j,x^\alpha}\big( \sigma_{\alpha\alpha}*\bar{\rho} (x) \big)^2 + \bar{\rho} \partial_{x^i, x^j,x^\alpha} \big( \sigma_{\alpha\alpha}*\bar{\rho}(x) \big)^2 \,.
	\end{aligned}
\end{equation*}
Hence
\begin{equation*}
	\begin{aligned}
		\frac{\d}{\d t} & \|\partial_{x^i,x^j}\bar{\rho}\|^2_{L^2}  + 2 \sum^d_{\alpha=1} \int_{\mathbb{T}^d} | \partial_{x^i,x^j,x^\alpha}\bar{\rho} |^2 \big( \sigma_{\alpha\alpha}*\bar{\rho} (x) \big)^2 \leq 2\epsilon \int_{ \mathbb{T}^d } | \partial_{x^i,x^j}\nabla\bar{\rho} |^2 \\
		& + 8 (4\epsilon)^{-1} \|K\|^2_{L^\infty} \Big( \| \partial_{x^i,x^j} \bar{\rho} \|^2_{L^2} \| \bar{\rho} \|^2_{L^2} + \| \partial_{x^j} \bar{\rho} \|^2_{L^2} \| \partial_{x^i} \bar{\rho}\|^2_{L^2} + \| \partial_{x^j} \bar{\rho} \|^2_{L^2} \| \partial_{x^i} \bar{\rho} \|^2_{L^2} \\
		& + \| \partial_{x^i,x^j} \bar{\rho} \|^2_{L^2} \Big) + 2 \sum^d_{\alpha=1} \int_{ \mathbb{T}^d } | \partial_{x^i,x^j,x^\alpha} \bar{\rho}  (I_2) | \\
		\leq & 2\epsilon \int_{ \mathbb{T}^d } | \partial_{x^i,x^j} \nabla \bar{\rho} |^2 + 8 (4\epsilon)^{-1} \|K\|^2_{L^\infty} \Big( \|\partial_{x^i,x^j} \bar{\rho} \|^2_{L^2} \| \bar{\rho} \|^2_{L^2} + \| \partial_{x^j} \bar{\rho} \|^2_{L^2} \|\partial_{x^i} \bar{\rho}\|^2_{L^2} \\
		& + \| \partial_{x^j} \bar{\rho} \|^2_{L^2} \| \partial_{x^i} \bar{\rho} \|^2_{L^2} + \| \partial_{x^i,x^j} \bar{\rho} \|^2_{L^2} \Big) + 2\epsilon \sum^d_{\alpha=1} \int_{ \mathbb{T}^d } | \partial_{x^i,x^j,x^\alpha}\bar{\rho}|^2 + 2 (4\epsilon)^{-1} \sum^d_{\alpha=1} \int_{ \mathbb{T}^d } |I_2|^2 \,.
	\end{aligned}
\end{equation*}
Note that
\begin{equation*}
	\begin{aligned}
		\int_{ \mathbb{T}^d } |I_2|^2
		\leq & + 7 \Big( \| \partial_{x^i} \big( \sigma_{\alpha\alpha}*\bar{\rho}(x) \big)^2 \|^2_{L^\infty} \| \partial_{x^j,x^\alpha} \bar{\rho} \|^2_{L^2} + \| \partial_{x^j} \big( \sigma_{\alpha\alpha}*\bar{\rho}(x) \big)^2 \|^2_{L^\infty} \| \partial_{x^i,x^\alpha} \bar{\rho} \|^2_{L^2} \\
		& + \| \partial_{x^i, x^j} \big( \sigma_{\alpha\alpha}*\bar{\rho}(x) \big)^2 \|^2_{L^\infty} \| \partial_{x^\alpha} \bar{\rho} \|^2_{L^2} +\| \partial_{x^\alpha} \big( \sigma_{\alpha\alpha}*\bar{\rho}(x) \big)^2\|^2_{L^\infty}\|\partial_{x^i,x^j} \bar{\rho}\|^2_{L^2}\\
		& + \|\partial_{x^i, x^\alpha} \big( \sigma_{\alpha\alpha}*\bar{\rho}(x) \big)^2 \|^2_{L^\infty} \|\partial_{x^j} \bar{\rho} \|^2_{L^2} + \| \partial_{ x^j,x^\alpha} \big( \sigma_{\alpha\alpha}*\bar{\rho}(x) \big)^2 \|^2_{L^\infty} \| \partial_{x^i} \bar{\rho} \|^2_{L^2} \\
		& + \| \partial_{x^i, x^j,x^\alpha} \big( \sigma_{\alpha\alpha}*\bar{\rho}(x) \big)^2 \|^2_{L^\infty}\| \bar{\rho} \|^2_{L^2} \Big)\,.
	\end{aligned}
\end{equation*}
Note also that
\begin{equation*}
	\begin{aligned}
		\| \partial_{x^i} \big( \sigma_{\alpha\alpha}*\bar{\rho}(x) \big)^2 \|^2_{L^\infty} \leq &  4 \| \sigma_{\alpha\alpha} \|^2_{L^\infty} \| \partial_{x^i} \sigma_{\alpha\alpha} \|^2_{L^\infty}
		\leq 4 \|\sigma\|^4_{ W^{2,\infty} } \,,\\
		\| \partial_{x^j} \big( \sigma_{\alpha\alpha}*\bar{\rho}(x) \big)^2 \|^2_{L^\infty} \leq & 4 \| \sigma_{\alpha\alpha} \|^2_{L^\infty} \| \partial_{x^j} \sigma_{\alpha\alpha} \|^2_{L^\infty} \leq 4 \|\sigma\|^4_{ W^{2,\infty} } \,, \\ 
		\| \partial_{x^i, x^j} \big( \sigma_{\alpha\alpha}*\bar{\rho}(x) \big)^2 \|^2_{L^\infty}
		\leq & 8 ( \| \partial_{x^i} \sigma_{\alpha\alpha} \|^2_{L^\infty} \| \partial_{x^j} \sigma_{\alpha\alpha} \|^2_{L^\infty} + \|\sigma_{\alpha\alpha} \|^2_{L^\infty} \| \partial_{x^i,x^j} \sigma_{\alpha\alpha} \|^2_{L^\infty} ) \\
		\leq & 16 \|\sigma\|^4_{ W^{2,\infty} }  \,, \\
		\| \partial_{x^\alpha} \big( \sigma_{\alpha\alpha}*\bar{\rho}(x) \big)^2\|^2_{L^\infty} \leq &  4 \| \sigma_{\alpha\alpha} \|^2_{L^\infty} \| \partial_{x^\alpha} \sigma_{\alpha\alpha} \|^2_{L^\infty}
		\leq 4 \|\sigma\|^4_{ W^{2,\infty} }  \\
		\|\partial_{x^i, x^\alpha} \big( \sigma_{\alpha\alpha}*\bar{\rho}(x) \big)^2 \|^2_{L^\infty}\leq & 8 ( \| \partial_{x^\alpha} \sigma_{\alpha\alpha} \|^2_{L^\infty} \| \partial_{x^i} \sigma_{\alpha\alpha} \|^2_{L^\infty} +  \| \sigma_{\alpha\alpha} \|^2_{L^\infty} \| \partial_{x^i,x^\alpha} \sigma_{\alpha\alpha} \|^2_{L^\infty} ) \\
		\leq & 16 \| \sigma \|^4_{ W^{2,\infty} } \,, \\
		\| \partial_{ x^j,x^\alpha} \big( \sigma_{\alpha\alpha}*\bar{\rho}(x) \big)^2 \|^2_{L^\infty} \leq & 8 ( \| \partial_{x^\alpha} \sigma_{\alpha\alpha} \|^2_{L^\infty} \| \partial_{x^j} \sigma_{\alpha\alpha} \|^2_{L^\infty} +  \| \sigma_{\alpha\alpha} \|^2_{L^\infty} \| \partial_{x^j,x^\alpha} \sigma_{\alpha\alpha} \|^2_{L^\infty} ) \\
		\leq & 16 \| \sigma \|^4_{ W^{2,\infty} } \,, 
			\end{aligned}
	\end{equation*}
\begin{equation*}
\begin{aligned}
		\| \partial_{x^i, x^j,x^\alpha} \big( \sigma_{\alpha\alpha}*\bar{\rho}(x) \big)^2 \|^2_{L^\infty} \leq & 16 \Big( \| \partial_{x^j,x^i} \sigma_{\alpha\alpha} \|^2_{L^\infty} \| \partial_{x^\alpha} \sigma_{\alpha\alpha} \|^2_{L^\infty} + \| \partial_{x^j} \sigma_{\alpha\alpha} \|^2_{L^\infty} \| \partial_{x^\alpha,x^i} \sigma_{\alpha\alpha} \|^2_{L^\infty} \\
		+ \| \partial_{x^i} & \sigma_{\alpha\alpha} \|^2_{L^\infty} \| \partial_{x^\alpha,x^j} \sigma_{\alpha\alpha} \|^2_{L^\infty} + \| \sigma_{\alpha\alpha} \|^2_{L^\infty} \| \partial_{x^i,x^j} \sigma_{\alpha\alpha} \|^2_{L^\infty} \| \partial_{x^\alpha} \bar{\rho} \|^2_{L^1} \Big )\\
		\leq & 48 \|\sigma\|^4_{ W^{2,\infty} } + 16 \| \sigma \|^4_{ W^{2,\infty} } \| \partial_{x^\alpha} \bar{\rho} \|^2_{L^2} \,,
	\end{aligned}
\end{equation*}
i.e.
\begin{equation*}
	\begin{aligned}
		\| \partial_{x^i}\big( \sigma_{\alpha\alpha}*\bar{\rho}(x) \big)^2 \|^2_{L^\infty}
		\leq & 4 \|\sigma\|^4_{ W^{2,\infty} } \,, \ 
		\| \partial_{x^j} \big( \sigma_{\alpha\alpha}*\bar{\rho}(x) \big)^2 \|^2_{L^\infty} \leq  4 \| \sigma \|^4_{ W^{2,\infty} } \,,\\ 
		\| \partial_{x^i, x^j} \big( \sigma_{\alpha\alpha}*\bar{\rho}(x) \big)^2 \|^2_{L^\infty} \leq & 
		16 \| \sigma \|^4_{ W^{2,\infty} }\,, \ \| \partial_{x^\alpha} \big( \sigma_{\alpha\alpha}*\bar{\rho}(x) \big)^2 \|^2_{L^\infty} \leq  4 \| \sigma \|^4_{ W^{2,\infty} } \,,\\ 
			\| \partial_{ x^j,x^\alpha} \big( \sigma_{\alpha\alpha}*\bar{\rho}(x) \big)^2 \|^2_{L^\infty}
		\leq  &16 \| \sigma \|^4_{ W^{2,\infty} }\,, \
		\| \partial_{ x^i,x^\alpha} \big( \sigma_{\alpha\alpha}*\bar{\rho}(x) \big)^2 \|^2_{L^\infty}
		\leq 16 \| \sigma \|^4_{ W^{2,\infty} } \,, \\
		\| \partial_{x^i, x^j,x^\alpha} \big( \sigma_{\alpha\alpha}*\bar{\rho}(x) \big)^2 \|^2_{L^\infty}
		\leq & 48 \|\sigma\|^4_{ W^{2,\infty} } + 16 \| \sigma \|^4_{ W^{2,\infty} } \| \partial_{x^\alpha} \bar{\rho} \|^2_{L^2} \,.
	\end{aligned}
\end{equation*}
Hence
\begin{equation*}
	\begin{aligned}
		&\frac{\d}{\d t} \| \partial_{x^i,x^j} \bar{\rho}\|^2_{L^2} + 2 \sum^d_{\alpha=1} \int_{\mathbb{T}^d} | \partial_{x^i,x^j,x^\alpha} \bar{\rho}|^2 \big( \sigma_{\alpha\alpha}*\bar{\rho}(x) \big)^2 \leq 4\epsilon \sum^d_{\alpha=1} \int_{\mathbb{T}^d} | \partial_{x^i,x^j,x^\alpha} \bar{\rho}|^2 \\
		& + 8 (4\epsilon)^{-1} \|K\|^2_{L^\infty} \Big( \|\partial_{x^i,x^j} \bar{\rho} \|^2_{L^2} \| \bar{\rho} \|^2_{L^2} + \| \partial_{x^j} \bar{\rho} \|^2_{L^2} \|\partial_{x^i} \bar{\rho} \|^2_{L^2} + \| \partial_{x^j} \bar{\rho}\|^2_{L^2} \|\partial_{x^i} \bar{\rho} \|^2_{L^2} + \|\partial_{x^i,x^j} \bar{\rho} \|^2_{L^2} \Big)\\
		& + 672 (4\epsilon)^{-1}\|\sigma\|^4_{ W^{2,\infty} }  \sum^d_{\alpha=1} \Big( \|\partial_{x^j,x^\alpha}\bar{\rho} \|^2_{L^2} + \|\partial_{x^i,x^\alpha} \bar{\rho} \|^2_{L^2} + \|\partial_{x^\alpha} \bar{\rho} \|^2_{L^2} + \|\partial_{x^i,x^j} \bar{\rho} \|^2_{L^2}+  \| \partial_{x^j} \bar{\rho}  \|^2_{L^2}\\ 
		&\qquad\qquad\qquad\qquad + \| \partial_{x^i} \bar{\rho}  \|^2_{L^2}  + \| \bar{\rho} \|^2_{L^2} + \|\partial_{x^\alpha} \bar{\rho}\|^2_{L^2} \| \bar{\rho} \|^2_{L^2} \Big)\,.
	\end{aligned}
\end{equation*}
By Assumption \ref{Ass-A2}, we have
\begin{equation*}
	\begin{aligned}
		\frac{\d}{\d t} & \| \partial_{x^i,x^j} \bar{\rho} \|^2_{L^2} + ( 2 \underline{\sigma}^2 - 4\epsilon ) \sum^d_{\alpha=1} \int_{ \mathbb{T}^d } | \partial_{x^i,x^j,x^\alpha} \bar{\rho} |^2 \\
		\leq & C_2 \big( \|\partial_{x^i,x^j} \bar{\rho}\|^2_{L^2} ( \|\bar{\rho} \|^2_{L^2} + 1 ) + \| \partial_{x^i} \bar{\rho} \|^2_{L^2} \| \partial_{x^j} \bar{\rho} \|^2_{L^2} + \| \partial_{x^j} \nabla \bar{\rho} \|^2_{L^2} + \| \partial_{x^i} \nabla \bar{\rho} \|^2_{L^2} \big)\\
		& + C_2 \big( \| \nabla \bar{\rho} \|^2_{L^2} + \| \bar{\rho} \|^2_{L^2} + \| \nabla \bar{\rho} \|^2_{L^2} \| \bar{\rho} \|^2_{L^2} \big) \\
		\leq & C_2 \big( \| \partial_{x^i} \nabla \bar{\rho} \|^2_{L^2} \|\bar{\rho}\|^2_{L^2} + \| \nabla \bar{\rho} \|^4_{L^2} + \| \partial_{x^j} \nabla \bar{\rho} \|^2_{L^2} + \| \partial_{x^i} \nabla \bar{\rho} \|^2_{L^2} \big)\\
		& + C_2 \big( \| \nabla \bar{\rho} \|^2_{L^2} + \| \bar{\rho} \|^2_{L^2} + \| \nabla \bar{\rho} \|^2_{L^2} \| \bar{\rho} \|^2_{L^2}  \big)\,,
	\end{aligned}
\end{equation*}
where we choose $\epsilon$ such that $ \underline{\sigma}^2 > 2\epsilon $.
By integrating the inequality above and using the bounds \eqref{Zeroth-order-estimate}, \eqref{First-order-estimate} and \eqref{Dissipation-first-order-estimate}, we deduce
\begin{equation*}
	\begin{aligned}
		\| \partial_{x^i,x^j} \bar{\rho} \|^2_{L^2} - \| \partial_{x^i,x^j} \bar{\rho}(0) \|^2_{L^2} & + ( 2( \underline{\sigma} )^2 - 4\epsilon ) \sum^d_{\alpha=1} \int^t_0 \int_{ \mathbb{T}^d } ( \partial_{x^i,x^j,x^\alpha}\bar{\rho} )^2 
		\leq C_2 e^{4Ct} \,.
	\end{aligned}
\end{equation*}
Hence
\begin{equation*}
	\begin{aligned}
		\| \partial_{x^i,x^j} \bar{\rho} \|^2_{L^2} \leq C_2 e^{4Ct}\,, \ \ \int^t_0 \int_{ \mathbb{T}^d } \| \partial_{x^i,x^j} \nabla \bar{\rho} \|^2_{L^2} \leq \frac{1}{ (  2 ( \underline{\sigma} )^2 - 4\epsilon )} C_2 e^{4Ct} := C^d_2 e^{4Ct}\,.
	\end{aligned}
\end{equation*}

Proceeding by induction on the order of derivatives and repeating the above argument, we obtain that for all $k\leq s$ and for all $\beta_1,\cdots,\beta_k\in\{1,\cdots,d\}$, 
\begin{equation*}\label{order-s-1-estimate}
	\begin{aligned}
		\|\partial_{\beta_1,\cdots,\beta_k}\bar{\rho}\|^2_{L^2} \leq C_k e^{2kCt}
	\end{aligned}
\end{equation*}
and
\begin{equation*}\label{Dissipation-order-s-1-estime}
	\begin{aligned}
		\int^t_0\int_{\mathbb{T}^d} \| \partial_{\beta_1,\cdots,\beta_k}\nabla\bar{\rho} \|^2_{L^2} \leq \frac{1}{(  2(\underline{\sigma})^2 -4\epsilon )}C_k e^{2kCt}:=C^d_k e^{2kCt}\,,
	\end{aligned}
\end{equation*}
where $C_k=C_k(k,d,\underline{\sigma},\|K\|_{L^\infty},\|\sigma\|_{W^{2,\infty}})$ and $C^d_k=C^d_k(k,d,\underline{\sigma},\|K\|_{L^\infty},\|\sigma\|_{W^{2,\infty}} )$.
Combining the estimates for all orders of derivatives, we obtain the desired result.

\end{proof}

\subsection{Global existence of classical solutions.} This part of the proof follows a classical iterative argument. 
\vspace{2mm}

\textbf{\emph{Step 1: The basic result of an iterative scheme.}} We devise the iterative scheme for \eqref{Limit-Eq} as follows
\begin{equation}\label{Iterative-scheme-limit-Eq}
	\begin{cases}
			\partial_t \bar{\rho}^n + \nabla \cdot ( V^{n-1} \bar{\rho}^n  )
		= \sum^d_{\alpha=1} \partial^2_{x^\alpha} ( U^{n-1}_{\alpha} \bar{\rho}^n )\,,\\
		V^{n} = V( \bar{\rho}^n ) = K*\bar{\rho}^n \,, U^{n}_{\alpha} = U_\alpha ( \bar{\rho}^n ) = ( \sigma_{\alpha\alpha}*\bar{\rho}^n )^2 \\
		\bar{\rho}^n(0,\cdot) = \bar{\rho}^n_0 = \bar{\rho}_0 *\eta_\varepsilon \,, \ n \geq 0 \,,
	\end{cases}
\end{equation}
where  $\eta_\varepsilon(x)$ is the standard mollifier so that $\bar{\rho}^n(0) \to \bar{\rho}_0$ in $H^s( \mathbb{T}^d )$ and  we assume $\bar{\rho}^{-1} := 0 $ in the first iteration. Observe that \eqref{Iterative-scheme-limit-Eq} is linear for each $n$ and $\bar{\rho}^n(0) \in C^\infty( \mathbb{T}^d )$, we have
\begin{lemma}
	For a finite $T > 0$, the system \eqref{Iterative-scheme-limit-Eq} has a sequence of $C^\infty([0,T] \times \mathbb{T}^d )$ solutions $\{ \bar{\rho}^n, V^n, U^n_{\alpha} \}$ for all $\alpha$ with $\inf \bar{\rho}^n > 0$ and $\int_{\mathbb{T}^d} \bar{\rho}^n \d x = 1 $. 
\end{lemma}
\begin{proof}
	We use induction on $n$. The case $n=0$ is clear. Assume the solutions $\{ \bar{\rho}^{i}, V^{i-1}, U^{i-1}_{\alpha}\}$ for $ 0 \leq i \leq n $ and $ 1 \leq \alpha \leq d $ are $C^\infty([0,T] \times \mathbb{T}^d )$. Since $K \in L^\infty$, $\sigma \in W^{2,\infty}$ and $\mathbb{T}^d$ is compact, we easily obtain $V^n$ and $U^n_{\alpha}$ for all $\alpha$ are $C^\infty( [0,T] \times \mathbb{T}^d )$. By the standard linear solvability theory \cite[Chapter 7, Section 7.1]{Eva2010}, the equation \eqref{Iterative-scheme-limit-Eq} has a unique solution $\bar{\rho}^n\in C^\infty([0,T]\times \mathbb{T}^d ) $. We now proof $\inf \bar{\rho}^n > 0$. Rewrite \eqref{Iterative-scheme-limit-Eq} as
	$$\partial_t \bar{\rho}^n - \sum^d_{\alpha=1} U^{n-1}_{\alpha} \partial^2_{ x^\alpha } \bar{\rho}^n + \sum^d_{\alpha=1}( V^{n-1}_\alpha - 2 \partial_{x^\alpha}
	U^{n-1}_\alpha ) \partial_{x^\alpha} \bar{\rho}^n + \sum^d_{\alpha=1 }(\partial_{x^\alpha} V^{n-1} - \partial^2_{\alpha} U^{n-1}_\alpha ) \bar{\rho}^n = 0 \,.$$
Let $\bar{\rho}^n = \tilde{\rho}^n e^{ C t }$, then
	$$\partial_t \tilde{\rho}^n -\sum^d_{\alpha=1} U^{n-1}_{\alpha} \partial^2_{ x^\alpha } \tilde{\rho}^n + \sum^d_{\alpha=1}( V^{n-1}_\alpha - 2 \partial_{x^\alpha}
U^{n-1}_\alpha ) \partial_{x^\alpha} \tilde{\rho}^n + \sum^d_{\alpha=1 }(\partial_{x^\alpha} V^{n-1} - \partial^2_{\alpha} U^{n-1}_\alpha + C ) \tilde{\rho}^n = 0 \,.$$
We choose $C$ such that $C< \sum^d_{\alpha=1 }( \partial^2_{\alpha} U^{n-1}_\alpha - \partial_{x^\alpha} V^{n-1} )$. By the maximum principle, 
	 $\tilde{\rho}^n$ does not attain a negative minimum for $t>0$. Note that
	$\inf \tilde{\rho}^n_0 = \inf \bar{\rho}^n_0 > 0$, so
	\begin{equation}\label{bar-rho-n-nonnegative}
		\inf \bar{\rho}^n > 0\,.
	\end{equation}
Rewriting the equation in \eqref{Iterative-scheme-limit-Eq} as
$$	\partial_t \bar{\rho}^n + \sum^d_{\alpha=1} \partial_{x^\alpha} \Big( ( V^{n-1}_\alpha \bar{\rho}^n ) - \sum^d_{\alpha=1} \partial_{x^\alpha} ( U^{n-1}_\alpha \bar{\rho}^n ) \Big)
= 0 $$
and integrating  over $x$ directly implies the conservation of mass. 
\end{proof}

\textbf{\emph{Step 2: Uniform bound of $\bar{\rho}^n $.}} 
We now show that there exists a sufficiently small $T_*>0$ such that $\|\bar{\rho}^n\|^2_{H^s}$ is uniformly bounded on the interval $[0,T_*]$. We again use induction on $n$. If $n=0$, we have $\bar{\rho}^0(t)=\bar{\rho}_0$, then for all $T\geq 0$
$$	\sup_{ 0 \leq t \leq T } \| \bar{\rho}^0(t) \|^2_{ H^s } \leq 2\| \bar{\rho}_0 \|^2_{ H^s } \,.$$
Assume that there exists a fixed $T_*>0$ such that for any $t\leq T_*$
\begin{equation}\label{A-rho-n-1}
	\sup_{ 0 \leq t \leq T_* } \| \bar{\rho}^{n-1} \|^2_{ H^s } \leq 2 \| \bar{\rho}_0 \|^2_{ H^s }\,. 
\end{equation}
 Applying similar computations as in Lemma \ref{A-priori-H^s-estimates} to the equation \eqref{Iterative-scheme-limit-Eq}, we have 
\begin{equation*}
	\begin{aligned}
		\frac{\d}{\d t} \|\bar{\rho}^n\|^2_{H^s}  + (2\underline{\sigma}^2-4\epsilon)  \|\nabla \bar{\rho}^n\|^2_{H^s} 
		\leq C\|\bar{\rho}^{n-1}\|^2_{H^s}\|\bar{\rho}^n\|^2_{H^s}\,,
	\end{aligned}
\end{equation*}
where $C=C(s,d,\|K\|_\infty,\|\sigma\|_{W^{2,\infty}})$.
 Integrating the above inequality over $[0,t]$ for $t\leq T_0\leq T_*$, where $T_0$ is to be determined, yields 
 \begin{equation*}
 \|\bar{\rho}^n\|^2_{H^s} \leq \| \bar{\rho}_0 \|^2_{ H^s } + C t \sup_{ 0 \leq s \leq t }\|\bar{\rho}^{n-1}(s)\|^2_{H^s} \sup_{ 0 \leq s \leq t } \|\bar{\rho}^n(s)\|^2_{H^s} \,.
 \end{equation*}
 Then
 \begin{equation*}
 	\sup_{ 0 \leq t \leq T_0 } \|\bar{\rho}^n(t)\|^2_{H^s} \leq \| \bar{\rho}_0 \|^2_{ H^s } + C T_0 \sup_{ 0 \leq t \leq T_0 }\|\bar{\rho}^{n-1}(t)\|^2_{H^s} \sup_{ 0 \leq t \leq T_0 } \|\bar{\rho}^n(t)\|^2_{H^s} \,.
 \end{equation*}
 By \eqref{A-rho-n-1}, we have
\begin{equation*}
	\sup_{ 0 \leq t \leq T_0 } \|\bar{\rho}^n(t)\|^2_{H^s} \leq \| \bar{\rho}_0 \|^2_{ H^s } + 2C T_0\|\bar{\rho}_0\|^2_{H^s} \sup_{ 0 \leq t \leq T_0 } \|\bar{\rho}^n(t)\|^2_{H^s} \,.
\end{equation*}
 i.e.
 \begin{equation*}
 	\sup_{ 0 \leq t \leq T_0 } \|\bar{\rho}^n\|^2_{H^s} \leq \frac{ \| \bar{\rho}_0 \|^2_{ H^s}  }{1-C T_0\|\bar{\rho}_0\|^2_{H^s} }\,.
 \end{equation*}
 By choosing  $T_0$ such that $ 1 - C T_0\|\bar{\rho}_0\|^2_{H^s} \geq 1/2$, i.e. $ T_0\leq  \frac{1}{2C\|\bar{\rho}_0\|^2_{H^s}}$, one can deduce that 
 \begin{equation*}
 	\sup_{ 0 \leq t \leq T_0 } \|\bar{\rho}^n\|^2_{H^s}  \leq 2\|\bar{\rho}_0\|^2_{H^s}\,.
 \end{equation*}
Taking $T_*=T_0$ gives
 \begin{equation}\label{Uniform-bound-bar-rho-n}
 	\sup_{ 0 \leq t \leq T_* } \| \bar{\rho}^n \|^2_{ H^s } \leq 2 \| \bar{\rho}_0 \|^2_{ H^s } \,.
 \end{equation}

\textbf{\emph{Step 3: Convergence of $\bar{\rho}^n$ in $C ( 0, T_*; L^2 ( \mathbb{T}^d ) )$.}} It follows from \eqref{Iterative-scheme-limit-Eq} that
	\begin{equation*}
		\begin{aligned}
			& \partial_t ( \bar{\rho}^{ n + 1 } - \bar{\rho}^n ) + \nabla \cdot \big( V^n ( \bar{\rho}^{ n + 1 } - \bar{\rho}^n ) \big) + \nabla \cdot \big(   V ( \bar{\rho}^n - \bar{\rho}^{ n - 1 } ) \bar{\rho}^n \big) \\
			& = \sum^d_{\alpha=1} \partial^2_{x^\alpha} \big( U^n_\alpha ( \bar{\rho}^{ n + 1 } - \bar{\rho}^n ) \big) + \sum^d_{\alpha=1} \partial^2_\alpha \big[ \big(  U^n_\alpha - U^{n-1}_\alpha \big) \bar{\rho}^n \big]\,.
		\end{aligned}
	\end{equation*}
Let $G^{ n + 1 } = \bar{\rho}^{ n + 1 } - \bar{\rho}^n $. Then
	\begin{equation*}
		\begin{aligned}
			& \partial_t G^{ n + 1 }  + \nabla \cdot \big( V^n G^{ n + 1 } \big) + \nabla \cdot \big( V (G^n ) \bar{\rho}^n \big) = \sum^d_{\alpha=1} \partial^2_{x^\alpha} \big( U^n_\alpha  G^{ n + 1 }  \big) + \sum^d_{\alpha=1} \partial^2_{x^\alpha} \big[ \big( U^n_\alpha - U^{n-1}_\alpha \big) \bar{\rho}^n \big]\,.
		\end{aligned}
	\end{equation*}
Hence
\begin{equation*}
	\begin{aligned}
		\frac{\d}{\d t} & \| G^{n+1} \|^2_{L^2} =  2 \int_{ \mathbb{T}^d } G^{n+1} \partial_t G^{n+1}
		\\
		= & - 2\int_{\mathbb{T}^d} G^{n+1} \nabla \cdot \big( V( \bar{\rho}^n ) G^{ n + 1 } \big) - 2 \int_{ \mathbb{T}^d } G^{n+1} \nabla \cdot \big( V (G^n ) \bar{\rho}^n \big) \\
		& + 2\sum^d_{\alpha=1} \int_{ \mathbb{T}^d } G^{n+1} \partial^2_{x^\alpha} \big( U^n_\alpha G^{ n + 1 } \big) + 2 \sum^d_{\alpha=1} \int_{ \mathbb{T}^d } G^{n+1} \partial^2_{x^\alpha} \big[ \big( U^n_\alpha - U^{n-1}_\alpha \big) \bar{\rho}^n \big]\\
		= & 2\int_{ \mathbb{T}^d } G^{n+1} V^n \cdot \nabla G^{ n + 1 } + 2\int_{ \mathbb{T}^d } V (G^n ) \bar{\rho}^n \cdot \nabla G^{n+1} \\
		& - 2 \sum^d_{\alpha=1} \int_{ \mathbb{T}^d } \partial_{x^\alpha} G^{n+1} \partial_{x^\alpha} \big( U^n_\alpha G^{ n + 1 } \big) + 2 \sum^d_{\alpha=1} \int_{ \mathbb{T}^d } G^{n+1} \partial^2_{x^\alpha} \big[ \big( U^n_\alpha - U^{n-1}_\alpha \big) \bar{\rho}^n \big]\\
		\leq & 2 (4\epsilon)^{-1} \| V^n \|^2_{L^\infty} \| G^{n+1} \|^2_{L^2} + 2\epsilon \| \nabla G^{ n + 1 }\|^2_{L^2} + 2 (4\epsilon)^{-1} \| V ( G^n ) \bar{\rho}^n \|^2_{L^2} + 2 \epsilon \|\nabla G^{n+1} \|^2_{L^2} \\
		&- 2 \sum^d_{\alpha=1} \int_{ \mathbb{T}^d } \partial_{x^\alpha} G^{n+1} \partial_{x^\alpha} \big( U^n_\alpha  G^{ n + 1 } \big) + d \| G^{n+1} \|^2_{L^2} + \sum^d_{\alpha=1} \| \partial^2_{x^\alpha} \big[ \big( U^n_\alpha - U^{n-1}_\alpha \big) \bar{\rho}^n \big] \|^2_{L^2} \,,
	\end{aligned}
\end{equation*}
where we have used Young’s inequality with $\epsilon$. Note that
\begin{equation*}
	\begin{aligned}
		-& 2\int_{ \mathbb{T}^d } \partial_{x^\alpha} G^{n+1} \partial_{x^\alpha} \big( U^n_\alpha G^{ n + 1 } \big) = -2\int_{ \mathbb{T}^d } (\partial_{x^\alpha} G^{n+1})^2 U^n_{\alpha} -2\int_{ \mathbb{T}^d } \partial_{x^\alpha} G^{n+1} \partial_{x^\alpha} U^n_\alpha G^{n+1}\\
		= & - 2 \int_{ \mathbb{T}^d } ( \partial_{x^\alpha} G^{n+1} )^2 U^n_{\alpha} -\int_{ \mathbb{T}^d } \partial_{x^\alpha} (G^{n+1})^2 \partial_{x^\alpha} U^n_\alpha \\
		= & - 2 \int_{ \mathbb{T}^d } ( \partial_{x^\alpha} G^{n+1} )^2 U^n_{\alpha} + \int_{ \mathbb{T}^d }  ( G^{n+1} )^2 \partial^2_{x^\alpha} U^n_\alpha \,.
	\end{aligned}
\end{equation*}
Hence
\begin{equation*}
	\begin{aligned}
		\frac{\d}{\d t} & \| G^{n+1} \|^2_{L^2} + 2 \sum^d_{\alpha=1} \int_{ \mathbb{T}^d } ( \partial_{x^\alpha} G^{n+1} )^2 U^n_{\alpha} \\
		= & 2 (4\epsilon)^{-1} \| V^n \|^2_{L^\infty} \| G^{n+1} \|^2_{L^2} + 2 \epsilon \| \nabla G^{ n + 1 }\|^2_{L^2} + 2 (4\epsilon)^{-1} \| V ( G^n ) \bar{\rho}^n \|^2_{L^2} + 2\epsilon \|\nabla G^{n+1} \|^2_{L^2} \\
		& + \sum^d_{\alpha=1} \int_{ \mathbb{T}^d } (G^{n+1})^2 \partial^2_{x^\alpha} U^n_\alpha + d \| G^{n+1} \|^2_{L^2} + \sum^d_{\alpha=1} \| \partial^2_{x^\alpha} \big[ \big( U^n_\alpha - U^{n-1}_\alpha \big) \bar{\rho}^n \big] \|^2_{L^2}\\
		\leq & 2 (4\epsilon)^{-1} \|K\|^2_{L^\infty} \| G^{n+1} \|^2_{L^2} + 2\epsilon \| \nabla G^{n+1} \|^2_{L^2} + 2 (4\epsilon)^{-1} \|K\|^2_{L^\infty} \|G^n\|^2_{L^2}\|\bar{\rho}^n\|^2_{L^2} + 2 \epsilon \| \nabla G^{n+1}\|^2_{L^2}\\
		& + \sum^d_{\alpha=1} \| \partial^2_{x^\alpha} U^n_\alpha \|_{L^\infty} \| G^{n+1} \|^2_{L^2} + d \| G^{n+1} \|^2_{L^2} + \sum^d_{\alpha=1} \| \partial^2_{x^\alpha} \big[ \big( U^n_\alpha - U^{n-1}_\alpha \big) \bar{\rho}^n \big] \|^2_{L^2}\\
		\leq & 2 (4\epsilon)^{-1} \|K\|^2_{L^\infty} \| G^{n+1} \|^2_{L^2} + 2\epsilon \| \nabla G^{n+1} \|^2_{L^2} + 2 (4\epsilon)^{-1} \| K \|^2_{L^\infty} \| G^n \|^2_{L^2} \| \bar{\rho}^n \|^2_{L^2} + 2\epsilon \| \nabla G^{n+1} \|^2_{L^2}\\
		& + 2d \| \sigma \|^2_{W^{2,\infty}} \| G^{n+1} \|^2_{L^2} + d \| G^{n+1} \|^2_{L^2} + \sum^d_{\alpha=1} \| \partial^2_{x^\alpha} \big[ \big( U^n_\alpha - U^{n-1}_\alpha \big) \bar{\rho}^n \big] \|^2_{L^2}\,,
	\end{aligned}
\end{equation*}
where we use
\begin{equation*}
	\begin{aligned}
		\| \partial^2_{x^\alpha} U^n_\alpha \|_{L^\infty} = & \| \partial^2_{x^\alpha} \big( \sigma_{\alpha\alpha}*\bar{\rho}^n \big)^2 \|_{L^\infty} = 2 \big\| ( \partial_{x^\alpha} \sigma_{\alpha\alpha}*\bar{\rho}^n )^2 + \sigma_{\alpha\alpha}*\bar{\rho}^n \partial^2_{x^\alpha} \sigma_{\alpha\alpha}*\bar{\rho}^n \big\|_{L^\infty} \\
		\leq & 2 \|\partial_{x^\alpha} \sigma_{\alpha\alpha} \|^2_{L^\infty} + 2 \| \sigma_{\alpha\alpha} \|_{L^\infty} \| \partial^2_{x^\alpha} \sigma_{\alpha\alpha} \|_{L^\infty}
		\leq 2 \|\sigma\|^2_{ W^{2,\infty} } \,.
	\end{aligned}
\end{equation*}
Note that
\begin{equation*}
	\begin{aligned}
	\| &\partial^2_{x^\alpha} \big[ \big( U^n_\alpha - U^{n-1}_\alpha \big) \bar{\rho}^n \big] \|^2_{L^2}=\int_{ \mathbb{T}^d } |  \partial^2_{x^\alpha} \big[ \big( U^n_\alpha - U^{n-1}_\alpha \big) \bar{\rho}^n \big] |^2\\
		\leq & 3\int_{ \mathbb{T}^d } \Big( |  \partial^2_{x^\alpha} \big( U^n_\alpha - U^{n-1}_\alpha \big) \bar{\rho}^n |^2 + 4 | \partial_{x^\alpha} \big( U^n_\alpha - U^{n-1}_\alpha \big) \partial_{x^\alpha} \bar{\rho}^n |^2 + | \big( U^n_\alpha - U^{n-1}_\alpha \big)   \partial^2_{x^\alpha} \bar{\rho}^n |^2 \Big)\\
		\leq & 3 \| \partial^2_{x^\alpha} \big( U^n_\alpha - U^{n-1}_\alpha \big)\|^2_{L^\infty}\|\bar{\rho}^n\|^2_{L^2} + 12 \| \partial_{x^\alpha} \big( U^n_\alpha - U^{n-1}_\alpha \big)\|^2_{L^\infty}\|\partial_{x^\alpha} \bar{\rho}^n\|^2_{L^2} \\
		&+ 3 \| U^n_\alpha - U^{n-1}_\alpha \|^2_{L^\infty} \| \partial^2_{x^\alpha} \bar{\rho}^n \|^2_{L^2}\,.
	\end{aligned}
\end{equation*}
By the bound \eqref{Uniform-bound-bar-rho-n} and the fact that for $k=0,1,2$
\begin{equation*}
	\begin{aligned}
    \partial^k_{x^\alpha} \big( U^n_\alpha - U^{n-1}_\alpha \big) = &  \partial^k_{x^\alpha} \big[ \big( \sigma_{\alpha\alpha}*(\bar{\rho}^n + \bar{\rho}^{n-1}) \big)(t,x) \big( \sigma_{\alpha\alpha}*(\bar{\rho}^n - \bar{\rho}^{n-1} ) \big)(t,x) \big] \\
		= & \sum_{0 \leq j \leq k }\tbinom{k}{j} \big( \partial^{ k-j }_{x^\alpha} \sigma_{\alpha\alpha}*( \bar{\rho}^n + \bar{\rho}^{n-1} ) \big)(t,x) \big( \partial^j_{x^\alpha} \sigma_{\alpha\alpha}*(\bar{\rho}^n - \bar{\rho}^{n-1} ) \big) (t,x) \,,
	\end{aligned}
\end{equation*}
we obtain
\begin{equation*}
	\begin{aligned}
		\| \partial^2_{x^\alpha} \big[ \big( U^n_\alpha - U^{n-1}_\alpha \big) \bar{\rho}^n \big] \|^2_{L^2}
		\leq  C\|G^n\|^2_{L^2}\,,
	\end{aligned}
\end{equation*}
where $C = C( \| \sigma \|_{ W^{2,\infty} }, \|\bar{\rho}_0\|_{H^s} )$.
Hence
\begin{equation*}
	\begin{aligned}
		\frac{\d}{\d t} & \| G^{n+1} \|^2_{L^2} + 2 \sum^d_{\alpha=1} \int_{ \mathbb{T}^d } ( \partial_{x^\alpha} G^{n+1})^2 U^n_{\alpha} \\
		\leq & 2 (4\epsilon)^{-1} \| K \|^2_{L^\infty} \| G^{n+1} \|^2_{L^2} + 2\epsilon \| \nabla G^{n+1} \|^2_{L^2} + 2(4\epsilon)^{-1} \| K \|^2_{L^\infty} \| G^n \|^2_{L^2} \| \bar{\rho}^n \|^2_{L^2} + 2 \epsilon \| \nabla G^{n+1} \|^2_{L^2} \\
		& + 2d \|\sigma\|^2_{ W^{2,\infty} } \| G^{n+1} \|^2_{L^2} + d \| G^{n+1} \|^2_{L^2} + dC \| G^n \|^2_{L^2} \,.
	\end{aligned}
\end{equation*}
By Assumption \ref{Ass-A2} and the bound \eqref{Uniform-bound-bar-rho-n}, we have
\begin{equation}\label{Ineq-G-n+1-L2}
	\begin{aligned}
		\frac{\d}{\d t} \| G^{n+1} \|^2_{L^2} + ( 2 \underline{\sigma}^2 - 4\epsilon ) \| \nabla G^{n+1} \|^2_{L^2}
		\leq C_1 \| G^{n+1} \|^2_{L^2} + C_2 \| G^n \|^2_{L^2} \,,
	\end{aligned}
\end{equation}
where $C_1 = C_1(d,\|\bar{\rho}_0\|_{H^s},\|K\|_{L^\infty},\|\sigma\|_{W^{2,\infty}}) $, $C_2 = C_2(d,\|\bar{\rho}_0\|_{H^s},\|K\|_{L^\infty},\|\sigma\|_{W^{2,\infty}})$ and we choose $\epsilon$ such that $\underline{\sigma}^2 > 2\epsilon $.
 Using the Gronwall’s inequality to the above equation, it easily follows that
	\begin{equation*}
		\| G^{ n + 1 } \|^2_{L^2} \leq C_2e^{C_1 t }\int^t_0 \| G^n \|^2_{L^2} \d s  \leq C_2e^{C_1 T_* }\int^t_0 \| G^n \|^2_{L^2} \d s := \tilde{ C} \int^t_0 \| G^n \|^2_{L^2} \d s \,.
	\end{equation*}
	Thus the inductive method gives
	\begin{equation*}
		\| \bar{\rho}^{ n + 1 } - \bar{\rho}^n \|^2_{L^2} \leq \dfrac{ \tilde{C}^{n+1} t^{n+1} }{ (n+1)! } \max_{t \in [0, T_*]}\| \bar{\rho}^1 - \bar{\rho}^0 \|^2_{L^2}\,. 
	\end{equation*}
	Then there exists a unique limit $\bar{\rho}$ such that
	\begin{equation*}\label{Limit-L2}
		\bar{\rho}^n \to \bar{\rho} \text{ in } C ( 0, T_*; L^2( \mathbb{T}^d) ) \text{ as } \ n\to\infty.
	\end{equation*}
We incidentally provide here two properties of $\bar{\rho}$. one is $\inf \bar{\rho} > 0$. Because $\bar{\rho}^n$ converges to $\bar{\rho}$  in $C ( 0, T_*; L^2( \mathbb{T}^d) )$. This means that $\bar{\rho} \geq \bar{\rho}^n - \varepsilon$ for any $\varepsilon > 0$, and therefore $\inf \bar{\rho} > 0$ by \eqref{bar-rho-n-nonnegative}. The other is
	\begin{equation*}
		\| \bar{\rho} \|^2_{L^\infty(  0,T_*; H^s(\mathbb{T}^d) ) } \leq 2 \| \bar{\rho}_0 \|^2_{ H^s } \,.
	\end{equation*}
	Indeed, the bound \eqref{Uniform-bound-bar-rho-n} indicates that there exists a subsequence $\bar{\rho}^{n_j}$ of $\bar{\rho}^n$ that converges weakly to $\bar{\rho}$ in $L^\infty \{ [0,T_*]; H^s ( \mathbb{T}^d ) ) \}$.
	the lower semicontinuity of norm gives that
	\begin{equation*}
		\| \bar{\rho} \|^2_{L^\infty(  0,T_*;H^s( \mathbb{T}^d) ) } \leq \liminf_{ j \to \infty } \| \bar{\rho}^{n_j} \|^2_{L^\infty(  0,T_*;H^s( \mathbb{T}^d ) ) } \leq 2 \| \bar{\rho}_0 \|^2_{ H^s } \,. 
	\end{equation*}
	For brevity, we will still denote $n_j$ by $n$. 
	
	\textbf{\emph{Step 4: Local existence of classical solutions.}} From the above discussion, the Sobolev interpolation inequality $\|u\|_{H^r} \leq \Vert u \Vert ^{1-r/s}_{L^2} \Vert u \Vert^{r/s}_{H^s}$ for $0\leq r\leq s$ implies that 
	\begin{equation*}
	\bar{\rho}^{n} \to \bar{\rho} \textrm{ in } C ( 0, T_* ; H^r ( \mathbb{T}^d ) ) \,, \textrm{ as } n \to \infty,
	\end{equation*}
	for $0 \leq r < s$. If one choose $r \in (d/2+2,s)$, the Sobolev imbedding theorem yields
	\begin{equation*}
		\bar{\rho}^{n} \to \bar{\rho} \textrm{ in } C \{ [0, T_* ]; C^2 ( \mathbb{T}^d ) ) \} \,, \textrm{ as } n \to \infty \,.
	\end{equation*}
From the equation in \eqref{Iterative-scheme-limit-Eq}, we further deduce $\partial_t \bar{\rho}^n \to \partial_t \bar{\rho}$ in $C \{ [0, T_* ]; C ( \mathbb{T}^d ) ) \} $. These show $\bar{\rho}$ is a classical solution to \eqref{Limit-Eq}. Let $\bar{\rho}_1$ and $\bar{\rho}_2$ be two solutions to \eqref{Limit-Eq} with initial value $\bar{\rho}_0$. Uniqueness then follows from a standard argument similar to the one that yields \eqref{Ineq-G-n+1-L2}.

\textbf{\emph{Step 5: Global existence of classical solutions.}} Let $T_{max}$ denote the maximal existence time of the solution. By the local existence result, there exists $T_*>0$ such that a unique solution exists on $[0,T_*)$, hence $0<T_*\leq T_{max}$. We claim that $T_{max}=\infty$. Assume to the contrary that $T_{max}<\infty$. Then the a priori estimate (Lemma \ref{A-priori-H^s-estimates}) implies that
$$ \| \bar{\rho}(T_{max}) \|_{H^s}<\infty \,.$$
Now, taking $\bar{\rho}(T_{max})$ as the initial data, the above local existence result yields a solution on $[0, T_{max}+\delta)$ for some $\delta>0$. This contradicts the definition of $T_{max}$ as the maximal existence time. Therefore $T_{max}=\infty$ and the solution is global.

\section{ Proof of Theorem \ref{Thm-entropy-solution-Liou} }\label{sec-proof-Thm-weak-solution-Liou-Eq}

In this section, we prove the existence of weak solutions to \eqref{Liouville-Eq} with initial data $\rho^N_0$ by the classical approximation argument, and further establish the entropy inequality \ref{Entropy- dissipation-inequality} required for the  previous proof. This argument is inspired by the proof in \cite[Proposition 1]{JW2018} and the treatment of the Landau master equation in \cite{CG2025, FW2025}.

Let $\eta_\varepsilon(x) := \varepsilon^{-d} \eta(x/\varepsilon) $ be a sequence of standard mollifiers such that 
$ \int_{\mathbb{T}^d} \eta_\varepsilon(x) \d x = 1 $ and $\eta_\varepsilon \in C^\infty_c( \mathbb{T}^d )$. We consider 
$$K^\varepsilon = K*\eta_\varepsilon \,, \  \sigma^\varepsilon = \sigma*\eta_\varepsilon \,, \ \rho^{0,\varepsilon}_N = \rho^0_N*\eta^N_\varepsilon(X) = \rho^0_N*\Pi^N_{i=1} \eta_\varepsilon(x_i) \,.$$
By the standard linear solvability theory \cite[Chapter 7, Section 7.1]{Eva2010}, there exists a unique solution $\rho^\varepsilon_N \in C^\infty( \mathbb{T}^{dN} )$ for all $ t \geq 0$ and $\varepsilon > 0$ to 
\begin{equation}\label{Regularized-liou-eq}
	\begin{aligned}
		\partial_t \rho^\varepsilon_N & + \sum^N_{i=1} \nabla_{x_i} \cdot \Big( \rho^\varepsilon_N \Big( \dfrac{1}{N} \sum \limits_{k = 1}^N K^\varepsilon( x_i-x_k ) \Big) \Big)
		= \sum^N_{i=1} \sum^d_{\alpha=1} \partial^2_{x^\alpha_i} \Big( \rho^\varepsilon _N \dfrac{1}{N^2} \Big( \sum\limits_{ k=1 }^N \sigma^\varepsilon_{\alpha\alpha}( x_i-x_k ) \Big)^2 \Big)
	\end{aligned}
\end{equation}
with initial data $ \rho^\varepsilon_N(0) = \rho^{0,\varepsilon}_N $.
 It is easy to obtain $ \rho^\varepsilon_N \geq 0 $ and 
\begin{equation*}\label{mass-conservation-liou-eq}
	\| \rho^\varepsilon_N \|_{L^\infty(0,T,L^1(\mathbb{T}^{dN}))} = \| \rho^{0,\varepsilon}_N \|_{L^1(\mathbb{T}^{dN})} = \| \rho^0_N \|_{ L^1( \mathbb{T}^{dN} ) }\,.
\end{equation*}
We next derive several useful estimates for the regularized equation \eqref{Regularized-liou-eq}.

\vspace{2mm}

\noindent{\emph{Uniform estimates for
	entropy.}} From the equation \eqref{Regularized-liou-eq}, we have that
\begin{equation*}
	\begin{aligned}
	\frac{\d}{\d t} \int_{ \mathbb{T}^{dN} } \rho^\varepsilon_N \log \rho^\varepsilon_N \d X = & \int_{ \mathbb{T}^{dN} } \partial_t \rho^\varepsilon_N ( 1 + \log \rho^\varepsilon_N )\d X = \int_{ \mathbb{T}^{dN} } \partial_t \rho^\varepsilon_N \log \rho^\varepsilon_N \d X \\
	= & - \dfrac{1}{N} \sum\limits^N_{i=1} \sum^d_{\alpha=1} \int_{ \mathbb{T}^{dN} } \partial_{x^\alpha_i} \Big( \sum^N_{k=1} K^\varepsilon_\alpha ( x_i-x_k ) \rho^\varepsilon_N \Big) \log \rho^\varepsilon_N \d X \\
	&+ \frac{1}{N^2} \sum\limits^N_{i=1} \sum^d_{\alpha=1} \int_{ \mathbb{T}^{dN} } \partial^2_{x^\alpha_i} \Big( \rho^\varepsilon_N \Big( \sum\limits_{ k=1 }^N \sigma^\varepsilon_{\alpha\alpha}( x_i-x_k ) \Big)^2 \Big) \log \rho^\varepsilon_N \d X \,.
	\end{aligned}
\end{equation*}
By integration by parts
\begin{equation}\label{Entropy-Time-Evolution}
\begin{aligned}
\frac{\d}{\d t} & \int_{ \mathbb{T}^{dN} } \rho^\varepsilon_N \log \rho^\varepsilon_N \d X + \frac{1}{N^2} \sum\limits^N_{i=1} \sum^d_{\alpha=1} \int_{ \mathbb{T}^{dN} } \Big( \sum\limits_{k = 1}^N \sigma^\varepsilon_{\alpha\alpha}( x_i-x_k ) \Big)^2 \frac{ ( \partial_{x^\alpha_i} \rho^\varepsilon_N )^2 }{ \rho^\varepsilon_N } \d X \\
= & \dfrac{1}{N} \sum\limits^N_{i=1} \sum^N_{k=1} \int_{\mathbb{T}^{dN}} K^\varepsilon ( x_i-x_k ) \cdot \nabla \rho^\varepsilon_N + \frac{1}{N^2} \sum\limits^N_{i=1} \sum^d_{\alpha=1} \int_{ \mathbb{T}^{dN} } \partial^2_{x^\alpha_i} \Big( \sum\limits_{ k=1 }^N \sigma^\varepsilon_{\alpha\alpha}( x_i-x_k ) \Big)^2 \rho^\varepsilon_N \d X\,.
\end{aligned}
\end{equation}
By Assumption \ref{Ass-A2} and the definition of $\sigma^\varepsilon$, the above equation simplifies to
\begin{equation*}
\begin{aligned}
\frac{\d}{\d t} \int_{ \mathbb{T}^{dN} } & \rho^\varepsilon_N \log \rho^\varepsilon_N \d X + \underline{\sigma}^2 \sum\limits^N_{i=1} \int_{ \mathbb{T}^{dN} } \frac{ | \nabla \rho^\varepsilon_N |^2 }{ \rho^\varepsilon_N } \d X\\
& \leq \dfrac{1}{N} \sum\limits^N_{i=1} \sum^N_{k=1} \int_{ \mathbb{T}^{dN} }   K^\varepsilon( x_i-x_k ) \cdot \nabla \rho^\varepsilon_N \d X
+ 4Nd \| \sigma \|^2_{ W^{2,\infty} } \,.
\end{aligned}
\end{equation*}
Also, observing that $ \| K^\varepsilon \|_{L^\infty} \leq \|K\|_{L^\infty} $ and applying Young’s inequality to the first term on the rhs of the above equation, we get
\begin{equation*}
\begin{aligned}
\dfrac{1}{N} & \sum\limits^N_{i=1} \sum^N_{k=1} \int_{ \mathbb{T}^{dN} } K^\varepsilon ( x_i-x_k ) \cdot \nabla \rho^\varepsilon_N \d X \leq \| K\|_{L^\infty} \sum^N_{i=1} \int_{ \mathbb{T}^{dN} } | \nabla \rho^\varepsilon_N |\d X\\
\leq & \frac{ \underline{\sigma}^2 }{2} \sum^N_{i=1} \int_{ \mathbb{T}^{dN} } \frac{ | \nabla \rho^\varepsilon_N |^2 }{ \rho^\varepsilon } \d X + \frac{ N \|K\|^2_{L^\infty} }{ 2 \underline{\sigma}^2 } \,.
\end{aligned}
\end{equation*}
Hence
\begin{equation*}
\begin{aligned}
\frac{\d}{\d t} \int_{ \mathbb{T}^{dN} } \rho^\varepsilon_N \log \rho^\varepsilon_N \d X & + \underline{\sigma}^2 \sum \limits^N_{i=1} \int_{ \mathbb{T}^{dN} } \frac{ | \nabla \rho^\varepsilon_N |^2 }{ \rho^\varepsilon_N } \d X \\
\leq & \frac{ \underline{\sigma}^2 }{2} \sum^N_{i=1} \int_{ \mathbb{T}^{dN} } \frac{ |\nabla \rho^\varepsilon_N |^2 }{ \rho^\varepsilon } \d X + \frac{N \|K\|^2_{L^\infty} }{ 2\underline{\sigma}^2 } + 4Nd \|\sigma\|^2_{ W^{2,\infty} } \,.
\end{aligned}
\end{equation*}
Note that $	| \nabla_{x_i} \rho^\varepsilon_N |^2 / \rho^\varepsilon_N = 4 | \nabla_{x_i} \sqrt{\rho^\varepsilon_N} |^2$, so
\begin{equation*}\label{Uniform-estimate-entropy}
	\begin{aligned}
		\int_{ \mathbb{T}^{dN} } \rho^\varepsilon_N \log \rho^\varepsilon_N (X) \d X & + 2 \underline{\sigma}^2 \sum\limits^N_{i=1} \int^t_0 \int_{ \mathbb{T}^{dN} }  | \nabla_{x_i} \sqrt{ \rho^\varepsilon_N } |^2 \d X \d s\\
		\leq & \frac{ N \|K\|^2_{L^\infty} }{ 2 \underline{\sigma}^2 } t + 4Nd \| \sigma \|^2_{ W^{2,\infty} } t + \int_{ \mathbb{T}^{dN} } \rho^0_N \log \rho^0_N \d X \,.
	\end{aligned}
\end{equation*}
For the last term, we used the inequality $\int_{ \mathbb{T}^{dN} }  \rho^{0,\varepsilon}_N \log \rho^{0,\varepsilon}_N \leq \int_{ \mathbb{T}^{dN} }  \rho^0_N \log \rho^0_N $, which follows from Jensen's inequality.

\vspace{2mm}

\noindent{\emph{Uniform Bound for $\int_{ \mathbb{T}^{dN} } | \rho^\varepsilon_N \log \rho^\varepsilon_N | \d X$.}} We first separate the integral as 
\begin{equation*}
	\begin{aligned}
		\int_{ \mathbb{T}^{dN} } &| \rho^\varepsilon_N \log \rho^\varepsilon_N | \d X \\ = & \int_{ \rho^\varepsilon_N \leq 1 } \rho^\varepsilon_N |\log \rho^\varepsilon_N | + \int_{ \rho^\varepsilon_N \geq 1} \rho^\varepsilon_N | \log \rho^\varepsilon_N |
		= - \int_{ \rho^\varepsilon_N \leq 1 } \rho^\varepsilon_N \log \rho^\varepsilon_N  + \int_{ \rho^\varepsilon_N \geq 1 } \rho^\varepsilon_N \log \rho^\varepsilon_N \,.
	\end{aligned}
\end{equation*}
Let $F = (2\pi)^{ -\frac{Nd}{2} } e^{ -|X|^2/2 }$ be the density function
of standard normal distribution. Using the inequality $ x \log x - x + 1 \geq 0 $ for all $ x > 0 $, we find that
\begin{equation*}
	\begin{aligned}
		\int_{ \rho^\varepsilon_N \leq 1 } \rho^\varepsilon_N \log \rho^\varepsilon_N \d X = & \int_{ \rho^\varepsilon_N \leq 1} F \Big( \frac{ \rho^\varepsilon_N}{F} \log \frac{ \rho^\varepsilon_N}{F} - \frac{ \rho^\varepsilon_N}{F} + 1 \Big) \d X + \int_{ \rho^\varepsilon_N \leq 1} \rho^\varepsilon_N \log F + \rho^\varepsilon_N - F \d X \\
		\geq & \int_{ \rho^\varepsilon_N \leq 1 } \rho^\varepsilon_N \log F + \rho^\varepsilon_N - F \\
		\geq & -\frac{Nd}{2} \log 2 \pi \int_{\rho^\varepsilon_N \leq 1} \rho^\varepsilon_N - \frac{1}{2} \int_{ \rho^\varepsilon_N \leq 1 } |X|^2 \rho^\varepsilon_N + \int_{ \rho^\varepsilon_N \leq 1} \rho^\varepsilon_N - \int_{ \rho^\varepsilon_N \leq 1} F\\
		\geq & -\frac{Nd}{2} \log 2\pi  -\frac{1}{2} \int_{ \rho^\varepsilon_N \leq 1} |X|^2 \rho^\varepsilon_N - 1\,.
	\end{aligned}
\end{equation*}
 Hence
\begin{equation*}
	\begin{aligned}
		\int_{ \mathbb{T}^{dN} } | \rho^\varepsilon_N \log \rho^\varepsilon_N | \d X 
		\leq & \frac{Nd}{2} \log 2\pi  + \frac{1}{2} \int_{ \rho^\varepsilon_N \leq 1} |X|^2 \rho^\varepsilon_N \d X + 1 + \frac{ N \|K\|^2_{L^\infty} }{ 2 \underline{\sigma}^2 }t \\
		& + 4Nd \| \sigma \|^2_{ W^{2,\infty} } t + \int_{ \mathbb{T}^{dN} } \rho^0_N \log \rho^0_N \d X\,.
	\end{aligned}
\end{equation*}
This implies $\int_{ \mathbb{T}^{dN} } | \rho^\varepsilon_N \log \rho^\varepsilon_N | \d X < C $ with $C$ independent of $\varepsilon$. Thus, we deduce that $\{\rho^\varepsilon_N \}_{\varepsilon}$ is uniformly integrable on $[0,T] \times \mathbb{T}^{dN}$. Note that $\rho^0_N\in L^2$. A similar argument for \eqref{Zeroth-order-estimate} and \eqref{Dissipation-zeroth-order-estimate} gives
\begin{equation}\label{rho-N-L2-estimate}
	\begin{aligned}
		\|\rho^\varepsilon_N\|^2_{L^2} \leq C e^{Ct} \,, \ \sum^N_{i=1} \int^t_0 \| \nabla_{x_i} \rho^\varepsilon_N \|^2_{L^2} \d s \leq Ce^{Ct} \,, \ t\in[0,T]\,.
	\end{aligned}
\end{equation}
By the Banach-Alaoglu theorem and the Dunford-Pettis theorem, combined with a diagonal extraction argument, we can find a subsequence $\{\varepsilon_k\}_{ k \in \mathbb{N} }$ with $\varepsilon_k \to 0$ such that
\begin{equation*}\label{L2-weak-convergence}
	\rho^{\varepsilon_k}_N(t) \rightharpoonup \rho_N(t)\,, \textrm{ in } L^2( 0,T; \mathbb{T}^{dN} )\,,
\end{equation*}
\begin{equation}\label{L2-div-weak-convergence}
	\nabla_{x_i} \rho^{\varepsilon_k}_N(t) \rightharpoonup \nabla_{x_i} \rho_N(t) \,, \textrm{ in } L^2( 0,T; \mathbb{T}^{dN} )\,,
\end{equation}
\begin{equation}\label{L2-div-sqrt-weak-convergence}
	\nabla_{x_i} \sqrt{ \rho^{ \varepsilon_k}_N } \rightharpoonup \nabla_{x_i} \sqrt{\rho_N} \,, \textrm{ in } L^2( 0,T; \mathbb{T}^{dN} )\,,
\end{equation}
\begin{equation}\label{L1-weak-convergence}
	\rho^{\varepsilon_k}_N(t) \rightharpoonup \rho_N(t) \,, \textrm{ in } L^1( 0,T; \mathbb{T}^{dN} )\,,
\end{equation}
\begin{equation*}
	\rho^{\varepsilon_k}_N(t) \rightharpoonup \rho_N(t)\,, \textrm{ in } L^1( \mathbb{T}^{dN} )\,, \  \forall t\in \{ \tau_k\}\,,
\end{equation*}
where $\{\tau_k \}$ is a dense set of $[0,T]$. For any test function $\varphi \in C^2_c( \mathbb{T}^{dN} )$. Note that  $\rho^\varepsilon_N$ satisfies
\begin{equation*}\label{varepsilon-weak-form}
	\begin{aligned}
		\int_{ \mathbb{T}^{dN} } \rho^\varepsilon_N(t) \varphi(t) \d X & - \int_{ \mathbb{T}^{dN} } \rho^\varepsilon_N(0) \varphi(0) \d X = \dfrac{1}{N} \sum\limits^N_{i=1} \sum \limits_{ k=1 }^N \int^t_0 \int_{ \mathbb{T}^{dN} } \rho^\varepsilon_N K^\varepsilon ( x_i-x_k ) \cdot \nabla_{x_i} \varphi\d X \d s \\
		& + \dfrac{1}{N^2} \sum^N_{i=1} \sum^d_{\alpha=1} \int^t_0 \int_{ \mathbb{T}^{dN} } \rho^\varepsilon_N  \Big( \sum\limits_{ k=1 }^N \sigma^\varepsilon_{\alpha\alpha}( x_i-x_k ) \Big)^2 \partial^2_{x^\alpha_i} \varphi \d X \d s\,.
	\end{aligned}
\end{equation*}
Then for $0 \leq t_1 \leq t_2 \leq T $, we have
\begin{equation*}
	\begin{aligned}
			\biggl| \int_{ \mathbb{T}^{dN} } \rho^\varepsilon_N \varphi(t_2) \d X & - \int_{ \mathbb{T}^{dN} } \rho^\varepsilon_N \varphi(t_1) \d X \biggr| \leq \dfrac{1}{N} \sum\limits^N_{i=1} \sum\limits_{ k=1 }^N \biggl| \int^{t_2}_{t_1} \int_{ \mathbb{T}^{dN} } \rho^\varepsilon_N K^\varepsilon ( x_i-x_k ) \cdot \nabla_{x_i} \varphi \d X \d s 	\biggr|	\\
			& + \dfrac{1}{N^2} \sum^N_{i=1} \sum^d_{\alpha=1} \biggl| \int^{t_2}_{t_1} \int_{ \mathbb{T}^{dN} } \rho^\varepsilon_N  \Big( \sum\limits_{ k=1 }^N \sigma^\varepsilon_{\alpha\alpha}( x_i-x_k ) \Big)^2  \partial^2_{x^\alpha_i} \varphi \d X \d s \biggr|\\
			\leq & N \|K\|_{L^\infty} \| \nabla_{x_i} \varphi \|_{L^\infty} | t_2 - t_1 | + Nd \|\sigma\|^2_{L^\infty} \|\varphi\|_{ W^{2,\infty} } |t_2-t_1|\,.
			\end{aligned}
\end{equation*}
This means that $ \{ \int_{ \mathbb{T}^{dN} } \rho^\varepsilon_N \varphi \d X \}_{ \varepsilon > 0 }$ is equicontinuous on $[0,T]$. We also find that $\int_{ \mathbb{T}^{dN} } \rho^\varepsilon_N \varphi \d X$ is uniformly bounded, independent of $\varepsilon$. By the Ascoli–Arzel\`a theorem,  we obtain for the above subsequence $\{\varepsilon_k\}_{ k \in \mathbb{N} }$ with $ \varepsilon_k \to 0$ and for any $\varphi \in C^2_c( \mathbb{T}^{dN} )$ that
\begin{equation*}\label{rho-N-uniform-convergence}
	\int_{ \mathbb{T}^{dN} } \rho^{\varepsilon_k}_N \varphi(t) \d X \to \int_{ \mathbb{T}^{dN} } \rho_N \varphi(t) \d X\,, \ \forall t\in [0,T]\,,
\end{equation*}
where $ \int_{ \mathbb{T}^{dN} } \rho_N \varphi(t) \d X$ is Lipschitz continuous in $t$. We next verify that $\rho_N$ is a weak solution to \eqref{Liouville-Eq} and derive the required entropy dissipation inequality.

\vspace{2mm}

\noindent{\emph{Weak solution.}} Let $\varphi \in C^\infty_c([0,T] \times \mathbb{T}^{dN} )$. 
Note that  $\rho^{\varepsilon_k}_N$ satisfies
\begin{equation*}
	\begin{aligned}
		\int_{ \mathbb{T}^{dN} }   \rho^{ \varepsilon_k}_N(T) &\varphi(T) \d X - \int_{ \mathbb{T}^{dN} }  \rho^{\varepsilon_k}_N(0) \varphi(0) \d X - \int^t_0 \int_{ \mathbb{T}^{dN} }   \rho^{ \varepsilon_k }_N \partial_t \varphi \d X \d s\\
		 = &\dfrac{1}{N} \sum\limits^N_{i=1} \sum \limits_{ k=1 }^N\int^t_0\int_{ \mathbb{T}^{dN} }   \rho^{\varepsilon_k}_N 		K^\varepsilon ( x_i-x_k ) \cdot \nabla_{x_i} \varphi\d X \d s \\
        &- \dfrac{1}{N^2}\sum^N_{i=1} \sum^d_{\alpha=1} \int^t_0\int_{ \mathbb{T}^{dN} }   \rho^{\varepsilon_k}_N  \Big( \sum\limits_{ k=1 }^N \sigma^\varepsilon_{\alpha\alpha}( x_i-x_k ) \Big)^2 \partial^2_{ x^\alpha_i } \varphi \d X \d s\,.
	\end{aligned}
\end{equation*}
The weak convergence in \eqref{L1-weak-convergence} yields
\begin{equation*}
	\begin{aligned}
		\int_{ \mathbb{T}^{dN} }  \rho^{\varepsilon_k}_N & \partial_t \varphi \d X \d s \to \int_{ \mathbb{T}^{dN} }  \rho_N \partial_t \varphi \d X \d s\,, \ \int_{ \mathbb{T}^{dN} }  \varphi(T) \rho^{ \varepsilon_k }_N(T) \d X\to \int_{ \mathbb{T}^{dN} }  \varphi(T) \rho_N (T) \d X\,,\\
		& \int_{ \mathbb{T}^{dN} }  \varphi(0) \rho^{\varepsilon_k}_N(0) \d X \to \int_{ \mathbb{T}^{dN} }  \varphi(0) \rho_N(0) \d X\,, \ \textrm{ as } \varepsilon_k \to 0\,.
	\end{aligned}
\end{equation*}
From the properties of mollifiers and the fact that $K \in L^\infty$ and $ \sigma \in W^{2,\infty}$, we have
$K^{\varepsilon_k} \to K$ in $L^2( \mathbb{T}^{dN} )$ and $\big( \sum_{ k=1 }^N \sigma^{\varepsilon_k}_{\alpha\alpha}( x_i-x_k ) \big)^2 \to \big( \sum_{ k=1 }^N \sigma_{\alpha\alpha}( x_i-x_k ) \big)^2$ in $L^2( \mathbb{T}^{dN} )$ as $ \varepsilon_k \to 0$. Combining these with \eqref{rho-N-L2-estimate} and the weak convergence in \eqref{L1-weak-convergence}, we get
\begin{equation*}
	\begin{aligned}
		&\int^t_0 \int_{ \mathbb{T}^{dN} } \rho^{\varepsilon_k}_N K^{\varepsilon_k} ( x_i-x_k ) \cdot \nabla_{x_i} \varphi \d X \d s - \int^t_0 \int_{ \mathbb{T}^{dN} } \rho_N K ( x_i-x_k ) \cdot \nabla_{x_i} \varphi \d X \d s\\
		= & \int^t_0 \int_{ \mathbb{T}^{dN} } ( K^{\varepsilon_k} - K ) \cdot \nabla_{x_i} \varphi \rho^{\varepsilon_k}_N \d X \d s + \int^t_0 \int_{ \mathbb{T}^{dN} } ( \rho^{\varepsilon_k}_N - \rho_N ) K \cdot \nabla_{x_i} \varphi \d X \d s \\
		\leq &  \|\nabla_{x_i} \varphi\|_{L^\infty} \sup_{t\in[0,T]} \| \rho^{\varepsilon_k}_N \|_{L^2} \|K^{\varepsilon_k} - K \|_{L^2} T + \int^t_0 \int_{ \mathbb{T}^{dN} } ( \rho^{\varepsilon_k}_N - \rho_N) K \cdot \nabla_{x_i} \varphi \d X \d s \to 0
	\end{aligned}
\end{equation*}
and
\begin{equation*}
	\begin{aligned}
		\int^t_0 & \int_{ \mathbb{T}^{dN} } \rho^{\varepsilon_k}_N \Big( \sum\limits_{ k=1 }^N \sigma^{\varepsilon_k}_{\alpha\alpha}( x_i -x_k ) \Big)^2 \partial^2_{x^\alpha_i} \varphi \d X \d s - \int^t_0 \int_{ \mathbb{T}^{dN} } \rho_N \Big( \sum\limits_{ k=1 }^N \sigma_{\alpha\alpha}( x_i-x_k ) \Big)^2  \partial^2_{x^\alpha_i} \varphi \d X \d s\\
		= & \int^t_0 \int_{ \mathbb{T}^{dN} } \rho^{\varepsilon_k}_N \Big[ \Big( \sum\limits_{ k=1 }^N \sigma^{\varepsilon_k}_{\alpha\alpha}( x_i-x_k ) \Big)^2 - \Big( \sum\limits_{ k=1}^N \sigma_{\alpha\alpha}( x_i-x_k ) \Big)^2 \Big] \partial^2_{x^\alpha_i} \varphi \d X \d s \\
		& + \int^t_0 \int_{ \mathbb{T}^{dN} } ( \rho^{\varepsilon_k}_N - \rho_N ) \Big( \sum\limits_{ k=1 }^N \sigma_{\alpha\alpha}( x_i-x_k ) \Big)^2 \partial^2_{x^\alpha_i} \varphi \d X \d s \\
		\leq & \sup_{ t \in [0,T] }\| \rho^{\varepsilon_k}_N \|_{L^2} \| \partial^2_{x^\alpha_i} \varphi \|_{L^\infty} T \biggl\| \Big( \sum\limits_{ k=1 }^N \sigma^{\varepsilon_k}_{\alpha\alpha}( x_i-x_k ) \Big)^2 -  \Big( \sum\limits_{ k=1 }^N \sigma_{\alpha\alpha}( x_i-x_k ) \Big)^2 \biggr\|_{L^2} \\
		& + \int^t_0 \int_{ \mathbb{T}^{dN} } ( \rho^{\varepsilon_k}_N - \rho_N )\Big( \sum\limits_{ k=1 }^N \sigma_{\alpha\alpha}( x_i-x_k ) \Big)^2 \partial^2_{x^\alpha_i} \varphi \d X \d s \to 0\,.
	\end{aligned}
\end{equation*}
Hence $\rho_N$ is a weak solution of \eqref{Liouville-Eq}.

\noindent{\emph{Entropy dissipation inequality.}} From \eqref{Entropy-Time-Evolution} and $\int_{ \mathbb{T}^{dN} }  \rho^{0,\varepsilon}_N \log \rho^{0,\varepsilon}_N \leq \int_{ \mathbb{T}^{dN} }  \rho^0_N \log \rho^0_N $, we have
\begin{equation}\label{Entropy-equality-varepsilon}
	\begin{aligned}
		\int_{ \mathbb{T}^{dN} } \rho^{\varepsilon_k}_N & \log \rho^{\varepsilon_k}_N \d X + \frac{1}{N^2} \sum\limits^N_{i=1} \sum^d_{\alpha=1} \int^t_0 \int_{ \mathbb{T}^{dN} } \Big( \sum\limits_{ k=1 }^N \sigma^{\varepsilon_k}_{\alpha\alpha}( x_i-x_k ) \Big)^2 \frac{ ( \partial_{x^\alpha_i} \rho^{\varepsilon_k}_N )^2 }{ \rho^{\varepsilon_k}_N } \d X \d s\\
		\leq &	\int_{\mathbb{T}^{dN} }\rho^0_N \log \rho^0_N \d X + \dfrac{1}{N} \sum\limits^N_{i=1} \sum^N_{k=1} \int^t_0 \int_{ \mathbb{T}^{dN} } K^{\varepsilon_k} ( x_i-x_k ) \cdot  \nabla \rho^{\varepsilon_k}_N \d X \d s \\
		& + \frac{1}{N^2} \sum\limits^N_{i=1} \sum^d_{\alpha=1} \int^t_0 \int_{ \mathbb{T}^{dN} } \partial^2_{x^\alpha_i}  \Big( \sum\limits_{ k=1 }^N \sigma^{\varepsilon_k}_{\alpha\alpha}( x_i-x_k ) \Big)^2 \rho^{\varepsilon_k}_N  \d X \d s\,.
	\end{aligned}
\end{equation}
For the second term on the rhs of \eqref{Entropy-equality-varepsilon}, by
$K^{\varepsilon_k} \to K$ in $L^2(\mathbb{T}^{dN})$ and the weak convergence in \eqref{L2-div-weak-convergence}, we have
\begin{equation}\label{Est-rhs-second-term}
	\begin{aligned}
		\int^t_0 & \int_{ \mathbb{T}^{dN} } K^{\varepsilon_k} ( x_i-x_k ) \cdot \nabla \rho^{\varepsilon_k}_N \d X \d s - \int^t_0 \int_{ \mathbb{T}^{dN} } K ( x_i-x_k ) \cdot \nabla \rho_N \d X \d s \\
		= & \int^t_0 \int_{ \mathbb{T}^{dN} }  \big( K^{\varepsilon_k}( x_i-x_k ) - K ( x_i-x_k ) ) \cdot \nabla \rho^{\varepsilon_k}_N \d X \d s\\
		& + \int^t_0 \int_{ \mathbb{T}^{dN} } K ( x_i-x_k ) \cdot ( \nabla_{x_i} \rho^{\varepsilon_k}_N - \nabla \rho_N ) \d X \d s \\
		\leq & \| K^{\varepsilon_k} - K \|_{L^2} \int^t_0 \| \nabla \rho^{\varepsilon_k}_N \|_{L^2} \d s - \int^t_0 \int_{ \mathbb{T}^{dN} }  K( x_i-x_k ) \cdot ( \nabla \rho^{\varepsilon_k} - \nabla \rho_N ) \d X \d s \to 0\,.
	\end{aligned}
\end{equation}
By the properties of mollifiers and $ \sigma \in W^{2,\infty}$, we obtain that $ \partial^2_{x^\alpha_i} \big( \sum_{ k=1 }^N \sigma^{\varepsilon_k}_{\alpha\alpha}( x_i-x_k ) \big)^2$ converges strongly to $\partial^2_{x^\alpha_i} \big( \sum_{ k=1 }^N \sigma_{\alpha\alpha}( x_i-x_k ) \big)^2 $ in $L^2( \mathbb{T}^{dN} )$ as $\varepsilon_k \to 0$. Combining this with the weak convergence in \eqref{L1-weak-convergence}, we obtain for the third item on the rhs of \eqref{Entropy-equality-varepsilon} that
\begin{equation}\label{Est-rhs-third-term}
	\begin{aligned}
		\int^t_0 & \int_{ \mathbb{T}^{dN} }  \partial^2_{x^\alpha_i}  \Big( \sum\limits_{ k=1 }^N \sigma^{\varepsilon_k}_{\alpha\alpha}( x_i-x_k ) \Big)^2 \rho^\varepsilon_N \d X \d s - \int^t_0 \int_{ \mathbb{T}^{dN} }  \partial^2_{x^\alpha_i} \Big( \sum\limits_{ k=1 }^N \sigma_{\alpha\alpha}( x_i-x_k ) \Big)^2  \rho_N \d X \d s\\
		= & \int^t_0 \int_{ \mathbb{T}^{dN} }  \Big(  \partial^2_{x^\alpha_i} \Big( \sum\limits_{k=1}^N \sigma^{\varepsilon_k}_{\alpha\alpha}( x_i-x_k ) \Big)^2 - \partial^2_{x^\alpha_i} \Big( \sum\limits_{ k=1 }^N \sigma_{\alpha\alpha}( x_i-x_k ) \Big)^2  \Big) \rho^{\varepsilon_k}_N \d X \d s\\
		& + \int^t_0 \int_{ \mathbb{T}^{dN} }  \partial^2_{x^\alpha_i}  \Big( \sum\limits_{ k=1 }^N \sigma_{\alpha\alpha}( x_i-x_k ) \Big)^2 \big( \rho^{\varepsilon_k}_N - \rho_N) \d X \d s \\
		\leq & \sup_{t\in[0,T]}\| \rho^{\varepsilon_k}_N \|_{L^2} T \biggl\| \partial^2_{x^\alpha_i} \Big( \sum\limits_{ k=1 }^N \sigma^{\varepsilon_k}_{\alpha\alpha}( x_i-x_k ) \Big)^2 - \partial^2_{x^\alpha_i} \Big( \sum\limits_{ k=1 }^N \sigma_{\alpha\alpha}( x_i-x_k ) \Big)^2 \biggr\|_{L^2}\\
		& + \int^t_0 \int_{ \mathbb{T}^{dN} }  \partial^2_{x^\alpha_i}  \Big( \sum\limits_{ k=1 }^N \sigma_{\alpha\alpha}( x_i-x_k) \Big)^2 \big( \rho^{\varepsilon_k}_N -\rho_N) \d X \d s \to 0\,.
	\end{aligned}
\end{equation}
For the first term on the lhs of \eqref{Entropy-equality-varepsilon}, we have
\begin{equation}\label{Est-lhs-first-term}
	\int_{ \mathbb{T}^{dN} }  \rho_N\log \rho_N \d X \leq \liminf_{\varepsilon_k \to 0} \int_{ \mathbb{T}^{dN} }  \rho^{\varepsilon_k}_N \log \rho^{\varepsilon_k}_N \d X
\end{equation}
by the lower semicontinuity of entropy. For the second term on the lhs, we rewrite it as
\begin{equation*}
	\begin{aligned}
		&\frac{1}{N^2} \sum\limits^N_{i=1} \sum^d_{\alpha=1} \int^t_0 \int_{ \mathbb{T}^{dN} }  \Big( \sum\limits_{k=1}^N \sigma^\varepsilon_{\alpha\alpha}( x_i-x_k ) \Big)^2 \frac{ (\partial_{x^\alpha_i} \rho^{\varepsilon_k}_N )^2 }{ \rho^{\varepsilon_k}_N } \d X \d s \\
		= & \frac{1}{N^2} \sum\limits^N_{i=1} \int^t_0 \int_{ \mathbb{T}^{dN} }  \Big( \sum\limits_{k=1}^N \sigma^\varepsilon( x_i-x_k ) \Big)^2 :  \frac{\nabla_{x_i} \rho^{\varepsilon_k}_N  \otimes \nabla_{x_i} \rho^{\varepsilon_k}_N }{ \rho^{\varepsilon_k}_N } \d X \d s\\
		= & \frac{4}{N^2} \sum\limits^N_{i=1} \int^t_0 \int_{ \mathbb{T}^{dN} }  \tilde{\sigma}^{\varepsilon_k}_i :  \nabla_{x_i} \sqrt{ \rho^{\varepsilon_k}_N } \otimes \nabla_{x_i} \sqrt{ \rho^{\varepsilon_k}_N } \d X \d s:= \frac{4}{N^2} \sum\limits^N_{i=1} \int^t_0 D_i ( \rho^{\varepsilon_k}_N ) \d s\,,
	\end{aligned}
\end{equation*}
where $ \tilde{\sigma}^{\varepsilon_k}_i = \big( \sum_{k=1}^N \sigma^{\varepsilon_k}( x_i-x_k ) \big)^2$.
Note that $\sigma$ is positive definite. For each $i$, $ D_i (\rho^{\varepsilon_k}_N) $ has the representation
\begin{equation*}
	D_i ( \rho^{\varepsilon_k}_N ) = \sup_{\phi \in C^\infty_c ( \mathbb{T}^{dN};\mathbb{R}^d ) } \int_{ \mathbb{T}^{dN} }  2\Big( \tilde{\sigma}^{\varepsilon_k}_i \nabla_{x_i} \sqrt{\rho^{\varepsilon_k}_N} \Big) \cdot \phi- \phi \cdot \tilde{\sigma}^{\varepsilon_k}_i \phi  \d X \,.
\end{equation*}
This holds by approximating $\nabla_{x_i} \sqrt{ \rho^{\varepsilon_k}_N }$ in $C^\infty_c$ by mollification and truncation. Similarly, 
\begin{equation*}
	\begin{aligned}
		&\frac{1}{N^2} \sum\limits^N_{i=1} \sum^d_{\alpha=1} \int^t_0 \int_{ \mathbb{T}^{dN} }  \Big( \sum\limits_{k=1}^N \sigma_{\alpha\alpha}( x_i-x_k ) \Big)^2 \frac{ ( \partial_{x^\alpha_i} \rho_N )^2 }{ \rho_N } \d X \d s \\
		= & \frac{4}{N^2} \sum\limits^N_{i=1} \int^t_0 \int_{ \mathbb{T}^{dN} }  \tilde{\sigma}_i :  \nabla_{x_i} \sqrt{ \rho_N } \otimes \nabla_{x_i} \sqrt{ \rho_N } \d X \d s:= \frac{4}{N^2} \sum\limits^N_{i=1} \int^t_0 D_i (\rho_N) \d s\,,
	\end{aligned}
\end{equation*}
where $ \tilde{\sigma}_i = \big( \sum_{k=1}^N \sigma(x_i-x_k) \big)^2$.  For each $i$, $D_i (\rho_N) $ has the representation
\begin{equation*}
	D_i(\rho_N) = \sup_{ \phi \in C^\infty_c ( \mathbb{T}^{dN};\mathbb{R}^d ) } \int_{ \mathbb{T}^{dN} }  2 \big( \tilde{\sigma}_i \nabla_{x_i} \sqrt{\rho_N} \big) \cdot \phi - \phi \cdot \tilde{\sigma}_i \phi \d X\,.
\end{equation*}
By the definition of the supremum, for any $\delta>0$, there exists a $\phi_\delta$ such that
\begin{equation*}
	\int_{ \mathbb{T}^{dN} }  2 \big( \tilde{\sigma}_i \nabla_{x_i} \sqrt{\rho_N} \big) \cdot \phi_{\delta} -\phi_{\delta} \cdot \tilde{\sigma}_i \phi_\delta  \d X \geq D_i(\rho_N) - \delta
\end{equation*}
and 
\begin{equation*}
	D_i (\rho^{\varepsilon_k}_N) \geq \int_{ \mathbb{T}^{dN} }  2 \Big( \tilde{\sigma}^{\varepsilon_k}_i \nabla_{x_i} \sqrt{\rho^{\varepsilon_k}_N} \Big) \cdot \phi_\delta -\phi_\delta \cdot \tilde{\sigma}^{\varepsilon_k}_i \phi_\delta \d X\,.
\end{equation*}
Then we have
\begin{equation*}
\liminf_{\varepsilon_k \to 0}	D_i (\rho^{\varepsilon_k}_N) \geq \liminf_{\varepsilon_k \to 0} \int_{ \mathbb{T}^{dN} }  2 \Big( \tilde{\sigma}^{\varepsilon_k}_i \nabla_{x_i} \sqrt{ \rho^{\varepsilon_k}_N } \Big) \cdot \phi_\delta - \phi_\delta \cdot \tilde{\sigma}^{\varepsilon_k}_i \phi_\delta  \d X\,.
\end{equation*}
By the properties of mollifiers and the fact $ \sigma\in W^{2,\infty}$, we have $\tilde{\sigma}^{\varepsilon_k}_i \to \tilde{\sigma}_i$ in $L^2(\mathbb{T}^{dN})$ as $\varepsilon_k \to 0$.  Combining this with the weak convergence in \eqref{L2-div-sqrt-weak-convergence}, we obtain
\begin{equation*}
	\begin{aligned}
		\int_{ \mathbb{T}^{dN} } & 2 \Big( \tilde{\sigma}^{\varepsilon_k}_i \nabla_{x_i} \sqrt{\rho^{\varepsilon_k}_N} \Big) \cdot \phi_\delta - \phi_\delta \cdot \tilde{\sigma}^{\varepsilon_k}_i \phi_\delta  \d X - \int_{ \mathbb{T}^{dN} }  2 \big( \tilde{\sigma}_i \nabla_{x_i} \sqrt{\rho_N} \big) \cdot \phi_{\delta} - \phi_{\delta}\cdot \tilde{\sigma}_i \phi_\delta  \d X\\
		= & 2 \int_{ \mathbb{T}^{dN} }  \Big( ( \tilde{\sigma}^{\varepsilon_k}_i - \tilde{\sigma}_i ) \nabla_{x_i} \sqrt{\rho^{\varepsilon_k}_N} \Big) \cdot \phi_\delta \d X+ 2\int_{ \mathbb{T}^{dN} }  \Big( \tilde{\sigma}_i \Big( \nabla_{x_i} \sqrt{\rho^{\varepsilon_k}_N} - \nabla_{x_i} \sqrt{ \rho^{\varepsilon_k}_N } \Big) \Big) \cdot \phi_\delta \d X  \\
        &- \int_{ \mathbb{T}^{dN} } \phi_\delta\cdot (\tilde{\sigma}^{\varepsilon_k}_i - \tilde{\sigma}_i) \phi_\delta  \d X \\
		\leq & 2 \| \tilde{\sigma}^{\varepsilon_k}_i - \tilde{\sigma}_i \|_{L^2} \Big\| \nabla_{x_i} \sqrt{\rho^{\varepsilon_k}_N} \Big\|_{L^2} \|\phi_\delta\|_{L^\infty} + 2 \int_{ \mathbb{T}^{dN} }  \Big( \tilde{\sigma}_i \Big( \nabla_{x_i}\sqrt{\rho^{\varepsilon_k}_N} - \nabla_{x_i} \sqrt{\rho^{\varepsilon_k}_N} \Big) \Big) \cdot \phi_\delta \d X\\
		& - \int_{ \mathbb{T}^{dN} } \phi_\delta\cdot (\tilde{\sigma}^{\varepsilon_k}_i - \tilde{\sigma}_i) \phi_\delta  \d X \ \to 0, \ \textrm{ as } \varepsilon_k \to 0\,.
	\end{aligned}
\end{equation*}
This implies that
\begin{equation*}
\liminf_{\varepsilon_k \to 0} \int_{ \mathbb{T}^{dN} }  2 \Big( \tilde{\sigma}^{\varepsilon_k}_i \nabla_{x_i} \sqrt{\rho^{\varepsilon_k}_N} \Big) \cdot \phi_\delta - \phi_\delta\cdot \tilde{\sigma}^{\varepsilon_k}_i \phi_\delta  \d X = \int_{ \mathbb{T}^{dN} }  2 \big( \tilde{\sigma}_i \nabla_{x_i} \sqrt{\rho_N} \big) \cdot \phi_{\delta} - \phi_{\delta}\cdot \tilde{\sigma}_i \phi_\delta  \d X\,.
\end{equation*}
Hence
\begin{equation*}
	\begin{aligned}		\liminf_{\varepsilon_k \to 0} &  D_i ( \rho^{\varepsilon_k}_N ) \geq 		\liminf_{\varepsilon_k \to 0} \int_{ \mathbb{T}^{dN} }  2 \Big( \tilde{\sigma}^{\varepsilon_k}_i \nabla_{x_i} \sqrt{\rho^{\varepsilon_k}_N} \Big) \cdot \phi_\delta - \phi_\delta\cdot \tilde{\sigma}^{\varepsilon_k}_i \phi_\delta  \d X\\
		= & \int_{ \mathbb{T}^{dN} }  2 \big( \tilde{\sigma}_i \nabla_{x_i} \sqrt{\rho_N} \big)\cdot \phi_{\delta} - \phi_{\delta}\cdot \tilde{\sigma}_i \phi_\delta  \d X \geq D_i (\rho_N) - \delta \,.
	\end{aligned}
\end{equation*}
Leting $\delta \to 0$ yields 
\begin{equation*}
	\begin{aligned}
		\liminf_{\varepsilon_k \to 0} D_i ( \rho^{\varepsilon_k}_N ) \geq  D_i (\rho_N)\,.
	\end{aligned}
\end{equation*}
Hence 
\begin{equation}\label{Est-lhs-second-term}
	\begin{aligned}		\liminf_{\varepsilon_k \to 0} & \frac{1}{N^2} \sum\limits^N_{i=1} \sum^d_{\alpha=1} \int^t_0 \int_{ \mathbb{T}^{dN} }   \Big( \sum\limits_{k=1}^N \sigma^\varepsilon_{\alpha\alpha}( x_i-x_k ) \Big)^2 \frac{ ( \partial_{x^\alpha_i} \rho^{\varepsilon_k}_N )^2 }{ \rho^{\varepsilon_k}_N } \d X \d s\\
		\geq & \frac{1}{N^2} \sum\limits^N_{i=1} \sum^d_{\alpha=1} \int^t_0 \int_{ \mathbb{T}^{dN} } \Big( \sum\limits_{k=1}^N \sigma_{\alpha\alpha}(x_i-x_k) \Big)^2 \frac{(\partial_{x^\alpha_i} \rho_N )^2 }{\rho_N} \d X \d s\,.
	\end{aligned}
\end{equation}
Taking the limit inferior as $\varepsilon_k\to 0$ on both sides of \eqref{Entropy-equality-varepsilon} and using the results from \eqref{Est-rhs-second-term}-\eqref{Est-lhs-second-term}, we obtain the desired inequality \eqref{Entropy- dissipation-inequality}.

	\vspace*{3mm}
	
	\noindent{\bf Conflict of interest statement:} We declare that there are no conflicts of interest to this work.

	
	\section*{Acknowledgments}
	The authors sincerely thank Prof. Zhenfu Wang for his valuable advice and guidance, Prof. Daniel Lacker for generously providing helpful literature, and Mr. Xuanrui Feng for the insightful discussions.
	
\bigskip
	

\begin{thebibliography}{99}

\bibitem{AH2010} S. M. Ahn and S.-Y. Ha, Stochastic flocking dynamics of the Cucker-Smale model with multiplicative white noises. {\em J. Math. Phys.}, {\bf 51} (2010), no. 10, 103301, 17 pp.

\bibitem{BFFT2012} J. Baladron, D. Fasoli, O. Faugeras, and J. Touboul, Mean-field description and propagation of chaos in networks of Hodgkin-Huxley and FitzHugh-Nagumo neurons. {\em J. Math. Neurosci.}, {\bf 2} (2012), Art. 10, 50 pp.

\bibitem{BCL2026} L. Bol, L. Chen, and Y. Li, Two-dimensional signal-dependent parabolic-elliptic Keller-Segel system and its mean-field derivation. {\em J. Differential Equations}, {\bf 450} (2026), Paper No. 113712, 49 pp.

\bibitem{BCC2011} F. Bolley, J. A. Ca\~nizo, and J. A. Carrillo, Stochastic mean-field limit: non-Lipschitz forces and swarming. {\em Math. Models Methods Appl. Sci.}, {\bf 21} (2011), no. 11, 2179-2210.

\bibitem{BJW2023} D. Bresch, P. E. Jabin, and Z. Wang, 
Mean field limit and quantitative estimates with singular attractive kernels. {\em Duke Math. J.}, {\bf 172} (2023), no. 13, 2591–2641.

\bibitem{Bru2010} R. A. Brualdi, Introductory Combinatorics. 5th ed., Pearson Prentice Hall, Upper Saddle River, NJ, 2010.
		
\bibitem{CCS2019} J. A. Carrillo, Y.-P. Choi, and S. Salem, Propagation of chaos for the Vlasov-Poisson-Fokker-Planck equation with a polynomial cut-off. {\em Commun. Contemp. Math.}, {\bf 21} (2019), no. 4, 1850039, 28 pp.

\bibitem{CG2025} J. A. Carrillo and S. Guo, From Fisher information decay for the Kac model to the Landau-Coulomb hierarchy. {\em arXiv: 2502.18606} (2025).

\bibitem{CGJ2024} J. A. Carrillo, S. Guo, and P. E. Jabin, Mean-field derivation of Landau-like equations. {\em Appl. Math. Lett.}, {\bf 158} (2024), Paper No. 109195, 5 pp.

\bibitem{CDLL2019} P. Cardaliaguet, F. Delarue, J.-M. Lasry, and P.-L. Lions, The master equation and the convergence problem in mean field games.
{\em Ann. of Math. Stud.}, {\bf 201}, 
Princeton University Press, Princeton, NJ, 2019.

\bibitem{CST2022} J.-F. Chassagneux, L. Szpruch, and A. Tse, Weak quantitative propagation of chaos via differential calculus on the space of measures. {\em Ann. Appl. Probab.}, {\bf 32} (2022), no. 3, 1929–1969.

\bibitem{Chi1994} T.-S. Chiang, McKean-Vlasov equations with discontinuous coefficients. {\em Soochow J. Math.}, {\bf 20} (1994), no. 4, 507–526.
		
\bibitem{CS2019} Y.-P. Choi and S. Salem, Cucker-Smale flocking particles with multiplicative noises: stochastic mean-field limit and phase transition. {\em Kinet. Relat. Models}, {\bf 12} (2019), no. 3, 573–592.

\bibitem{CF2016} M. Coghi and F. Flandoli, Propagation of chaos for interacting particles subject to environmental noise. {\em Ann. Appl. Probab.}, {\bf 26} (2016), no. 3, 1407–1442.

\bibitem{CO2025} J. Correa and C. Olivera, From particle systems to the stochastic compressible Navier-Stokes equations of a barotropic fluid. {\em J. Nonlinear Sci.}, {\bf 35} (2025), no. 3, Paper No. 50, 47 pp.

\bibitem{CT2024} C. Crucianelli and L. Tangpi, Interacting particle systems on sparse W-random graphs. {\em arXiv: 2410.11240} (2024).

\bibitem{Deg1986} P. Degond, Global existence of smooth solutions for the Vlasov-Fokker-Planck equation in 1 and 2 space dimensions. {\em Ann. Sci. École Norm. Sup. (4)}, {\bf 19} (1986), no.4, 519–542.

\bibitem{DL2024} K. Du and L. Li, A collision-oriented interacting particle system for Landau-type equations and the molecular chaos. {\em arXiv: 2408.16252} (2024).

\bibitem{DEGZ2020} A. Durmus, A. Eberle, A. Guillin, and R. Zimmer, An elementary approach to uniform in time propagation of chaos. {\em 
Proc. Amer. Math. Soc.}, {\bf 148} (2020), no. 12, 5387–5398.

\bibitem{Eva2010} L. C. Evans, Partial differential equations. 2nd ed., {\em Grad. Stud. Math.}, vol. 19, American Mathematical Society, Providence, RI, 2010. 

\bibitem{FW2025} X. Feng and Z. Wang, Kac's program for the Landau equation. {\em arXiv: 2506.14309} (2025). 

\bibitem{Gar1988} J. G\"artner, On the McKean-Vlasov limit for interacting diffusions. {\em Math. Nachr.}, {\bf 137} (1988), 197–248.

\bibitem{GGP2025} J. Grass, A. Guillin, and C. Poquet, Sharp propagation of chaos for McKean-Vlasov equation with non constant diffusion coefficient. {\em Electron. Commun. Probab.}, {\bf30} (2025), Paper No. 55, 12 pp.

\bibitem{GPG2025} J. Grass, C. Poquet, and A. Guillin, Propagation of chaos in Fisher information. {\em arXiv: 2511.20078} (2025).

\bibitem{GLM2025} A. Guillin, P. Le Bris, and P. Monmarch\'e, Uniform in time propagation of chaos for the 2D vortex model and other singular stochastic systems.
{\em J. Eur. Math. Soc.}, {\bf 27} (2025), no. 6, 2359–2386.
		
\bibitem{HJNXZ2017} S.-Y. Ha, J. Jeong, S. E. Noh, Q. Xiao, and X. Zhang, Emergent dynamics of Cucker-Smale flocking particles in a random environment. {\em J. Differential Equations}, {\bf 262} (2017), no. 3, 2554-2591.

\bibitem{HKMZ2020} S.-Y. Ha, D. Ko, C. Min, and X. Zhang, Emergent collective behaviors of stochastic Kuramoto oscillators. {\em Discrete Contin. Dyn. Syst. Ser. B}, {\bf 25} (2020), no. 3, 1059-1081.

\bibitem{Hua2023} X. Huang, Long time entropy-cost type propagation of chaos. {\em arXiv: 2308.15181 } (2023). 

\bibitem{Hua2025} X. Huang, Coupling by change of measure for conditional McKean-Vlasov SDEs and applications. {\em Stochastic Process. Appl.}, {\bf 179} (2025), Paper No. 104508, 17 pp.

\bibitem{Hua2026} X. Huang, Quantitative propagation of chaos in  $L^{\eta} \, (\eta \in [0,1])$-Wasserstein distance for mean field interacting particle system. {\em
Potential Anal.}, {\bf 64} (2026), no. 3, Paper No. 44.

\bibitem{JW2016} P. E. Jabin and Z. Wang, 
Mean field limit and propagation of chaos for Vlasov systems with bounded forces. {\em J. Funct. Anal.}, {\bf 271} (2016), no. 12, 3588–3627.

\bibitem{JW2018} P. E. Jabin and Z. Wang, Quantitative estimates of propagation of chaos for stochastic systems with $W^{-1,\infty}$ kernels. {\em Invent. Math.}, {\bf 214} (2018), no. 1, 523–591.

	

\bibitem{JM1998} B. Jourdain and S. M\'el\'eard, Propagation of chaos and fluctuations for a moderate model with smooth initial data. {\em Ann. Inst. H. Poincar\'e Probab. Statist.}, {\bf 34} (1998), no. 6, 727–766.

\bibitem{Lac2018} D. Lacker, On a strong form of propagation of chaos for McKean-Vlasov equations. {\em Electron. Commun. Probab.}, {\bf 23} (2018), Paper No. 45, 11 pp.

\bibitem{LPLSFD2007} N. E. Leonard, D. A. Paley, F. Lekien, R. Sepulchre, D. M. Fratantoni, and R. E. Davis, Collective motion, sensor networks, and ocean sampling, {\em Proc. IEEE}, \textbf{95} (2007), no. 1, 48-74.

\bibitem{MB2002} A. J. Majda and A. L. Bertozzi, Vorticity and incompressible flow. Cambridge Texts in Applied Mathematics, vol. 27, Cambridge University Press, Cambridge, 2002.

\bibitem{Mal2001} F. Malrieu, Logarithmic Sobolev inequalities for some nonlinear PDE's.
 {\em Stochastic Process. Appl.}, {\bf 95} (2001), no. 1, 109–132.

\bibitem{Mal2003} F. Malrieu, Convergence to equilibrium for granular media equations and their Euler schemes. {\em Ann. Appl. Probab.}, {\bf 13} (2003), no. 2, 540–560.
		
\bibitem{Mel1996} S. M\'el\'eard, Asymptotic behaviour of some interacting particle systems; McKean-Vlasov and Boltzmann models. In: {\em Probabilistic models for nonlinear partial differential equations (Montecatini Terme, 1995)}. {\em Lecture Notes in Math.}, Vol. 1627, Springer-Verlag, Berlin, 1996, pp. 42–95.


\bibitem{NRS2022} Q.-H. Nguyen, M. Rosenzweig, and S. Serfaty, Mean-field limits of Riesz-type singular flows. {\em Ars Inven. Anal.}, (2022), Paper No. 4, 45 pp.

\bibitem{Nik2025} P. Nikolaev, Quantitative relative entropy estimates for interacting particle systems with common noise. {\em SIAM J. Math. Anal.}, {\bf 57} (2025), no. 3, 3071–3109.

\bibitem{NW2026} N. Ning and J. Wu, Well-posedness and propagation of chaos for McKean-Vlasov stochastic variational inequalities. {\em J. Theoret. Probab.}, {\bf 39} (2026), no. 1, Paper No. 5, 48 pp.

\bibitem{Pav2014} G. A. Pavliotis, Stochastic processes and applications.
Diffusion processes, the Fokker-Planck and Langevin equations. {\em Texts Appl. Math.}, vol. 60, Springer, New York, 2014.

\bibitem{Ros2020} M. Rosenzweig, The mean-field limit of stochastic point vortex systems with multiplicative noise. {\em arXiv: 2011.12180} (2020).

\bibitem{RS2023} M. Rosenzweig and S. Serfaty,
Global-in-time mean-field convergence for singular Riesz-type diffusive flows. {\em Ann. Appl. Probab.}, {\bf 33} (2023), no. 2, 754–798.

\bibitem{RS2024} M. Rosenzweig and S. Serfaty, Relative entropy and modulated free energy without confinement via self-similar transformation. {\em arXiv: 2402.13977} (2024).

\bibitem{Spo1991} H. Spohn, Large scale dynamics of interacting particles. {\em Texts and Monographs in Physics.} Springer, Berlin, 1991.

\bibitem{Szn1991} A.-S. Sznitman, Topics in propagation of chaos. In: {\em \'Ecole d'\'Et\'e de Probabilit\'es de Saint-Flour XIX--1989}, Springer-Verlag, Berlin, {\bf 1464} (1991), 165-251.
		
\bibitem{Vil2009} C. Villani, Optimal Transport, Old and New. {\em Grundlehren der mathematischen Wissenschaften}, vol. 338, Springer-Verlag, Berlin, 2009.

\bibitem{WLH2026} J. Wang, K. Li, and H. Huang, Rigorous derivation of the mean-field limit for the signal-dependent Keller-Segel system. {\em arXiv: 2602.01138} (2026).

\end{thebibliography}

\end{document}